\documentclass[11pt]{amsart}
\usepackage{amsmath,amsthm,amsfonts,amscd,amssymb,eucal,latexsym,mathrsfs,stmaryrd,enumitem,mathtools}
\usepackage[all]{xy}

\newtheorem{theorem}{Theorem}[section]

\newtheorem{lemma}[theorem]{Lemma}
\newtheorem{proposition}[theorem]{Proposition}

\theoremstyle{definition}
\newtheorem{definition}[theorem]{Definition}
\newtheorem{remark}[theorem]{Remark}

\newcounter{theoremintro}
\newtheorem{theoremi}[theoremintro]{Theorem}

\newcommand{\id}{{\rm id}}

\newcommand{\cD}{{\mathcal D}}
\newcommand{\cE}{{\mathcal E}}

\newcommand{\cF}{{\mathcal F}}
\newcommand{\cV}{{\mathcal V}}

\newcommand{\fC}{{\mathfrak Y}}
\newcommand{\sA}{{\mathscr A}}

\newcommand{\sC}{{\mathscr C}}
\newcommand{\sD}{{\mathscr D}}

\newcommand{\sP}{{\mathscr P}}
\newcommand{\sQ}{{\mathscr Q}}
\newcommand{\sR}{{\mathscr R}}
\newcommand{\sT}{{\mathscr T}}
\newcommand{\sU}{{\mathscr U}}
\newcommand{\sW}{{\mathscr W}}

\newcommand{\Fb}{{\mathbb F}}
\newcommand{\Zb}{{\mathbb Z}}
\newcommand{\Pb}{{\mathbb P}}

\newcommand{\Nb}{{\mathbb N}}

\newcommand{\alg}{{\rm alg}}

\newcommand{\m}{{\rm m}}

\newcommand{\eps}{\varepsilon}

\newcommand{\Hamm}{{\rm Hamm}}
\newcommand{\prdct}{{\text{prod}}}
\newcommand{\primespace}{\hspace*{0.6mm}}

\DeclareMathOperator{\Hom}{Hom}

\DeclareMathOperator{\Sym}{Sym}

\allowdisplaybreaks

\begin{document}

\title{Entropy, Shannon orbit equivalence, and sparse connectivity}
\author{David Kerr}
\address{David Kerr,
Department of Mathematics,
Texas A\&M University,
College Station, TX 77843-3368, USA}
\email{kerr@math.tamu.edu}

\author{Hanfeng Li}
\address{Hanfeng Li,
Center of Mathematics, Chongqing University, Chongqing 40133, China
and
Department of Mathematics, SUNY at Buffalo, Buffalo, NY 14260-2900, USA
}
\email{hfli@math.buffalo.edu}

\date{December 5, 2019}

\begin{abstract}
We say that two free p.m.p.\ actions of countable groups are Shannon orbit equivalent
if there is an orbit equivalence between them whose associated cocycle partitions have finite Shannon entropy.
We show that if the acting groups are sofic and each has a w-normal amenable subgroup which is neither locally
finite nor virtually cyclic then Shannon orbit equivalence implies that the actions have the same
maximum sofic entropy. This extends a result of Austin beyond the finitely generated amenable setting
and has the consequence that two Bernoulli actions of a group with the properties in question
are Shannon orbit equivalent if and only if they are measure conjugate. 
Our arguments apply more generally to actions satisfying a sparse connectivity condition which we call property SC,
and yield an entropy inequality under the assumption that one of the actions has this property.
\end{abstract}

\maketitle

\tableofcontents

\section{Introduction}

One of the remarkable features of the notion of amenability for groups is that its
fundamental characterizations in terms of nonparadoxicality on the one hand and approximate invariance on the other
lead to two very different and in many ways incompatible conceptions of what its opposite should be,
namely freeness and property (T) (i.e., universal spectral gap).\footnote{This bifurcation is also reflected in the
two logically independent ways in which
the ideas of amenability and nonamenability have been adapted to operator algebras,
on the one hand through the notions of finiteness and pure infiniteness and on the other
through injectivity and finite-dimensional approximation. All of this traces back
to the basic Dedekindian alternative for defining what it means for a set to be finite,
either as the property that every injection from the set to itself is surjective,
or by the existence of a bijection between the set and $\{ 1,\dots , n\}$ for some positive integer $n$.
}
While noncyclic free groups
distill the idea of paradoxicality to its starkest form and represent the simplest and most venerable
obstruction to amenability, the fact that approximate invariance affords so much technical leverage
has meant that amenability is frequently contrasted instead with
property (T),
even sometimes in the form of a direct counterpositioning that exploits the tension between them,
as in Margulis's proof of his celebrated normal subgroup theorem.\footnote{Again this is consistent with operator algebra theory, where amenability has become synonymous
with certain kinds of finite-dimensional approximation.}

In the theory of orbit equivalence,
the dichotomy between approximate invariance and spectral gap plays out with particularly dramatic consequences.
Here a sharp wedge is already driven between amenability and nonamenability:
while the Ornstein--Weiss tiling theorem
establishes that any two free ergodic p.m.p.\ actions of countably infinite amenable groups are orbit equivalent \cite{OrnWei87},
Epstein showed, completing a line of development in \cite{Zim80,BezGol81,Hjo05,GabPop05,Ioa11},
that every countable nonamenable group admits uncountably many orbit inequivalent
free ergodic p.m.p.\ actions \cite{Eps08}.
At the more extreme end of spectral anti-amenability, a
theorem of Popa shows that Bernoulli actions (and, more generally, weakly mixing malleable actions) 
of property (T) groups are in fact superrigid for cocycles
taking values in a countable group,
which implies, in the case that the group contains no nontrivial finite normal subgroups, that the action
is orbit equivalence superrigid (i.e., an orbit equivalence with any p.m.p.\ action
of any group implies that the groups are isomorphic and the actions measure conjugate) \cite{Pop07}.
In this sense free groups exhibit more of an affinity with amenable groups, for all
nontrivial Bernoulli actions of a given free group are orbit equivalent
(Bowen \cite{Bow11}) and all nontrivial Bernoulli actions of all noncyclic free groups are stably orbit equivalent
(Bowen \cite{Bow11a}),
despite the fact that free ergodic p.m.p.\ actions of free groups of different ranks are never orbit
equivalent (Gaboriau \cite{Gab00}).

Paradoxically enough, as one ventures further into the theory of Bernoulli superrigidity
it is precisely around this alignment of amenability with freeness that 
the general picture seems to coalesce.
Indeed what one discovers is that
superrigidity is governed
less by spectral gap per se 
than by certain expressions of anti-freeness or anti-treeability.
To begin with, Popa's cocycle superrigidity theorem in \cite{Pop07}
actually covers a broader class of groups,
namely those containing a w-normal subgroup with relative property (T),
and was subsequently augmented in \cite{Pop08} by a second cocycle superrigidity theorem that gives
the same conclusion for groups that contain two commuting infinite subgroups
at least one of which is nonamenable.
Peterson and Sinclair then demonstrated
in \cite{PetSin12} that $L^2$-rigidity is sufficient to imply Bernoulli cocycle superrigidity,
which enlarges the menu of groups to include those which are nonamenable but have property Gamma,
while Ioana and Tucker-Drob subsequently observed that nonamenable
inner amenable groups can also be added to the list \cite{TucDro14}.
Orbit equivalence superrigidity results were also established for a variety of 
p.m.p.\ actions and groups in earlier groundbreaking papers of Furman on lattices in Lie groups \cite{Fur99} 
and of Monod and Shalom on bounded cohomology \cite{MonSha06} as well as in work of Kida 
on mapping class groups that merely assumes freeness of the action \cite{Kid08}. 
One common feature of the groups that effectuate 
Bernoulli cocycle or orbit equivalence superrigidity in all of these cases, 
a feature which notably distinguishes them from noncyclic free groups,
is that their first $\ell^2$-Betti number vanishes.
Indeed Peterson and Sinclair showed in \cite{PetSin12} that this is a necessary condition
for the cocycle superrigidity of Bernoulli actions with atomless base, which has led to the speculation that it might also
be sufficient within the class of nonamenable groups.
The funny thing here is that amenable groups
also have vanishing first $\ell^2$-Betti number and thus have to be explicitly ruled out.
This reflects the fact a group can have vanishing first $\ell^2$-Betti number
for two very different and incompatible reasons: either because of anti-tree-like behaviour
(even when viewed through the rather coarse lens of measure equivalence)
or because of tree-like behaviour of a degenerate rank-one kind
(understood in the similarly generous sense of being measure equivalent to $\Zb$, a property that characterizes amenability
among countably infinite groups).

One may nevertheless wonder whether, as Robin Tucker-Drob has half-jokingly mused to us, there may be
a way of reconceptualizing the idea of Bernoulli cocycle superrigidity so that it naturally extends to amenable groups.
That this is not so far-fetched is suggested by the recent paper \cite{DeSHayHofSin19},
which in the case of $L^2$-rigidity explains how amenability can be smuggled in through a perturbative maneuver.
What we show in the present paper is that 
if we shift gears in the study of Bernoulli rigidity
to the setting of what we call {\it Shannon orbit equivalence},
in which the cocycle partitions have finite Shannon entropy
(as happens in a bounded or integrable orbit equivalence), then amenable groups truly do lose their exceptional status
and can be reunited with some of their nonamenable brethren
under the kind of common umbrella that remains
a chimera in the framework of general orbit equivalence.
Amenability in this case becomes largely
realigned with anti-tree-like behaviour, which is now to be understood in a much stricter geometric sense.
One must still exclude the virtually cyclic groups,
which remain too strongly tree-like, as well as
the locally finite groups
(but not the locally virtually cyclic groups
which fall outside of these two classes, like the rational numbers).

The basic geometric principle behind this was first identified and exploited
by Austin to show that measure entropy is an invariant of integrable orbit equivalence
for ergodic p.m.p.\ actions of finitely generated amenable groups, and more generally that there is
an entropy scaling formula for stable versions of both integrable and bounded orbit equivalence \cite{Aus16}.
In the non-virtually-cyclic case, Austin's arguments also give the same conclusions for Shannon orbit equivalence,
as one can easily verify.
One corollary of Austin's work, given the Ornstein--Weiss entropy
classification for Bernoulli actions of countably infinite amenable groups,
is that if two Bernoulli actions of a finitely generated infinite amenable group are integrably
orbit equivalent (or even just Shannon orbit equivalent if the group is not virtually cyclic) then they must be
measure conjugate.

Our main result, stated next and obtained as a direct
consequence of Theorems~\ref{T-Shannon SC to entropy} and Theorem~\ref{T-ordinal SC},
is a generalization of Austin's entropy invariance result
to a wider class of groups and represents an initial step towards answering his Question~1.2 in \cite{Aus16}.
Here $h_\mu (\cdot )$ denotes the maximum sofic measure entropy (see Section~\ref{SS-sofic measure entropy}).
W-normality is a weakening of normality which is recalled
in Definition~\ref{D-w-normal}.

\begin{theoremi}\label{T-measure}
Let $G$ be a countable group containing a w-normal amenable subgroup which is
neither locally finite nor virtually cyclic.
Let $H$ be a countable group. Let $G\curvearrowright (X,\mu )$ and $H\curvearrowright (Y,\nu )$ be free
p.m.p.\ actions which are Shannon orbit equivalent. Then
\[
h_\nu (H\curvearrowright Y) \geq h_\mu (G\curvearrowright X).
\]
\end{theoremi}

This result is also new in the case that $G$ and $H$ are amenable and $G$ is not
finitely generated. For amenable groups, the maximum sofic measure entropy
is equal to the amenable measure entropy and is realized by
every sofic approximation sequence \cite{Bow12a,KerLi13}.

Note that many measure conjugacy invariants like mixing and completely positive entropy can be destroyed
even under a bounded orbit equivalence, as shown by Fieldsteel and Friedman
in the case of ergodic p.m.p.\ $\Zb^d$-actions when $d\geq 2$ \cite{FieFri86}. On the other hand,
by a theorem of Belinskaya two ergodic p.m.p.\ $\Zb$-actions are
integrably orbit equivalent if and only if they are measure conjugate up to an isomorphism of the group
(``flip conjugate'') \cite{Bel68}. We don't know however whether one can substitute ``Shannon'' for ``integrably''
in Belinskaya's theorem.

Our proof of Theorem~\ref{T-measure}
is based on the key geometric idea of \cite{Aus16} involving the construction
of connected subgraphs which are sparse but nevertheless dense at a certain coarse scale.
This accounts for the exclusion of both local finiteness and
virtual cyclicity. In fact the conclusion of the theorem can fail
even for bounded orbit equivalence when $G$ and $H$ are locally finite,
as discovered by Vershik \cite{Ver73,Ver95} and discussed further below,
and while Austin was able to
conjure away the second restriction by an auxiliary argument we don't
see a way to remove it here
(we also note that the argument in \cite{Aus16} for handling virtually cyclic groups does not seem to work for
Shannon orbit equivalence because of its use of an ergodic theorem).
The subgraphs at play in our case will live not inside F{\o}lner sets as in \cite{Aus16}
but rather in the phase space itself.
While this permits us to cross the threshold into nonamenability,
it still imposes restrictions on the group which,
surely without coincidence,
have steered us into the realm of vanishing first $\ell^2$-Betti number (see Theorem~7.2 of \cite{Luc02}).
Indeed our strategy can be seen to fail for free groups
(see Theorem~\ref{T-without SC}).
There remains however the question of whether Theorem~\ref{T-measure} can be extended to other classes of
groups with vanishing first $\ell^2$-Betti number or related anti-tree-like geometric properties,
in particular the classes of groups for which Bernoulli cocycle superrigidity is known to hold.
Our argument still relies heavily on amenability, but in contrast to \cite{Aus16} we apply it
in the form of Ornstein--Weiss tiling technology.

We obtain from Theorem~\ref{T-measure} the following consequence for Bernoulli actions.
By the {\it base entropy} of a Bernoulli action $G\curvearrowright (X^G,\mu^G )$ we mean the Shannon entropy of $\mu$,
i.e., the supremum of the Shannon entropies of all finite partitions of $X$.
When $G$ is sofic this coincides with the sofic entropy for every sofic
approximation sequence \cite{Bow10,KerLi11}.

\begin{theoremi}\label{T-Bernoulli 1}
Let $G$ and $H$ be countable sofic groups containing a w-normal amenable subgroup
which is neither locally finite nor virtually cyclic.
Let $G\curvearrowright (X^G,\mu^G )$ and $H\curvearrowright (Y^H,\nu^H )$ be Bernoulli actions
which are Shannon orbit equivalent. Then these actions have the same base entropy.
\end{theoremi}

The Ornstein--Weiss entropy classification of Bernoulli actions of
countably infinite amenable groups \cite{OrnWei87} and a coinduction argument of Stepin \cite{Ste75}
together show that if a countably infinite group contains an infinite amenable subgroup then any two of its
Bernoulli actions are measure conjugate whenever they have the same base entropy
(in fact this statement holds for any countably infinite group by \cite{Bow12,Sew18}).
Theorem~\ref{T-Bernoulli 1} thus specializes to the case $G=H$ as follows.

\begin{theoremi}\label{T-Bernoulli 2}
Let $G$ be a countable sofic group containing a w-normal amenable subgroup which is neither locally finite nor virtually cyclic.
Then two Bernoulli actions of $G$ are Shannon orbit equivalent if and only if they are measure conjugate.
\end{theoremi}

Many of the groups $G$ satisfying the hypothesis of
Theorem~\ref{T-Bernoulli 2} have the much stronger property that their
Bernoulli actions are orbit equivalence superrigid, for example if $G$ satisfies
Bernoulli cocycle superigidity
and in addition has no nontrivial finite normal subgroups (see \cite{Pop07} or Theorem~6.16 of \cite{KerLi16}).
As mentioned above, $G$ will satisfy Bernoulli cocycle superigidity if it has property (T),
is the product of an infinite group and a nonamenable group, or is inner amenable and nonamenable,
and each of these three possibilities can occur within the class of groups in Theorem~\ref{T-Bernoulli 2}.
In particular, examples of property (T) groups whose centre is not locally virtually cyclic can be constructed
by taking products of copies of the group appearing in Example~1.7.13(iii) of \cite{BekHarVal08}.
On the other hand, it remains an open question whether Bernoulli orbit equivalence superrigidity holds
for the wreath product $\Zb\wr F_2$,
which is also covered by the above theorems.

To establish Theorem~\ref{T-measure}, we abstract
the graph-theoretic argument from \cite{Aus16} into a more generally applicable geometric principle that
we apply in a direct way to the dynamics to yield
what we call {\it property SC} for both free p.m.p.\ actions and (by universally quantifying over such actions) groups.
Given that the argument in \cite{Aus16} is localized to F{\o}lner sets, it might seem more natural here to
follow the usual recipe and instead localize to sofic approximations, which one can do successfully
in the case of topological entropy, but in testing such an approach in the measure setting we have found
ourselves unable
to control the empirical distribution of the sofic dynamical models (``microstates'')
except under special circumstances, and even then we could only derive a result for bounded orbit equivalence
(one does however get some extra mileage in such circumstances, as we will show in another paper).
We prove Theorem~\ref{T-measure} by showing that its
conclusion is valid assuming that the action of $G$ has property SC (Theorem~\ref{T-Shannon SC to entropy})
and that this hypothesis is automatic for the groups $G$ in the statement of the theorem by virtue of these
$G$ having themselves property SC (Theorem~\ref{T-ordinal SC}).

In the case of amenable groups, we show in Proposition~\ref{P-amenable SC} that property SC
is equivalent to the group being neither locally finite nor virtually cyclic.
As mentioned above, the exclusion of local finiteness cannot be removed from Theorem~\ref{T-measure},
as a theorem from Vershik's thesis demonstrates \cite{Ver73} (see the presentation in \cite{Ver95}).
Indeed suppose that $G$ and $H$ are infinite locally finite groups
and suppose that there are nested finite subgroups $G_1 \subseteq G_2 \subseteq\dots$ of $G$
with $\bigcup_{n=1}^\infty G_n = G$
and nested finite subgroups $H_1 \subseteq H_2 \subseteq\dots$ of $H$ with $\bigcup_{n=1}^\infty H_n = H$
such that $|G_n| = |H_n|$ for all $n$. Then for any two free ergodic p.m.p.\ actions
$G\curvearrowright (X, \mu)$ and $H \curvearrowright (Y, \nu)$ there are integers
$1\leq n_1 < n_2 < \dots$ and an orbit equivalence $\Psi : X\to Y$ such that
$\Psi (G_{n_k}x)=H_{n_k}\Psi (x)$ for all $x\in X$ and $k\in\Nb$
(this is a special type of bounded orbit equivalence).
To see this, for every $n$
consider the $\sigma$-algebra $\xi_n$ of $G_n$-invariant Borel subsets of $X$. Then $\{\xi_n \}$ is an ergodic homogeneous sequence in the terminology of \cite{Ver95}. Similarly, one has the ergodic homogeneous sequence
$\{ \xi'_n \}$ whose $n$th term is the $\sigma$-algebra of $H_n$-invariant Borel subsets of $Y$.
Again using terminology from \cite{Ver95},
the sequences $\{ \xi_n\}$ and $\{ \xi'_n\}$ have the same type $\{r_n\}$, in this case given by
$r_n=[G_n : G_{n-1} ]=[H_n : H_{n-1} ]$ with $G_0$ and $H_0$ denoting the trivial subgroups of $G$ and $H$, respectively.
Now Corollary~1 on page 723 of \cite{Ver95} says that any two ergodic homogeneous sequences of the same type are lacunary isomorphic in the sense that there are integers $1\leq n_1 < n_2 < \dots$ and a measure isomorphism
$\Psi :X\to Y$ which for every $k$ sends $\xi_{n_k}$ to $\xi'_{n_k}$, so that
$\Psi (G_{n_k}x)=H_{n_k}\Psi (x)$, as desired.
We thus see in particular that if $G$ is a countably infinite locally finite group
then all of the free ergodic p.m.p.\ actions of $G$ are boundedly orbit equivalent to each other.
This includes all of the nontrivial Bernoulli actions of $G$, which exhaust the possible nonzero values of measure entropy.

We begin in Section~\ref{S-preliminaries} by setting up notation and reviewing basic concepts
and terminology concerning orbit equivalence and sofic measure entropy
(for general references on these topics see \cite{KerLi16,KecMil04}).
In Section~\ref{SS-SC} we define property SC and establish two permanence properties.
In Sections~\ref{SS-shrinking} and \ref{SS-SC'} we introduce a shrinking property
and two variants of property SC which will be of subsequent technical use.
In Section~\ref{SS-normal} we prove that property SC passes from a normal subgroup to the ambient group.
In Section~\ref{SS-without SC} we identify two classes of groups without property SC (Theorem~\ref{T-without SC}),
while Section~\ref{SS-with SC} is dedicated to showing that
the groups satisfying the hypothesis of Theorem~\ref{T-measure}
have property SC (Theorem~\ref{T-ordinal SC}).
In Section~\ref{SS-product groups} we derive a result on property SC that concerns product groups.
Finally, we devote Section~\ref{S-Shannon OE} to the proof
of Theorem~\ref{T-Shannon SC to entropy}, which together with Theorem~\ref{T-ordinal SC}
yields Theorem~\ref{T-measure}.
\medskip

\noindent{\it Acknowledgements.}
The first author was partially supported by NSF grant DMS-1800633.
Preliminary stages of this work were carried out during his six-month stay in 2017-2018 at the ENS de Lyon,
during which time he held ENS and CNRS visiting professorships and was supported by Labex MILYON/ANR-10-LABX-0070.
He thanks Damien Gaboriau and Mikael de la Salle at the ENS for their generous hospitality.
The second author was partially supported by NSF grants DMS-1600717 and DMS-1900746.
We thank Robin Tucker-Drob for comments.

\section{Preliminaries}\label{S-preliminaries}

\subsection{Basic notation and terminology}

Throughout the paper $G$ and $H$ denote countably infinite discrete groups,
with identity elements $e_G$ and $e_H$ (many of our results are also valid
for finite groups, usually for trivial reasons, but we make this exclusion for the convenience
of forcing the measure in a free probability-measure-preserving action to be atomless, as reiterated below).
We denote by $\cF(G)$ the set of all nonempty
finite subsets of $G$, and by $\overline{\cF}(G)$ the set of symmetric finite subsets of $G$ containing $e_G$.

A {\it left F{\o}lner sequence} for the group $G$ is a sequence $\{ F_n \}$ of nonempty finite subsets of $G$
such that $\lim_{n\to\infty} |gF_n \Delta F_n |/|F_n| = 0$ for all $g\in G$. If $G$ admits a left F{\o}lner sequence
then it is said to be {\it amenable}.

If P is a property then one says that a group is {\it virtually P}
if it has a subgroup of finite index with property P, and {\it locally P} if each of its
finitely generated subgroups has property P.

For a nonempty finite set $V$, we denote by $\Pb_V$ the algebra of all subsets of $V$,
by $\Sym (V)$ the group of all permutations of $V$,
and by $\m$ the uniform probability measure on $V$.

By a {\it standard probability space} we mean a standard Borel space (i.e., a Polish space with its Borel $\sigma$-algebra) equipped with a probability measure. By a
partition of such a space we will always mean one that is Borel (and also, if occasion demands, one
that is really only a partition of a conull subset of the space).
By a {\it p.m.p.\ (probability-measure preserving) action} of $G$ we mean an action $G\curvearrowright (X,\mu )$
of $G$ on a standard probability space by measure-preserving transformations.
We express such an action using the concatenation $(g,x)\mapsto gx$ for $g\in G$ and $x\in X$
(in principle this will result in an ambiguity when two actions of the same group on the same space are at play,
but in context the notation chosen for group elements will make the distinction clear).
Two such actions $G\curvearrowright (X,\mu )$ and $G\curvearrowright (Y,\nu )$ are {\it measure conjugate}
or {\it isomorphic} if there exist $G$-invariant conull sets $X_0 \subseteq X$ and $Y_0 \subseteq Y$ and
a $G$-equivariant measure isomorphism $X_0 \to Y_0$.

A {\it Bernoulli action} is a p.m.p.\ action of the form $G\curvearrowright (Y^G ,\nu^G )$
where $(Y,\nu )$ is a standard probability space and $(gy)_h = y_{g^{-1} h}$ for $y\in Y^G$ and $g,h\in G$.
It is {\it nontrivial} if $\nu$ does not have an atom with full measure.

A p.m.p.\ action $G\curvearrowright (X,\mu )$ is {\it free} if the
set $X_0$ of all $x\in X$ such that $sx \neq x$ for all $s\in G\setminus \{ e_G \}$ has measure one.
For the purposes of this paper there is never any harm in replacing $X$ by a $G$-invariant conull subset
and so we will always assume that $X_0 = X$ for the purposes of argumentation, even if theorem statements
themselves do not require it.

Given a p.m.p.\ action $G\curvearrowright (X,\mu )$ and a set $S\subseteq G$,
we define an {\it $S$-path} in $X$ to be a finite tuple $(x_0 , x_1 ,\dots , x_n )$ of points in $X$
such that for every $i=1,\dots , n$ there is a $g\in S$ for which $x_i = gx_{i-1}$,
in which case we call $n$ the {\it length}
of the path and say that the path {\it connects} $x_0$ and $x_n$.
When $n=1$ we also speak of an {\it $S$-edge}.

As indicated above, our reason for making the blanket assumption that the groups $G$ and $H$ be infinite
is that for any of their free p.m.p.\ actions on a standard probability space $(X,\mu )$ the
measure $\mu$ is forced to be atomless, a fact which we will often tacitly rely on. It is required for instance
in our various applications of Ornstein--Weiss tiling technology.

\subsection{Shannon orbit equivalence}\label{SS-Shannon}

We say that two free p.m.p.\ actions $G\curvearrowright (X,\mu )$ and
$H\curvearrowright (Y,\nu )$
are {\it orbit equivalent} if there are a $G$-invariant conull set $X_0 \subseteq X$,
an $H$-invariant conull set $Y_0 \subseteq Y$, and a measure isomorphism
$\Psi : X_0 \to Y_0$ such that $\Psi (Gx) = H\Psi (x)$ for all $x\in X_0$.
Such a $\Psi$ is called an {\it orbit equivalence}.

Associated to a $\Psi$ as in the above definition are the cocycles
$\kappa : G\times X_0 \to H$ and $\lambda : H\times Y_0 \to G$ determined by
\begin{align*}
\Psi (gx) &= \kappa (g,x)\Psi (x) , \\
\Psi^{-1} (ty) &= \lambda (t,y) \Psi^{-1} (y)
\end{align*}
for all $g\in G$, $x\in X_0$, $t\in H$, and $y\in Y_0$. The defining property of a
cocycle, referred to as the {\it cocycle identity}, is expressed in the case of $\kappa$ by
\begin{align*}
\kappa (fg , x) = \kappa (f ,gx) \kappa (g,x)
\end{align*}
for all $f,g \in G$ and $x\in X_0$.
Note also that
\begin{align*}
\kappa (\lambda (t,y),\Psi^{-1} (y)) &= t , \\
\lambda (\kappa (g,x),\Psi (x)) &= g
\end{align*}
for all $t\in H$, $y\in Y_0$, $g\in G$, and $x\in X_0$.

The {\it Shannon entropy} of a countable Borel partition $\sP$ of $X$ is defined by
\[
H_\mu (\sP ) = \sum_{P\in\sP} -\mu (P)\log \mu (P) .
\]
with $-x\log x$ being interpreted as $0$ when $x=0$.
We say that the actions are {\it Shannon orbit equivalent}
if the sets $X_0$ and $Y_0$ and the measure isomorphism $\Psi$
can be chosen so that for each $g\in G$ the countable Borel partition of $X_0$ consisting of the sets
\[
X_{g,t} := \{ x\in X_0 : \Psi (gx) = t\Psi (x) \}
\]
for $t\in H$ has finite Shannon entropy and, likewise,
for each $t\in H$ the countable partition of $X_0$ consisting of the sets $X_{g,t}$ for $g\in G$
has finite Shannon entropy. In this case we refer to $\Psi$ as a
{\it Shannon orbit equivalence}.

In general, we say that a map $f: X\rightarrow H$ is {\it Shannon}
if the countable partition $\{f^{-1}(t): t\in H\}$ of $X$ has finite Shannon entropy.
A cocycle $\kappa : G\times X\rightarrow H$ is {\it Shannon} if $\kappa(g, \cdot)$
is Shannon for every $g\in G$.

\subsection{Bounded and integrable orbit equivalence}\label{SS-bounded and integrable}

Let $G\curvearrowright (X,\mu )$ and $H\curvearrowright (Y,\nu )$ be
free p.m.p.\ actions which are orbit equivalent, with $X_0$, $Y_0$, and $\Psi$
witnessing the orbit equivalence as above,
and $\kappa$ and $\lambda$ denoting the associated cocycles.

We say that the cocycle $\kappa : G\times X_0 \to H$ is {\it bounded} if $\kappa (g,X_0 )$ is finite
for every $g\in G$, and define boundedness for $\lambda$ likewise.
If $X_0$, $Y_0$, and $\Psi$
can be chosen so that each of the cocycles $\kappa$ and $\lambda$
is bounded then we say that
the actions are {\it boundedly orbit equivalent}, and refer to $\Psi$ as a
{\it bounded orbit equivalence}.

Suppose now that $G$ and $H$ are finitely generated and write $\ell_G$
and $\ell_H$ for the word length functions with respect to some symmetric finite generating sets
for $G$ and $H$, respectively. We say that the cocycle $\kappa : G\times X_0 \to H$
is {\it integrable} if for every $g\in G$ one has
\[
\int_X \ell_H (\kappa (g,x))\, d\mu (x) < \infty ,
\]
and define integrability for $\lambda$ likewise. If $X_0$, $Y_0$, and $\Psi$
can be chosen so that each of the cocycles $\kappa$ and $\lambda$
is integrable then we say that
the actions are {\it integrably orbit equivalent}, and refer to $\Psi$ as an
{\it integrable orbit equivalence}.

Obviously every bounded orbit equivalence
is integrable. Lemma~2.1 of \cite{Aus16} shows that every integrable orbit equivalence is Shannon.

\subsection{Sofic approximations}

On the set $V^V$ of maps from a nonempty finite set $V$ to itself we define the
normalized Hamming distance by
\[
\rho_\Hamm (T,S) = \frac{1}{|V|} |\{ v\in V : Tv \neq Sv \} | .
\]
By a {\it sofic approximation} for $G$ we mean a (not necessarily multiplicative) map
$\sigma : G\to\Sym (V)$ for some nonempty finite set $V$.
Given a finite set $F\subseteq G$ and an $\delta > 0$, we say that such a $\sigma$
is an {\it $(F,\delta )$-approximation} if
\begin{enumerate}
\item $\rho_\Hamm (\sigma_{st} , \sigma_s \sigma_t ) \leq \delta$ for all $s,t\in F$, and

\item $\rho_\Hamm (\sigma_s , \sigma_t ) \geq 1-\delta$ for all distinct $s,t\in F$.
\end{enumerate}
By a {\it sofic approximation sequence} for $G$ we mean a sequence
$\Sigma = \{ \sigma_k : G\to\Sym (V_k ) \}_{k=1}^\infty$
of sofic approximations for $G$ such that
for every finite set $F\subseteq G$ and $\delta > 0$ there is a $k_0 \in\Nb$ such that
$\sigma_k$ is an $(F,\delta )$-approximation for every $k\geq k_0$.

By saying that a sofic approximation $\sigma : G\to\Sym (V)$ is {\it good enough} we mean that
it is an $(F,\delta )$-approximation for some finite set $F\subseteq G$ and $\delta > 0$
and that this condition is sufficient for the purpose at hand.

The group $G$ is said to be {\it sofic} if it admits a sofic approximation sequence.
This is the case when $G$ is amenable or residually finite,
and indeed soficity can be regarded in a natural way as a simultaneous generalization of these two properties
(see Section~10.2 of \cite{KerLi16}). In particular, free groups are sofic.
It remains unknown whether nonsofic groups exist.

\subsection{Sofic measure entropy}\label{SS-sofic measure entropy}

Let $G\curvearrowright (X,\mu )$ be a p.m.p.\ action.
Let $\sC$ be a finite partition of $X$, $F$ a finite subset of $G$ containing $e_G$, and $\delta > 0$.
We write $\alg (\sC )$ for the algebra generated by $\sC$, consisting of all unions of members of $\sC$,
and denote by $\sC_F$ the join $\bigvee_{s\in F} s\sC$.
Let $\sigma : G\to \Sym(V)$ be a sofic approximation for $G$. We define $\Hom_\mu (\sC ,F,\delta ,\sigma )$ to be the
set of homomorphisms $\varphi : \alg (\sC_F )\to \Pb_V$ satisfying
\begin{enumerate}
\item $\sum_{A\in\sC} \m (\sigma_g \varphi (A) \Delta \varphi (gA)) < \delta$ for all $g\in F$, and

\item $\sum_{A\in\sC_F} |\m (\varphi (A)) - \mu (A)| < \delta$.
\end{enumerate}
For a finite partition $\sP\leq\sC$ we define $|\Hom_\mu (\sC ,F,\delta ,\sigma )|_\sP$ to be the cardinality
of the set of restrictions of elements of $\Hom_\mu (\sC ,F,\delta ,\sigma )$ to $\sP$.

Suppose now that $G$ is sofic and let $\Sigma = \{ \sigma_k : G\to\Sym (V_k ) \}_{k=1}^\infty$
be a sofic approximation sequence for $G$. For a finite partition $\sP$ of $X$ we write,
notationally omitting the action for brevity,
\begin{align*}
h_{\Sigma ,\mu} (\sP,\sC,F,\delta ) &= \limsup_{k\to\infty} \frac{1}{|V_k|} \log |\Hom_\mu (\sC ,F,\delta ,\sigma_k ) |_\sP , \\
h_{\Sigma ,\mu} (\sP) &= \inf_{\sC\geq\sP} \inf_F \inf_{\delta > 0} h_{\Sigma ,\mu} (\sC,\sP,F,\delta )
\end{align*}
where the first infimum is over all finite partitions $\sC$ of $X$ refining $\sP$ and the second
is over all finite sets $F\subseteq G$ containing $e_G$. We also write
$h_{\Sigma ,\mu} (G\curvearrowright X,\sP)$ for $h_{\Sigma ,\mu} (\sP)$ when it is necessary
to explicitly indicate the action.
We define the {\it sofic measure entropy} of the action $G\curvearrowright (X,\mu )$ with respect to $\Sigma$ by
\begin{align*}
h_{\Sigma ,\mu} (G\curvearrowright X) = \sup_{\sP} h_{\Sigma ,\mu} (\sP ),
\end{align*}
where $\sP$ ranges over all finite partitions of $X$.

For a p.m.p.\ action $G\curvearrowright (X,\mu )$ of an arbitrary $G$ we define
the {\it maximum sofic measure entropy} by
\[
h_\mu (G\curvearrowright X) = \max_\Sigma h_{\Sigma ,\mu} (G\curvearrowright X)
\]
where $\Sigma$ ranges over all sofic approximation sequences for $G$
(in the case that $G$ is nonsofic we interpret this maximum to be $-\infty$).
The following proposition shows that the maximum does indeed exist.

\begin{proposition} \label{P-maximal measure entropy}
Suppose that $G$ is sofic. Let $G\curvearrowright (X, \mu)$ be a p.m.p.\ action.
Then there is a sofic approximation sequence $\Pi$ for $G$ such that
$h_{\Pi, \mu}(G\curvearrowright X)\ge h_{\Pi', \mu}(G\curvearrowright X)$ for every
sofic approximation sequence $\Pi'$ for $G$.
\end{proposition}

\begin{proof}
Put $M=\sup_{\Pi'} h_{\Pi', \mu}(G\curvearrowright X)$ where $\Pi'$ ranges over all sofic approximation sequences $\Pi'$ for $G$. Take a sequence $\{\Pi_n =\{\pi_{n, k}\}_{k\in \Nb}\}_{n\in \Nb}$ of sofic approximation sequences for $G$
such that $h_{\Pi_n, \mu}(G\curvearrowright X)\to M$ as $n\to \infty$.

Choose an increasing sequence $F_1\subseteq F_2\subseteq \dots$ of finite subsets of $G$ with union $G$.
For each $n\in \Nb$ there exists a $K_n\in \Nb$ such that for every $k\ge K_n$ the map $\pi_{n, k}$
is an $(F_n, 1/n)$-approximation for $G$.
Put $\sW=\{(n, k)\in \Nb^2: k\ge K_n\}$. Then $\sW$ is countably infinite,
and so we can take a bijection $\varphi: \Nb\rightarrow \sW$. Put $\pi_k=\pi_{\varphi(k)}$ for every $k\in \Nb$ and $\Pi=\{\pi_k\}_{k\in \Nb}$. Then $\Pi$ is a sofic approximation sequence for $G$. For any $n\in \Nb$, any finite Borel partitions $\sC\preceq \sU$ of $X$, any $L\in \cF(G)$ containing $e_G$, and any $\delta>0$, we have
\[
h_{\Pi,\mu}(\sC, \sU, L, \delta)\ge h_{\Pi_n ,\mu}(\sC, \sU, L, \delta).
\]
Thus $h_{\Pi, \mu}(G\curvearrowright X)\ge h_{\Pi_n, \mu}(G\curvearrowright X)$ for every $n\in \Nb$. Therefore $h_{\Pi, \mu}(G\curvearrowright X)=M$.
\end{proof}

The measure entropy $h_{\Sigma ,\mu} (G\curvearrowright X)$ is known not to depend on
the choice of sofic approximation sequence $\Sigma$ in the following cases:
\begin{enumerate}
\item the group $G$ is amenable, in which case
we always recover the amenable measure entropy \cite{KerLi13,Bow12a},

\item the action is Bernoulli \cite{Bow10,KerLi11},

\item the action
is an algebraic action of the form $G\curvearrowright (\widehat{(\Zb G)^n / (\Zb G)^n A} ,\mu )$
where $A\in M_n (\Zb G)$ is injective as an operator on $\ell^2 (G)^{\oplus n}$ and $\mu$ is the
normalized Haar measure \cite{Hay16}.
\end{enumerate}

\section{Property SC} \label{S-SC}

A reminder that $G$ and $H$ throughout the paper are countably infinite groups,
which in particular forces the measure in any of their free p.m.p.\ actions to be atomless.

\subsection{Definition and two permanence properties} \label{SS-SC}

\begin{definition} \label{D-SC}
Let $\fC$ be a class of free p.m.p.\ actions of a fixed $G$.
We say that $\fC$ has {\it property SC} (``sparsely connected'') if for any function $\Upsilon: \cF(G)\rightarrow [0, \infty)$
there exists an $S\in \overline{\cF}(G)$ such that for any $T\in \overline{\cF}(G)$,
there are $C, n\in \Nb$, and  $S_1, \dots, S_n\in \overline{\cF}(G)$  so that for any $G\curvearrowright (X, \mu)$ in $\fC$ there are Borel sets $W,\cV_1 , \dots ,\cV_n \subseteq X$ satisfying the following conditions:
\begin{enumerate}
\item $\sum_{j=1}^n \Upsilon(S_j)\mu(\cV_j)\le 1$,
\item $SW=X$,
\item if $w_1, w_2\in W$ satisfy $gw_1=w_2$ for some $g\in T$ then $w_1$ and $w_2$ are connected by a path of length at most $C$ in which each edge is an $S_j$-edge with both endpoints in $\cV_j$ for some $1\le j\le n$.
\end{enumerate}
We say that a free p.m.p.\ action $G\curvearrowright (X, \mu)$ has {\it property SC} if the singleton class
containing it has property SC.
We say that $G$ itself has {\it property SC} if the class of all free p.m.p\ actions $G\curvearrowright (X, \mu)$ has property SC.
\end{definition}

\begin{remark}\label{R-SC no bound}
When $\fC$ consists of either a single free p.m.p.\ action or
all free p.m.p.\ actions of a fixed $G$, one can omit the bound $C$ since its existence is automatic,
as we will verify in the paragraph following Proposition~\ref{P-Bernoulli}.
\end{remark}

We now record a couple of permanence properties. We will also later see in
Section~\ref{SS-normal} that if $G$ has a normal subgroup
with property SC then $G$ itself has property SC (Proposition~\ref{P-normal SC})
and that a prescribed finite-index subgroup of $G$ has property SC if and only if $G$ does
(Proposition~\ref{P-finite index}).

\begin{proposition} \label{P-SC under OE}
Let $G\curvearrowright (X, \mu)$ and $H\curvearrowright (Y, \nu)$ be
free p.m.p.\ actions which are boundedly orbit equivalent. Suppose that $G\curvearrowright (X, \mu)$ has property SC. Then $H\curvearrowright (Y, \nu)$ has property SC.
\end{proposition}

\begin{proof}
We may assume that $(X, \mu)=(Y, \nu)$ and that the identity map of $X$ provides a bounded orbit equivalence between the actions $G\curvearrowright (X, \mu)$ and $H\curvearrowright (X, \mu)$.  We may also assume that both $G\curvearrowright X$ and $H\curvearrowright X$ are free. Then we have the cocycles $\kappa$ and $\lambda$ as in Section~\ref{SS-Shannon}.

Let $\Upsilon_H$ be a function $\cF(H)\rightarrow [0, \infty)$.
Define a function $\Upsilon_G: \cF(G)\rightarrow [0, \infty)$ by $\Upsilon_G(F)=\Upsilon_H(\kappa(F, X))$.
Since $G\curvearrowright (X, \mu)$ has property SC, there exists an $S_G\in \overline{\cF}(G)$ such that for every $T_G\in \overline{\cF}(G)$ there are $C_G, n_G\in \Nb$, $S_{G, 1}, \dots, S_{G, n_G}\in \overline{\cF}(G)$, and Borel  subsets $W_G$ and $\cV_{G, j}$ of $X$ for $1\le j\le n_G$  satisfying the following conditions:
\begin{enumerate}
\item $\sum_{j=1}^{n_G} \Upsilon_G(S_{G, j})\mu(\cV_{G, j})\le 1$,
\item $S_G W_G=X$,
\item if $w_1, w_2\in W_G$ satisfy $gw_1=w_2$ for some $g\in T_G$ then $w_1$ and $w_2$ are connected by a path of length at most $C_G$ in which each edge is an $S_{G, j}$-edge with both endpoints in $\cV_{G, j}$ for some $1\le j\le n_G$.
\end{enumerate}

Set $S_H=\kappa(S_G, X)\in \overline{\cF}(H)$.

Let $T_H\in \overline{\cF}(H)$. Set $T_G=\lambda(T_H, X)\in \overline{\cF}(G)$.
Then we have $C_G, n_G$, $S_{G, j}$ for $1\le j\le n_G$, $W$, and $\cV_{G, j}$ for $1\le j\le n_G$ as above. Set $C_H=C_G$, $n_H=n_G$, $S_{H, j}=\kappa(S_{G, j}, X)\in \overline{\cF}(H)$ for $1\le j\le n_H=n_G$. Also, set
$W_H=W_G$ and $\cV_{H, j}=\cV_{G, j}$ for all $1\le j\le n_H=n_G$. Then
\begin{align*}
\sum_{j=1}^{n_H} \Upsilon_H(S_{H, j})\mu(\cV_{H, j})=\sum_{j=1}^{n_G} \Upsilon_H(\kappa(S_{G, j}, X))\mu(\cV_{G, j})=\sum_{j=1}^{n_G} \Upsilon_G(S_{G, j})\mu(\cV_{G, j})\le 1,
\end{align*}
verifying condition (i) in Definition~\ref{D-SC}.
Note that $X=S_GW_G\subseteq S_HW_G=S_HW_H$. Thus $X=S_HW_H$,
which verifies condition (ii) in Definition~\ref{D-SC}.
Let $h\in T_H$ and $w_1, w_2\in W_H$ with $hw_1=w_2$.  Then
\[
w_2=hw_1=\lambda(h, w_1)w_1\in T_Gw_1.
\]
Thus $w_1$ and $w_2$ are connected by a path of length at most $C_G$ in which each edge is an $S_{G, j}$-edge with both endpoints in $\cV_{G, j}$ for some $1\le j\le n_G$.
Such an edge is also an $S_{H, j}$-edge with both endpoints in $\cV_{H, j}$.
This verifies condition (iii) in Definition~\ref{D-SC}.
\end{proof}

Recall that a p.m.p.\ action $G\curvearrowright (X, \mu)$ is said to {\it weakly contain}
another p.m.p.\ action $G\curvearrowright (Y, \nu)$ if for every finite set $F\subseteq G$,
finite collection of Borel sets $B_1 , \dots , B_n \subseteq Y$, and $\delta > 0$ there exist
Borel sets $A_1 , \dots , A_n \subseteq X$ such that
$|\mu (s_i A_i \cap A_j ) - \nu (sB_i \cap B_j )| < \delta$ for all $s\in F$ and $1\leq i,j \leq n$
\cite[Section~10]{Kec10}.

\begin{proposition} \label{P-weak containment}
Let $G\curvearrowright (Y, \nu)$ be a free p.m.p.\ action with property SC. Then the class $\fC$ of all
free p.m.p.\ actions $G\curvearrowright (X, \mu)$ which weakly contain
$G\curvearrowright (Y, \nu)$ has property SC.
\end{proposition}

\begin{proof}
Let $\Upsilon$ be a function $\cF(G)\rightarrow [0, \infty)$. Since $G\curvearrowright (Y, \nu)$ has property SC, there is an $S\in \overline{\cF}(G)$ such that for any $T\in \overline{\cF}(G)$ there are $C, n\in \Nb$ and $S_1, \dots, S_n\in \overline{\cF}(G)$ and Borel sets $W, \cV_j\subseteq Y$ for $1\le j\le n$ satisfying the following conditions:
\begin{enumerate}
\item $\sum_{j=1}^n 2\Upsilon(S_j)\nu(\cV_j)\le 1$,
\item $SW=Y$,
\item if $w_1, w_2\in W$ satisfy $tw_1=w_2$ for some $t\in T$ then $w_1$ and $w_2$ are connected by a path of length at most $C$ in which every edge is an $S_j$-edge with both endpoints in $\cV_j$ for some $1\le j\le n$.
\end{enumerate}

Let $T\in \overline{\cF}(G)$. Then we have $C, n, S_1, \dots, S_n, W, \cV_1, \dots, \cV_n$ as above.
Write $[n] = \{ 1,\dots ,n\}$ and $[C] = \{ 1,\dots ,C\}$ for brevity.
Denote by $\Xi$ the set of $(f, h)$ such that $f$ is a function $[C]\rightarrow [n]$ and $h$ is a function $[C]\rightarrow G$ satisfying $h(j)\in S_{f(j)}$ for all $j\in [C]$. For each $t\in T$, denote by $\Xi_t$ the set of $(f, h)\in \Xi$ satisfying $h(C)h(C-1)\dots h(1)=t$. For each $(f, h)\in \Xi$, denote by $Y_{f, h}$ the set of $y\in Y$ satisfying $h(j-1)\dots h(1)y, h(j)h(j-1)\dots h(1)y\in \cV_{f(j)}$ for all $j\in [C]$. Then the above condition (iii) means that for every $t\in T$ the set $W\cap t^{-1}W$ is contained in $\bigcup_{(f, h)\in \Xi_t}Y_{f, h}$.

Put $S_{n+1}=T\in \overline{\cF}(G)$. Take $0<\delta<1$ such that $\delta(\Upsilon(T)|T|(2|T|+2C|\Xi|)+\sum_{j=1}^n\Upsilon(S_j))\le 1/2$.

Let $G\curvearrowright (X, \mu)$ be an action in $\fC$. Because it weakly contains $G\curvearrowright (Y, \nu)$, there are Borel sets $W'\subseteq X$, $\cV_j'\subseteq X$ for $j\in [n]$, and $X_{f, h}\subseteq X$ for $(f, h)\in \Xi$ satisfying the following conditions:
\begin{enumerate}
\item[(i)] $|\mu(SW')-\nu(SW)|<\delta$,
\item[(ii)] $|\mu(\cV_j')-\nu(\cV_j)|<\delta$ for all $j\in [n]$,
\item[(iii)] $\mu(h(j-1)\dots h(1)X_{f, h}\setminus \cV_j'), \mu(h(j)h(j-1)\dots h(1)X_{f, h}\setminus \cV_j')<\delta$ for all $(f, h)\in \Xi$ and $j\in [C]$,
\item[(iv)] $\mu((W'\cap t^{-1}W')\setminus \bigcup_{(f, h)\in \Xi_t}X_{f, h})<\delta$ for each $t\in T$.
\end{enumerate}

Put $W''=W'\cup (X\setminus SW')$. Then $SW''=X$, verifying condition (ii) in Definition~\ref{D-SC}.

For each $(f, g)\in \Xi$, denote by $X_{f, h}'$ the set of $x\in X_{f, h}$ satisfying $h(j-1)\dots h(1)x, h(j)h(j-1)\dots h(1)x\in \cV_{f(j)}'$ for all $j\in [C]$. Then $\mu(X_{f, h}\setminus X_{f, h}')<2C\delta$.

For each $t\in T$, put $W^\dag_t=(W'\cap t^{-1}W')\setminus \bigcup_{(f, h)\in \Xi_t}X_{f, h}'$. Then
\begin{align*}
\mu\bigg(\bigcup_{t\in T}W^\dag_t\bigg)&\le \sum_{t\in T}\mu(W^\dag_t)\\
&\le \sum_{t\in T}\mu\bigg((W'\cap t^{-1}W')\setminus \bigcup_{(f, h)\in \Xi_t}X_{f, h}\bigg)+\sum_{t\in T}\sum_{(f, h)\in \Xi_t}\mu(X_{f, h}\setminus X_{f, h}')\\
&< |T|\delta+2C|\Xi|\delta.
\end{align*}

Put $\cV_{n+1}'=T(T(X\setminus SW')\cup \bigcup_{t\in T}W^\dag_t)\subseteq X$.  Then
\begin{align*}
\sum_{j=1}^{n+1}\Upsilon(S_j)\mu(\cV_j')&\le \Upsilon(S_{n+1})\mu(\cV_{n+1}')+\sum_{j=1}^n\Upsilon(S_j)(\nu(\cV_j)+\delta) \\
&\le \Upsilon(T)|T|(|T|\delta+|T|\delta+2C|\Xi|\delta)+\frac12 +\delta \sum_{j=1}^n\Upsilon(S_j)\le 1,
\end{align*}
verifying condition (i) in Definition~\ref{D-SC}.

Let $t\in T$ and $w_1, w_2\in W''$ with $tw_1=w_2$. If $w_1\in X'_{f, h}$ for some
$(f, h)\in \Xi_t$, then $w_1$ and $w_2$ are connected by the path $w_1, h(1)w_1, \dots, h(C)\dots h(1)w_1=tw_1=w_2$ of length $C$ whose $j$th edge is an $S_{f(j)}$-edge with both endpoints in $\cV_{f(j)}'$ for all $1\le j\le C$. Thus we may assume that $w_1\in (W''\cap t^{-1}W'')\setminus \bigcup_{(f, h)\in \Xi_t}X'_{f, h}\subseteq (X\setminus SW')\cup t^{-1}(X\setminus SW')\cup W^\dag_t$. It follows that $(w_1, w_2)$ is an $S_{n+1}$-edge with both endpoints in $\cV_{n+1}'$. This verifies condition (iii) in Definition~\ref{D-SC}.
\end{proof}

Using the above proposition we derive the following.

\begin{proposition} \label{P-Bernoulli}
The following conditions are equivalent:
\begin{enumerate}
\item $G$ has property SC,
\item every free p.m.p.\ action $G\curvearrowright (X, \mu)$ has property SC,
\item there exists a nontrivial Bernoulli action of $G$ with property SC.
\end{enumerate}
\end{proposition}

\begin{proof}
(i)$\Rightarrow$(ii)$\Rightarrow$(iii). Trivial.

(iii)$\Rightarrow$(i). Combine Proposition~\ref{P-weak containment} with the theorem of Ab\'{e}rt and Weiss that every free
p.m.p.\ action of $G$ weakly contains every nontrivial Bernoulli action of $G$ \cite{AbeWei13}.
\end{proof}

We can now verify that the existence of the bound $C$ in Definition~\ref{D-SC} is automatic
when $\fC$ consists of either a single free p.m.p.\ action or all free p.m.p.\ actions of a fixed $G$.
It suffices to check the case of a single free p.m.p.\ action $G\curvearrowright (X, \mu)$ 
in view of the implication (iii)$\Rightarrow$(i) in Proposition~\ref{P-Bernoulli}.
Suppose then that the action $G\curvearrowright (X, \mu)$ satisfies the weaker formulation that omits the $C$
and suppose that we are given $S$, $T$,  $n$, $S_j$ for $1\le j\le n$, and $W$ and $\cV_j$ for $1\le j\le n$ satisfying the weaker condition with respect to the function $2\Upsilon$. For each $g\in T$ and $C\in \Nb$ we define $W_{g, C}$ to be the set of all $w_1\in W\cap g^{-1}W$ such that $w_1$ and $gw_1$ can be connected by a path of length at most $C$ in which each edge is an $S_j$-edge with both endpoints in $\cV_j$ for some $1\le j\le n$. Then $W\cap g^{-1}W$ is the increasing union of the sets $W_{g, C}$ for $C\in \Nb$. Thus given $\tau>0$ we can find some $C\in \Nb$ such that $\mu((W\cap g^{-1}W)\setminus W_{g, C})<\tau$ for all $g\in T$. Set $W^\dag=W\setminus \bigcup_{g\in T}((W\cap g^{-1}W)\setminus W_{g, C})$, $S_{n+1}=T$, and $\cV_{n+1}=T(W\setminus W^\dag)$.
Provided that $\tau$ is small enough so that $\Upsilon(T)|T|^2\tau\le 1/2$, we then have
\[
\sum_{j=1}^{n+1}\Upsilon(S_j)\mu(\cV_j)\le \frac12 +\Upsilon(S_{n+1})\mu(\cV_{n+1})\le \frac12 +\Upsilon(T)|T|^2\tau\le 1.
\]
Let $g\in T$ and $w_1, w_2\in W$ with $gw_1=w_2$. If $w_1\in W^\dag$ then $w_1\in W_{g, C}$,
in which case $w_1$ and $w_2$ are connected by a path of length at most $C$ in which each edge is an $S_j$-edge with both endpoints in $\cV_j$ for some $1\le j\le n$. If on the other hand $w_1\not\in W^\dag$, then $(w_1, w_2)$ is an $S_{n+1}$-edge with both endpoints in $\cV_{n+1}$.

\subsection{The shrinking property} \label{SS-shrinking}

The shrinking property introduced here will be of technical value in the proofs of
Proposition~\ref{P-normal SC} and Lemma~\ref{L-union SC} (both via Proposition~\ref{P-SC'}) as well as
Lemma~\ref{L-small} and Proposition~\ref{P-product}.
For the purpose
of these applications we provide a characterization of when it holds in Proposition~\ref{P-shrinking}.
Notice that, within the chain of quantification, the position of the sets $S_1$ (acting as the generator
of a graph in which the path in (iii) lives) and $S$ (expressing the ``scale'' 
at which the set $W$ is dense in $X$) is a reversal of what occurs in the definition of property SC.
The proof of Proposition~\ref{P-normal SC} gives an illustration of how this can be leveraged 
in direct conjunction with property SC.

\begin{definition} \label{D-shrinking}
We say that $G$ has the {\it shrinking property} if there is an $S_1\in \overline{\cF}(G)$ such that for every $\varepsilon>0$ there is an $S\in \overline{\cF}(G)$ so that
for every $\delta>0$ there is a $C\in \Nb$ such that given any free p.m.p.\ action
$G\curvearrowright (X, \mu)$ we can find Borel sets $Z\subseteq \cV\subseteq X$ satisfying the following conditions:
\begin{enumerate}
\item $S\cV=X$,
\item $\mu(\cV)\le \varepsilon$ and $\mu(Z)\le \delta$,
\item every point of $\cV$ is connected to some point of $Z$ by an $S_1$-path of length at most $C$ whose
points all belong to $\cV$.
\end{enumerate}
\end{definition}

\begin{lemma} \label{L-Folner for virtually cyclic}
Suppose that $s\in G$ generates an infinite subgroup $G_1$ such that $G/G_1$ is finite. Let $B$ be a subset of $G$
containing $e_G$ such that $G=\bigsqcup_{b\in B}G_1b$.  For each $n\in \Nb$ set $F_n=\{s^{j}: -n\le j\le n\}B$. Then $\{F_n\}_{n\in \Nb}$ is a left F{\o}lner sequence for $G$.
\end{lemma}

\begin{proof}
Since $G_1$ has finite index in $G$ and $G$ is finitely generated, $G_1$ has a subgroup $G_0$ such that $G_0$ is normal and has finite index in $G$. Then $G_0$ is generated by $s^N$ for some $N\in \Nb$.

For each $n\in \Nb$, put $A_n=\{s^{Nj}: -n\le j\le n\}\in \cF(G_0)$. For each $g\in G$, the map $x\mapsto gxg^{-1}$ is an automorphism of $G_0\cong \Zb$ and hence either $gxg^{-1}=x$ for all $x\in G_0$ or $gxg^{-1}=x^{-1}$ for all $x\in G_0$, which implies that $gA_ng^{-1}=A_n$ for every $n\in \Nb$.

Let $K\in \cF(G)$. Then we can find some $M\in \Nb$ such that $KF_NB\subseteq F_{MN}$. For any $n\in \Nb$ satisfying $n\ge N$, writing $n$ as $kN+m$ for some $0\le m\le N-1$ and $k\in \Nb$ we have
\[
KF_n\subseteq K F_NA_{k}B=A_kKF_NB\subseteq A_kF_{MN}=F_{(M+k)N},
\]
and hence $|KF_n|\le |F_{(M+k)N}|=\frac{2(M+k)N+1}{2kN+1}|F_{kN}|\le \frac{2(M+k)N+1}{2kN+1}|F_n|$. As $n\to \infty$ we have $k\to \infty$, whence $\frac{2(M+k)N+1}{2kN+1}\to 1$. Therefore $\{F_n\}_{n\in \Nb}$ is a left F{\o}lner sequence for $G$.
\end{proof}

\begin{lemma} \label{L-increasing virtual cyclic to SC}
Suppose that $G$ is locally virtually cyclic but neither locally finite nor virtually cyclic. Then $G$ has both property SC and the shrinking property.
\end{lemma}

\begin{proof}
Since $G$ is not locally finite, it has a finitely generated infinite subgroup $G_2$. As $G$ is locally virtually cyclic, $G_2$ is virtually cyclic and hence contains a subgroup $G_1$ such that $G_1$ is isomorphic to $\Zb$ and has finite index in $G_2$.

Take an increasing sequence $\{U_n\}_{n\in \Nb}$ of finite subsets of $G$ with union $G$.
Let $n\ge 3$. Denote by $G_n$ the subgroup of $G$ generated by $G_2$ and $U_n$. Then $G_n$ is finitely generated, and hence has a finite-index cyclic subgroup $G_n'$. Since $G_1$ is infinite and $G_n'$ has finite index in $G_n$, the intersection $G_1\cap G_n'$ must be nontrivial and hence have finite index in $G_n'$. Then $G_1\cap G_n'$ has finite index in $G_n$, whence $G_1$ has finite index in $G_n$.

Now we have that $\{G_n\}_{n\in \Nb}$ is an increasing sequence of proper subgroups of $G$ with union $G$, $G_1$ is isomorphic to $\Zb$, and $G_1$ has finite index in $G_n$ for every $n\in \Nb$. Then $[G_n: G_1]\to \infty$ as $n\to \infty$.
Take an $s\in G_1$ generating $G_1$.

Put $S_1=\{s, e_G, s^{-1}\}\in \overline{\cF}(G)$.
Let $\Upsilon$ be a function $\cF(G)\rightarrow [0, \infty)$.
Take $m\in \Nb$ such that $\Upsilon(S_1)\le [G_m:G_1]/6$.
Since $G_1$ has finite index in $G_m$, $G_1$ has a subgroup $G_0$ such that $G_0$ is a finite-index normal subgroup of  $G_m$. Then $G_0$ is generated by $s^N$ for some $N\in \Nb$.

For each $n\in \Nb$, put $A_n=\{s^{Nj}: -n\le j\le n\}\in \cF(G_0)$. For each $g\in G_m$, the map $x\mapsto gxg^{-1}$ is an automorphism of $G_0\cong \Zb$ and hence either $gxg^{-1}=x$ for all $x\in G_0$ or $gxg^{-1}=x^{-1}$ for all $x\in G_0$, which implies that $gA_ng^{-1}=A_n$ for every $n\in \Nb$.

Take a $B\in \cF(G_m)$ containing $e_G$ such that $G_m=\bigsqcup_{b\in B}G_1b$. For each $n\in \Nb$ put $K_n=\{s^j: -n\le j\le n\}\in \cF(G_1)$. Also, put $S=(K_NB)\cup (K_NB)^{-1}\in \overline{\cF}(G_m)$.

Let $T\in \overline{\cF}(G)$.
Take an $M\in \Nb$ bigger than $m$ such that $T\subseteq G_M$. Take a $D\in \cF(G_M)$ containing $e_G$ such that $G_M=\bigsqcup_{d\in D}G_md$.
Put $S_2=D^{-1}D\in \overline{\cF}(G_M)$ and $S_3=T\in \overline{\cF}(G_M)$. Take $\tau>0$ with $\tau|T|\Upsilon(T)\le 1/3$, and take $0<\eta<\min\{1/2, \tau/(4|S|+1)\}$.

For $F\in \cF(G_M)$ write $\partial_TF$ for the set $\{g\in F: Tg\nsubseteq F\}$.
For each $n\in \Nb$ put $F_n=K_nBD\in \cF(G_M)$. From Lemma~\ref{L-Folner for virtually cyclic} we know that $\{F_{nN}\}_{n\in \Nb}$ is an increasing left F{\o}lner sequence for $G_M$. Then we can find some $n_0\in \Nb$ such that $\Upsilon(S_2)/(2n_0N+1)< 1/6$ and  for any $n\ge n_0$ one has $|\partial_TF_{nN}|\le \eta |F_{nN}|$.
By the Ornstein--Weiss quasitower theorem \cite[Theorem 4.44]{KerLi16} there are some $l\in \Nb$ and $n_0<n_1<n_2<\dots <n_l$ in $\Nb$  such that
for any free p.m.p.\ action $G\curvearrowright (X, \mu)$
there are Borel sets $Z_1, \dots, Z_l\subseteq X$ satisfying the following conditions:
\begin{enumerate}
\item for each $1\le j\le l$ there is a set $F_{n_jN, x}\subseteq F_{n_jN}$ depending measurably on $x\in Z_j$
and satisfying $|F_{n_jN, x}|\ge (1-\eta/|T|)|F_{n_jN}|$ for each $x\in Z_j$ so that the sets $F_{n_jN, x}x$ for $x\in Z_j$ are pairwise disjoint,
\item $F_{n_iN}Z_i\cap F_{n_jN}Z_j=\emptyset$ for $i\neq j$,
\item $\mu(\bigcup_{j=1}^lF_{n_jN}Z_j)\ge 1-\eta$.
\end{enumerate}
Put $C=2n_lN+1$.

Now let $G\curvearrowright (X, \mu)$ be a free p.m.p.\ action.
Then we have $Z_1, \dots, Z_l$ as above.
Put
\[
W'=\bigcup_{j=1}^l\bigcup_{x\in Z_j}(K_{n_jN}D\cap (F_{n_jN, x}\setminus \partial_TF_{n_jN, x}))x,
\]
and $W=(X\setminus SW')\cup W'$.
Then $W'$  and $W$ are Borel subsets of $X$. Clearly $SW=X$, verifying condition (ii) in Definition~\ref{D-SC}.

Let $1\le j\le l$.
For each $x\in Z_j$
we have
\[
\partial_TF_{n_jN, x}\subseteq \partial_TF_{n_jN}\cup T^{-1}(F_{n_jN}\setminus F_{n_jN, x})
\]
and hence
\begin{align*}
|K_{n_jN}D\setminus (F_{n_jN, x}\setminus \partial_TF_{n_jN, x})|&\le |F_{n_jN}\setminus (F_{n_jN, x}\setminus \partial_TF_{n_jN, x})|\\
&=|(F_{n_jN}\setminus F_{n_jN, x})\cup \partial_TF_{n_jN, x}|\\
&\le |\partial_TF_{n_jN}\cup T^{-1}(F_{n_jN}\setminus F_{n_jN, x})|\\
&\le |\partial_TF_{n_jN}|+|T|\cdot |F_{n_jN}\setminus F_{n_jN, x}|\\
&\le 2\eta|F_{n_jN}|.
\end{align*}
Set
\[
\cV_1=\bigcup_{j=1}^lK_{n_jN}DZ_j\subseteq X,  \hspace*{4mm} \cV_2=\bigcup_{j=1}^lDZ_j\subseteq X, \hspace*{4mm} \mbox{and} \hspace*{4mm} \cV_3=T(W\setminus W')\subseteq X.
\]
Then
\begin{align*}
\Upsilon(S_1)\mu(\cV_1)
&\le \Upsilon(S_1)\sum_{j=1}^l|K_{n_jN}D|\mu(Z_j)\\
&= \frac{\Upsilon(S_1)}{[G_m: G_1]}\sum_{j=1}^l|K_{n_jN}BD|\mu(Z_j)\\
&=\frac{\Upsilon(S_1)}{[G_m: G_1]}\sum_{j=1}^l|F_{n_jN}|\mu(Z_j)\\
&\le \frac{\Upsilon(S_1)}{[G_m: G_1](1-\eta)}\le  \frac{1}{6(1-\eta)}\le \frac13 ,
\end{align*}
and
\begin{align*}
\Upsilon(S_2)\mu(\cV_2)
&\le \Upsilon(S_2)\sum_{j=1}^l|D|\mu(Z_j)\\
&\le \frac{\Upsilon(S_2)}{|K_{n_0N}|}\sum_{j=1}^l|K_{n_jN}BD|\mu(Z_j)\\
&=\frac{\Upsilon(S_2)}{2n_0N+1}\sum_{j=1}^l|F_{n_jN}|\mu(Z_j)\\
&\le \frac{\Upsilon(S_2)}{(2n_0N+1)(1-\eta)}\le \frac{1}{6(1-\eta)}\le \frac13 .
\end{align*}
For each $1\le j\le l$, since $K_NB\subseteq G_m$ we have
\[
F_{n_jN}=K_{n_jN}BD=A_{n_j-1}K_NBD=K_NBA_{n_j-1}D\subseteq SA_{n_j-1}D\subseteq SK_{n_jN}D.
\]
Thus $S\cV_1\supseteq \bigcup_{j=1}^lF_{n_jN}Z_j$. Therefore
\begin{align*}
\mu(W\setminus W')&\le \mu(X\setminus SW')\\
&\le \mu(X\setminus S\cV_1)+|S|\mu(\cV_1\setminus W')\\
&\le \mu\bigg(X\setminus \bigcup_{j=1}^lF_{n_jN}Z_j\bigg)+|S|\mu\bigg(\bigcup_{j=1}^l\bigcup_{x\in Z_j}(K_{n_jN}D\setminus (F_{n_jN, x}\setminus \partial_TF_{n_jN, x}))x\bigg)\\
&\le \eta+2\eta|S|\sum_{j=1}^l|F_{n_jN}|\mu(Z_j)\\
&\le \eta+\frac{2\eta|S|}{1-\eta}\le \eta+4\eta |S|\le \tau,
\end{align*}
and hence
\[
\Upsilon(S_3)\mu(\cV_3)\le \Upsilon(T)|T|\mu(W\setminus W')\le \tau \Upsilon(T)|T|\le \frac13.
\]
Combining estimates we obtain
\[
\sum_{j=1}^3\Upsilon(S_j)\mu(\cV_j)\le 1,
\]
verifying condition (i) in Definition~\ref{D-SC}.

Next let $g\in T$ and $w_1, w_2\in W$ with $gw_1=w_2$. If $w_1\not\in W'$ or $w_2\not\in W'$, then $(w_1, w_2)$ is an $S_3$-edge with both endpoints in $\cV_3$. Thus we may assume that
$w_1, w_2\in W'$.
For each $i=1, 2$, we have $w_i=h_ix_i$ for some $1\le j_i\le l$, $x_i\in Z_{j_i}$ and $h_i\in  K_{n_{j_i}N}D\cap(F_{n_{j_i}N, x_i}\setminus \partial_TF_{n_{j_i}N, x_i})$. Write $h_i$ as $t_id_i$ for some $t_i\in K_{n_{j_i}N}$ and $d_i\in D$. Then $w_i=t_id_ix_i$ is connected to $d_ix_i$ by some $S_1$-path in $\cV_1$ of length at most $n_lN$.
Since $h_1\in F_{n_{j_1}N, x_1}\setminus \partial_TF_{n_{j_1}N, x_1}$, we have $w_2=gw_1\in F_{n_{j_1}N, x_1}x_1$, and hence $x_1=x_2$. Then $(d_1x_1, d_2x_2)$ is an $S_2$-edge with both endpoints in $\cV_2$. Thus $w_1$ and $w_2$ are connected by a path of length at most $C$ in which each edge is an $S_j$-edge with both endpoints in $\cV_j$ for some $1\le j\le 3$, verifying condition (iii) in Definition~\ref{D-SC}. This proves that $G$ has property SC.

To establish the shrinking property, let $\varepsilon>0$ and consider the function $\Upsilon: \cF(G)\rightarrow [0, \infty)$ which takes the constant value $2/\varepsilon$. Then we have $S$ as above. Take $T=\{e_G\}$. Let $\delta>0$. In the above we can further require that $\eta\le \min\{\varepsilon/2, \delta/2\}$ and $\Upsilon(S_2)/(2n_0N+1)\le \delta/(2\varepsilon)$.  Given a
free p.m.p.\ action $G\curvearrowright (X, \mu)$, we have $\cV_j$ for $1\le j\le 3$  as above. Put $\cV=\cV_1\cup (X\setminus \bigcup_{j=1}^lF_{n_jN}Z_j)$ and $Z=\cV_2\cup (X\setminus \bigcup_{j=1}^lF_{n_jN}Z_j)\subseteq \cV$. From $S\cV_1\supseteq \bigcup_{j=1}^lF_{n_jN}Z_j$ we have $S\cV=X$, verifying condition (i) in Definition~\ref{D-shrinking}. We have
\[
\mu(\cV)\le \mu(\cV_1)+\mu\bigg(X\setminus \bigcup_{j=1}^lF_{n_jN}Z_j\bigg)\le \frac{\varepsilon}{2}
\Upsilon(S_1)\mu(\cV_1)+\eta\le \frac{\varepsilon}{2}+\frac{\varepsilon}{2}=\varepsilon,
\]
and
\[
\mu(Z)\le \mu(\cV_2)+\mu\bigg(X\setminus \bigcup_{j=1}^lF_{n_jN}Z_j\bigg)\le \frac{\varepsilon}{2}
\Upsilon(S_2)\mu(\cV_2)+\eta\le \frac{\varepsilon\Upsilon(S_2)}{2n_0N+1}+\frac{\delta}{2}\le \delta,
\]
verifying condition (ii) in Definition~\ref{D-shrinking}.
Clearly every point of $\cV$ is connected to some point of $Z$ through an $S_1$-path of length at most $n_lN\le C$ whose
points all belong to $\cV$, verifying condition (iii) in Definition~\ref{D-shrinking}. Thus $G$ has the shrinking property.
\end{proof}

For an action $G\curvearrowright X$ and a set $F\in \cF(G)$ we say that a subset $W$ of $X$ is the
{\it base of an $F$-tower} if the sets $tW$ for $t\in F$ are pairwise disjoint.

\begin{lemma} \label{L-measurable maximal}
Let $G\curvearrowright (X, \mu)$ be a free p.m.p.\ action. Let $F\in \overline{\cF}(G)$. Let $Y$ be a Borel subset of $X$ and $E$ a Borel subset of $Y^2$ such that $E\supseteq \Delta_Y=\{(y, y): y\in Y\}$, $E^*=E$ where $E^*:=\{(y, x): (x, y)\in E\}$, and $E$ is contained  in the union of the graphs of $g\in F$. Let $M_1\in \Nb$. Then there are Borel sets $W, \tilde{Y}\subseteq Y$ with $\mu(Y\setminus \tilde{Y})=0$ such that
\begin{enumerate}
\item $W$ is $(E, M_1)$-separated in the sense that no distinct two points of $W$ are connected by an $E$-path of length at most $2M_1$, and

\item every point in $\tilde{Y}$ is connected to some point in $W$ by an $E$-path of length at most $2M_1$.
\end{enumerate}
\end{lemma}

\begin{proof}
Note that if a subset $U$ of $Y$ is the base of an $F^{M_1}$-tower, then $U$ is $(E, M_1)$-separated. By \cite[Proposition A.22]{KerLi16} every Borel subset $U$ of $Y$ with $\mu(U)>0$ contains a Borel subset $U'$ with $\mu(U')>0$ such that $U'$ is the base of an $F^{M_1}$-tower and hence is $(E, M_1)$-separated.

If $\mu(Y)=0$, we may take $W=\tilde{Y}=\emptyset$. Thus we may assume that $\mu(Y)>0$. By the above there are $(E, M_1)$-separated Borel sets $W\subseteq Y$ with $\mu(W)>0$.
 Denote by $\eta_1$ the supremum of $\mu(W)$ over all such $W$. Take one such $W_1$ with $\mu(W_1)\ge \eta_1/2$.

Arguing inductively, suppose that we have chosen Borel sets $W_1, \dots, W_n\subseteq Y$ such that $W_1, \dots, W_n$ are pairwise disjoint and $\bigcup_{j=1}^nW_j$ is $(E, M_1)$-separated.
Denote by $V_n$ the set of points in $Y$ which can be connected to some point in $\bigcup_{j=1}^nW_j$ by an $E$-path of length at most $2M_1$. Then $V_n$ is Borel.
If $\mu(Y\setminus V_n)=0$, we can take $W=\bigcup_{j=1}^nW_j$ and $\tilde{Y}=V_n$. Otherwise, by the first paragraph of the proof  there are $(E, M_1)$-separated Borel sets $Z\subseteq Y\setminus V_n$ with $\mu(Z)>0$. Denote by $\eta_{n+1}$ the supremum of $\mu(Z)$ over all such $Z$. Take an $(E, M_1)$-separated Borel set $W_{n+1}\subseteq Y\setminus V_n$ such that $\mu(W_{n+1})\ge \eta_{n+1}/2$. Then $W_1, \dots, W_{n+1}$ are pairwise disjoint and $\bigcup_{j=1}^{n+1}W_j$ is $(E, M_1)$-separated.

Now we have constructed $\{W_j\}_{j\in \Nb}$. Then $\mu(W_j)\to 0$ as $j\to \infty$ and hence $\eta_j\to 0$ as $j\to \infty$. Set $W=\bigcup_{j\in \Nb}W_j\subseteq Y$. Then $W$ is Borel and $(E, M_1)$-separated. Denote by $\tilde{Y}$ the set of points in $Y$ which can be connected to some point in $W$ by an $E$-path of length at most $2M_1$. Then $\tilde{Y}$ is Borel. If $\mu(Y\setminus \tilde{Y})>0$,  by the first paragraph of the proof there is an $(E, M_1)$-separated Borel set $Z\subseteq Y\setminus \tilde{Y}$ with $\mu(Z)>0$, yielding $0<\mu(Z)\le \eta_n$ for all $n$, which is a contradiction. Therefore $\mu(Y\setminus \tilde{Y})=0$.
\end{proof}

\begin{lemma}\label{L-tree action}
Suppose that $G$ is finitely generated and not virtually cyclic.
Let $A$ be a generating set for $G$ in $\overline{\cF}(G)$.
Then there is a constant $b>0$ such that
given any $r, M\in\Nb$ and free p.m.p.\ action $G\curvearrowright (X, \mu)$
there exist Borel sets $Z\subseteq\cV\subseteq X$
such that $A^{2r}\cV=X$, $\mu(\cV)\le b/r$, $\mu(Z)\le 1/(M+1)$, and every point of $\cV$ is connected to some point of $Z$ by an $A$-path of length at most $2M$ with all points in $\cV$.
\end{lemma}

\begin{proof}
As $G$ is not virtually cyclic, we can find a $c>0$ such that $|A^n| \geq cn^2$ for all $n\in\Nb$
(Corollary~3.5 of \cite{Man12}). Set $b=5/c$.
Let $r, M$, and $G\curvearrowright (X, \mu)$ be as in the lemma statement.

Applying Lemma~\ref{L-measurable maximal} with $F$ taken to be $A$, $Y$ to be $X$, $E$ to be the union of the graphs of $g\in A$,  and $M_1$ to be $r$, we  find a Borel set  $U\subseteq X$ such that $U$ is the base of an $A^r$-tower and $\mu(A^{2r}U)=1$.
Set $X'=\bigcap_{g\in G}g(A^{2r}U)$. Then $X'$ is $G$-invariant and $\mu(X')=1$. Set $U'=U\cap X'$. Then $U'$ is Borel. For each $x\in X'$, we have $x\in A^{2r}w$ for some $w\in U$. Then $w\in A^{2r}x\subseteq X'$ and hence $w\in U'$. This shows that $X'=A^{2r}U'$.
Note also that
\begin{align*}
|A^r| \mu(U')
= \mu(A^rU')
\leq 1
\end{align*}
whence
\begin{align}\label{E-action E2}
\mu(U')\le  \frac{1}{|A^r|} \leq \frac{1}{cr^2}.
\end{align}

Let $w\in U'$. Set $T_w=U'\cap Gw$. Since $X'=A^{2r}U'$, we have $Gw=A^{2r}T_w$. In particular, $T_w$ is infinite. Consider the graph $(T_w, E'_w)$ whose edges are
those pairs of vertices which can be joined by an $A$-path of length at most $4r+1$. We claim that $(T_w, E'_w)$ is connected. It suffices to show that every $v\in T_w$ is connected to $w$ by some path in $(T_w, E'_w)$. Take an $A$-path from $w$ to $v$. For each point $z$ in this path, we may connect $z$ to some $u_z\in T_w$ by an $A$-path $p_z$ of length at most $2r$.
Inserting $p_z$ and the reverse of $p_z$ at $z$, we find an $A$-path
from $w$ to $v$ in which points of $T_w$ appear in every interval of length $4r+1$. Thus $v$ is connected to $w$ by some path in $(T_w, E'_w)$. This proves our claim.

Consider the graph $(U', E')$ whose edges are
those pairs of vertices which can be joined by an $A$-path of length at most $4r+1$.
Note that $E'$ is a Borel subset of $(U')^2$ and is contained in the union of the graphs of $g\in A^{4r+1}$.

Applying Lemma~\ref{L-measurable maximal} with $F=A^{4r+1}$, $Y=U'$, $E=E'$, and $M_1=M$,  we find an $(E', M)$-separated Borel subset $Z'$ of $U'$
such that $\mu(U'\setminus W')=0$, where $W'$ denotes the Borel set of all points in $U'$ which can be connected to some point in $Z'$ by an $E'$-path of length at most $2M$.
For each $0\le j\le 2M$, denote by $W'_j$ the set of points in $W'$ which can be connected to some point in $Z'$ by an $E'$-path of length  $j$, but cannot be connected to any point in $Z'$ by an $E'$-path of length less than $j$. Then the sets $Z'=W'_0, W'_1, W'_2, \dots, W'_{2M}$ form a Borel partition of $W'$.

Denote by $\Theta$ the set of finite sequences in $A$ with length at most $4r+1$.
Let $1\le j\le 2M$. Take a Borel map $f_j: W'_j\rightarrow W'_{j-1}$ such that $(v, f_j(v))\in E'$ for all $v\in W'_j$.
Also take a Borel map $h_j: W'_j\rightarrow \Theta$ such that for any $v\in W'_j$, say $h_j(v)=(g_1, \dots, g_l)$ with $1\le l\le 4r+1$ and $g_k\in A$ for $1\le k\le l$, one has $g_lg_{l-1}\dots g_1v=f_j(v)$. Denote by $\cV'$ the union of $W'$ and the set consisting of the points $g_1v, g_2g_1v, \dots, g_{l-1}\dots g_1v$ for all $1\le j\le 2M$, $v\in W'_j$, and $h_j(v)=(g_1, \dots, g_l)$. Also, denote by $\cE$ the union of $\Delta_{\cV'}$ and the set consisting of
all pairs of the form $(g_kg_{k-1}\dots g_1v, g_{k-1}\dots g_1v)$ or $(g_{k-1}\dots g_1v, g_kg_{k-1}\dots g_1v)$ for $1\le j\le 2M$, $v\in W'_j$, $h_j(v)=(g_1, \dots, g_l)$, and $1\le k\le l$.
Then $\cV'$ is Borel, and
\begin{align*}
\mu(\cV')\le \sum_{j=0}^{2M}(4r+1)\mu(W'_j)=(4r+1)\mu(W')\le 5r\mu(W')=5r\mu(U')\overset{\eqref{E-action E2}}\le 5r\cdot \frac{1}{cr^2} =\frac{b}{r}.
\end{align*}
Note that $\cE\subseteq (\cV')^2$ is Borel and contained in the union of the graphs of $g\in A$,
and that $\cE\supseteq \Delta_{\cV'}$ and $\cE^*=\cE$.
Also note that every point of $\cV'$ is connected to some point of $Z'$ by an $\cE$-path of length at most $2M(4r+1)$.

Let $z\in Z'$. Denote by $V'_z$ the set of all $w\in U'$ which can be connected to $z$ by an $E'$-path of length at most $M$. Since $T_z$ is infinite, $T_z\neq V'_z$.  Because $(T_z, E'_z)$ is connected, we conclude that $|V'_z|\ge M+1$. Since $Z'$ is $(E', M)$-separated, the sets $V'_z$ for $z\in Z'$ are pairwise disjoint, and hence for each $z\in Z'$ the set $V'_z$ is contained in the $\cE$-connected component of $\cV'$ containing $z$.
As every $\cE$-connected component of $\cV'$ contains some $z\in Z'$, it contains $V'_z$ and hence has at least $M+1$ points.

Applying Lemma~\ref{L-measurable maximal} with $F=A$, $Y=\cV'$, $E=\cE$, and $M_1=M$, we find an $(\cE, M)$-separated Borel subset $Z$ of $\cV'$ such that $\mu(\cV'\setminus \cV)=0$, where $\cV$ denotes the Borel set of all points in $\cV'$ which can be connected to some point in $Z$ by an $\cE$-path of length at most $2M$. Then every point in $\cV$ can be connected to some point of $Z$ by an $A$-path of length at most $2M$ with all points in $\cV$. Note that $\mu(\cV)=\mu(\cV')\le b/r$.

Let $z\in Z$. Denote by $V_z$ the set of $w\in \cV'$ which can be connected to $z$ by an $\cE$-path of length at most $M$.
Since the $\cE$-connected component of $\cV'$ containing $z$ has at least $M+1$ points, we have $|V_z|\ge M+1$.
Note that $V_z=U_zz$ for some subset $U_z$ of $A^{M}$. Denote by $\cD$ the set of all subsets of $A^{M}$ with cardinality at least $M+1$. The map $\psi$ from $Z$ to $\cD$ sending $z$ to $U_z$ is Borel.
Since the sets $V_z$ for $z\in Z$ are pairwise disjoint, we get
\begin{align*}
1\ge \mu\bigg(\bigcup_{z\in Z}V_z\bigg)=\mu\bigg(\bigcup_{z\in Z}U_zz\bigg)&=\sum_{D\in \cD}\mu\bigg(\bigcup_{z\in \psi^{-1}(D)}Dz\bigg) \\
&=\sum_{D\in \cD}|D|\mu(\psi^{-1}(D))\\
&\ge (M+1)\sum_{D\in \cD}\mu(\psi^{-1}(D))=(M+1)\mu(Z),
\end{align*}
whence
\[
\mu(Z)\le \frac{1}{M+1}.
\]

Since $\mu(\cV'\setminus \cV)=0$, we have
\[
\mu(A^{2r}\cV)=\mu(A^{2r}\cV')\ge \mu(A^{2r}W')=\mu(A^{2r}U')=\mu(X')=1,
\]
and hence $\mu(A^{2r}\cV)=1$.

Now put $\cV^\flat=\cV\cup (X\setminus A^{2r}\cV)$ and $Z^\flat=Z\cup (X\setminus A^{2r}\cV)$. Then $Z^\flat\subseteq \cV^\flat$, $A^{2r}\cV^\flat=X$, $\mu(\cV^\flat)=\mu(\cV)\le b/r$, $\mu(Z^\flat)=\mu(Z)\le 1/(M+1)$, and every point of $\cV^\flat$ is connected to some point of $Z^\flat$ by an $A$-path of length at most $2M$ with all points in $\cV^\flat$.
\end{proof}

\begin{proposition} \label{P-shrinking}
$G$ has the shrinking property if and only if it is neither locally finite nor virtually cyclic.
\end{proposition}

\begin{proof}
If $G$ is not locally virtually cyclic, then $G$ has  a finitely generated subgroup $G_1$ which is not virtually cyclic. By Lemma~\ref{L-tree action} the group $G_1$ has the shrinking property, and hence  so does $G$.
If $G$ is locally virtually cyclic but neither locally finite nor virtually cyclic, then by Lemma~\ref{L-increasing virtual cyclic to SC} the group $G$ has the shrinking property. This proves the ``if'' part.

To prove the ``only if'' part, we assume that $G$ has the shrinking property and argue by contradiction that $G$ cannot be locally finite or virtually cyclic.

Suppose that $G$ is locally finite. Let $S_1\in \overline{\cF}(G)$ witness the shrinking property for $G$. Denote by $G_1$ the finite subgroup of $G$ generated by $S_1$. Put $\varepsilon=1$.  Then we have $S\in \overline{\cF}(G)$ in Definition~\ref{D-shrinking}.
Take $0<\delta<1/|SG_1|$. Then we have $C$ in Definition~\ref{D-shrinking}.
Let $G\curvearrowright (X, \mu)$ be a free p.m.p.\ action. Then we have Borel sets $Z\subseteq \cV\subseteq X$ satisfying conditions (i)-(iii) in Definition~\ref{D-shrinking}. Note that $\cV\subseteq G_1Z$ and hence $X=S\cV\subseteq SG_1Z$. Thus $\delta\ge \mu(Z)\ge 1/|SG_1|$, which is a contradiction.

Suppose now that $G$ is virtually cyclic but not locally finite. Then $G$ has a finite-index normal subgroup $G_1$ isomorphic to $\Zb$. Take a generator $s$ for $G_1$. For each $n\in \Nb$ put $K_n:=\{s^j: -n\le j\le n\}$. Take a $B\in \cF(G)$ containing $e_G$ such that
$G=\bigsqcup_{b\in B}G_1b$.
Let $S_1\in \overline{\cF}(G)$ witness the shrinking property for $G$.
Take $m\in \Nb$ such that $BS_1\subseteq K_mB$.
Take $0<\varepsilon<1/(8|K_mB|)$. Then we have $S\in \overline{\cF}(G)$ in Definition~\ref{D-shrinking}. Take an $S'\in \overline{\cF}(G_1)$ such that $S'B\supseteq S$. Then $S'\subseteq K_M$ for some $M\in \Nb$. Take $R\in \Nb$ large enough so that  $(2R+2M+1)/(2R+1)\le 2$ and take $0<\delta<1/(2|S'K_RB|)$. Then we have $C\in \Nb$ in Definition~\ref{D-shrinking}. Let $G\curvearrowright (X, \mu)$ be a free p.m.p.\ action. Then we have Borel sets $Z\subseteq \cV\subseteq X$ satisfying conditions (i)-(iii) in Definition~\ref{D-shrinking}.
Put $\cV'=B\cV$, $\cV''=K_m\cV'$, and $Z'=BZ\subseteq \cV'$. Then $X=S\cV\subseteq S'B\cV=S'\cV'$ and hence $X=S'\cV'$.

We claim that every point of $\cV'$ can be connected to some point of $Z'$ by a $K_1$-path whose points all belong to $\cV''$.
Let $w'\in \cV'$. Then $w'=bw$ for some $w\in \cV$ and $b\in B$.
We can find some $1\le l\le C$ and $g_1, \dots, g_l\in S_1$ such that $g_j\dots g_1w\in \cV$ for all $1\le j\le l$ and $g_l\dots g_1w\in Z$.
Put $b_0=b$.
We define $b_j\in B$ and $h_j\in K_m$ for $1\le j\le l$ inductively by $b_{j-1}g_j^{-1}=h_j^{-1}b_j$. For each $1\le j\le l$, we have
\[
b_j^{-1}h_j\dots h_1bw=g_jb_{j-1}^{-1}h_{j-1}\dots h_1bw=\dots=g_jg_{j-1}\dots g_1w\in \cV,
\]
and hence $h_j\dots h_1bw\in b_j\cV\subseteq \cV'$. Furthermore, $h_l\dots h_1bw=b_lg_l\dots g_1w\in b_lZ\subseteq Z'$.
We can connect $h_j\dots h_1bw$ and $h_{j-1}\dots h_1bw$ by a $K_1$-path of length at most $m$ with all points in $\cV''$.
This proves our claim.

Put $U=X\setminus S'K_RZ'$. Let $u\in U$. Since $X=S'\cV'$, we have $gu=w'$ for some $g\in S'\subseteq K_M$ and $w'\in \cV'$. Then $w'\not\in K_RZ'$.
By the claim above we can find a $K_1$-path in $\cV''$ from $w'$ to some point $z'\in Z'$.
It follows that either $s^jw'\in \cV''$ for all $0\le j\le R$ or $s^{-j}w'\in \cV''$ for all $0\le j\le R$. Therefore
\[
\frac{1}{2R+1}\sum_{j=-R-M}^{R+M}\chi_{\cV''}(s^ju)\ge \frac{1}{2R+1}\sum_{j=-R}^R\chi_{\cV''}(s^jw')\ge \frac12 .
\]
Thus
\begin{align*}
\int_U\frac{1}{2R+1}\sum_{j=-R-M}^{R+M}\chi_{\cV''}(s^ju)\, d\mu(u)&\ge \int_U\frac{1}{2}\, d\mu(u)\\
&=\frac12 \mu(U)\\
&\ge \frac12 (1-|S'K_RB|\mu(Z))\\
&\ge \frac12 (1-|S'K_RB|\delta)>\frac14 ,
\end{align*}
while
\begin{align*}
\int_U\frac{1}{2R+1}\sum_{j=-R-M}^{R+M}\chi_{\cV''}(s^ju)\, d\mu(u)&\le \int_X\frac{1}{2R+1}\sum_{j=-R-M}^{R+M}\chi_{\cV''}(s^ju)\, d\mu(u)\\
&=\frac{2R+2M+1}{2R+1}\cdot \mu(\cV'')\\
&\le 2|K_mB|\mu(\cV)\\
&\le 2\varepsilon |K_mB|<\frac14 ,
\end{align*}
a contradiction.
This proves the ``only if'' part.
\end{proof}

\subsection{Variants of property SC} \label{SS-SC'}

For the purposes of Section~\ref{SS-without SC} it will be convenient to formulate the
following variations on property SC. Proposition~\ref{P-SC'} establishes relationships between
these two properties, property SC itself, and the shrinking property.

\begin{definition} \label{D-SC'}
Let $\fC$ be a class of free p.m.p.\ actions of a fixed $G$.
We say that $\fC$ has {\it property SC\primespace$'$} if for any function $\Upsilon: \cF(G)\rightarrow [0, \infty)$ there is some $\eta>0$ such that for any $T\in \overline{\cF}(G)$ there are $C, n\in \Nb$, and $S_1, \dots, S_n\in \overline{\cF}(G)$ so that for any $G\curvearrowright (X, \mu)$ in $\fC$ and any Borel sets $W_1, W_2\subseteq X$ with $\mu(W_1), \mu(W_2)\le \eta$, there are Borel sets $\cV_1, \dots, \cV_n\subseteq X$ satisfying the following conditions:
\begin{enumerate}
\item $\sum_{j=1}^n\Upsilon(S_j)\mu(\cV_j)\le 1$,
\item if $w_1\in W_1$ and $w_2\in W_2$ satisfy $gw_1=w_2$ for some $g\in T$ then $w_1$ and $w_2$ are connected by a path of length at most $C$ in which each edge is an $S_j$-edge with both endpoints in $\cV_j$ for some $1\le j\le n$.
\end{enumerate}
We say that $\fC$  has {\it property SC\primespace$''$} if the above conditions hold without the bound $C$.
\end{definition}

\begin{lemma} \label{L-locally finite not SC''}
Suppose that $G$ is locally finite. Then no free p.m.p.\ action
$G\curvearrowright (X,\mu )$ has property SC\primespace$''$.
\end{lemma}

\begin{proof}
Let $G\curvearrowright (X,\mu )$ be a free p.m.p.\ action.
Take a strictly increasing sequence $\{G_k\}_{k\in \Nb}$ of finite subgroups of $G$ such that $G=\bigcup_{k\in \Nb}G_k$.  For each $F\in \cF(G)$, denote by $\Phi(F)$ the smallest $k\in \Nb$ satisfying $F\subseteq G_k$. Define $\Upsilon: \cF(G)\rightarrow [0, \infty)$ by
$\Upsilon(F)=2|G_{\Phi(F)}|$.

Suppose that $G\curvearrowright (X, \mu)$ has property SC$''$.  Then there is an $\eta>0$  satisfying the conditions in Definition~\ref{D-SC'}.
Take an $m\in \Nb$ with $1/|G_m|<\eta$.
Put  $T=G_{m+1}\in \overline{\cF}(G)$. Then there are  $n\in \Nb$ and $S_1, \dots, S_n\in \overline{\cF}(G)$ satisfying the conditions in Definition~\ref{D-SC'}. We may assume that there is a $k$ such that
$S_1, \dots, S_k\subseteq G_m$ and none of $S_{k+1}, \dots, S_n$ is contained in $G_m$.

As $G_{m+1}$ is finite we can find a Borel set $Y\subseteq X$ such that the sets $gY$ for $g\in G_{m+1}$ form a partition of $X$ (see Example~6.1 and Proposition~6.4 in \cite{KecMil04}). Choose an $h\in G_{m+1}\setminus G_m$. Put $W_1=Y$ and $W_2=hY$, and note that $\mu(W_1)=\mu(W_2)=1/|G_{m+1}|<\eta$. Then there are Borel sets $\cV_1, \dots, \cV_n\subseteq X$ satisfying the conditions in Definition~\ref{D-SC'}.

Let $y\in Y$. By condition (ii) in Definition~\ref{D-SC'}, $y$ and $hy$ are connected by a path each of whose
edges is an $S_j$-edge with both endpoints in $\cV_j$ for some $1\le j\le n$. Since $h\not\in G_m$, this implies that $G_m y\cap \bigcup_{j=k+1}^n\cV_j\neq \emptyset$. We infer that $\mu(\bigcup_{j=k+1}^n\cV_j)\ge \mu(Y)=1/|G_{m+1}|$ and hence
\begin{align*}
1\ge \sum_{j=1}^n\Upsilon(S_j)\mu(\cV_j)&\ge \sum_{j=k+1}^n\Upsilon(S_j)\mu(\cV_j)\\
&\ge 2|G_{m+1}|\sum_{j=k+1}^n\mu(\cV_j)\ge 2|G_{m+1}|\mu\bigg(\bigcup_{j=k+1}^n\cV_j\bigg)\ge 2,
\end{align*}
a contradiction. We conclude that $G\curvearrowright (X, \mu)$ does not have property SC$''$.
\end{proof}

\begin{lemma} \label{L-virtually cyclic not SC''}
Suppose that $G$ is virtually cyclic. Then no free p.m.p.\ action
$G\curvearrowright (X,\mu )$ has property SC\primespace$''$.
\end{lemma}

\begin{proof}
By assumption $G$ has a subgroup of finite index isomorphic to $\Zb$. Then $G$ has a normal subgroup $G_1$ of finite index isomorphic to $\Zb$. Take a generator $s$ for $G_1$. Take a finite subset $B$ of $G$ containing $e_G$ such that $G$ is the disjoint union of the sets $hG_1$ for $h\in B$.
For each $m\in \Nb$ put $K_m=\{s^k: -m\le k\le m\}$.
For each $F\in \cF(G)$, denote by $\Phi(F)$ the smallest $m\in \Nb$ satisfying $F\subseteq BK_m$.
Put $M=\Phi(BB)$.
Define  $\Upsilon: \cF(G)\rightarrow [0, \infty)$ by
$\Upsilon(F)=3|K_M|\cdot |B|\cdot |BK_{\Phi(F)}|$.

Let $G\curvearrowright (X, \mu)$ be a free p.m.p.\ action
and suppose that it has property SC$''$.  Then there is an $\eta>0$  satisfying the conditions in Definition~\ref{D-SC'}.
Take an $m\in \Nb$ with $1/(m+1)<\eta$.
Put  $T=\{s^m, e_G, s^{-m}\}\in \overline{\cF}(G)$. Then there are  $n\in \Nb$ and $S_1, \dots, S_n\in \overline{\cF}(G)$ satisfying the conditions in Definition~\ref{D-SC'}.

Applying the Rokhlin lemma \cite[Lemma 4.77]{KerLi16} to $G_1\curvearrowright (X, \mu)$, we find a Borel set $Y\subseteq X$ such that the sets $s^kY$ for $k=0, 1, \dots, m$ are pairwise disjoint and $\mu(\bigcup_{k=0}^{m}s^kY)>1/2$. Put $W_1=Y$ and $W_2=s^{m}Y$. We have $\mu(W_1)=\mu(W_2)=\mu(Y)\le 1/(m+1)<\eta$.
Then there are Borel sets $\cV_1, \dots, \cV_n\subseteq X$ satisfying the conditions in Definition~\ref{D-SC'}.

Put $\cV=\bigcup_{j=1}^nBK_{\Phi(S_j)}\cV_j$. If $(v_1, v_2)$ is an $S_j$-edge with both $v_1$ and $v_2$ in $\cV_j$ for some $1\le j\le n$, then
$v_1$ and $v_2$ are connected by a $(B\cup K_1)$-path in $\cV$. From condition (ii) in Definition~\ref{D-SC'} we conclude that for each $y\in Y$ the points
$y$ and $s^{m}y$ are connected by a $(B\cup K_1)$-path  in $\cV$. Put $\cV'=B^{-1}\cV$ and $\cV''=K_M\cV'$.

Let $x\in \cV'$, $h\in B$, and $l\in \Nb$ be such that $s^khx\in \cV$ for all $0\le k\le l$.
Then either $h^{-1}s^ih=s^i$ for all $i\in \Zb$ or
$h^{-1}s^ih=s^{-i}$ for all $i\in \Zb$. If $h^{-1}s^ih=s^i$ for all $i\in \Zb$, then $x$ is connected to $s^lhx$ by the path
$x, sx, \dots, s^lx, hs^lx$, which is the $K_1$-path $x, sx, \dots, s^lx$ in $\cV'$ followed by the $B$-edge $(s^lx, hs^lx)$ with $s^lx\in \cV'$ and $hs^lx\in \cV$,
since $s^kx=h^{-1}s^khx\in \cV'$ for all $0\le k\le l$. If $h^{-1}s^ih=s^{-i}$ for all $i\in \Zb$, then $x$ is connected to $s^lhx$ by the path
$x, s^{-1}x, \dots, s^{-l}x, hs^{-l}x$, which is the $K_1$-path $x, s^{-1}x, \dots, s^{-l}x$ in $\cV'$ followed by the $B$-edge $(s^{-l}x, hs^{-l}x)$ with $s^{-l}x\in \cV'$ and $hs^{-l}x\in \cV$,
since $s^{-k}x=h^{-1}s^khx\in \cV'$ for all $0\le k\le l$.

Similarly, if $x\in \cV'$, $h\in B$, and $l\in \Nb$ are such that $s^{-k}hx\in \cV$ for all $0\le k\le l$ then $x$ is connected to $s^{-l}hx$ by a path which is a $K_1$-path in $\cV'$ followed by a $B$-edge of the form $(z, s^{-l}hx)$ with $z\in \cV'$ and $s^{-l}hx\in \cV$.

Let $x\in \cV'$ and $h_1, h_2\in B$ be such that $h_2h_1x, h_1x\in \cV$. Then $h_2h_1=hs^k$ for some $h\in B$ and $-M\le k\le M$. When $k\ge 0$, $x$ is connected to $h_2h_1x$ by the path $x, sx, \dots, s^kx, hs^kx$, which is the $K_1$-path $x, sx, \dots, s^kx$ in $\cV''$ followed by the $B$-edge $(s^kx, hs^kx)$ with $s^kx\in \cV'$ and $hs^kx\in \cV$.
When $k\le 0$, $x$ is connected to $h_2h_1x$ by the path $x, s^{-1}x, \dots, s^kx, hs^kx$, which is the $K_1$-path $x, s^{-1}x, \dots, s^kx$ in $\cV''$ followed by the $B$-edge $(s^kx, hs^kx)$ with $s^kx\in \cV'$ and $hs^kx\in \cV$.

From the above three paragraphs, we conclude that if $x\in \cV$ is connected to $y\in \cV$ by a $(B\cup K_1)$-path in $\cV$, then $x$ is connected to $y$ by a path in $\cV''$ that is a $K_1$-path followed by a $B$-edge.

For any $y\in Y$, since $y$ and $s^{m}y$ are connected by a $(B\cup K_1)$-path in $\cV$ we infer that $y$ is connected to $s^{m}y$ by a $K_1$-path in $\cV''$. Therefore $s^ky\in \cV''$ for all $0\le k\le m$. Consequently, $\mu(\cV'')\ge \mu(\bigcup_{k=0}^{m}s^kY)\ge 1/2$. On the other hand, we have
\begin{align*}
\mu(\cV'')\le |K_M|\cdot |B|\mu(\cV)\le  \sum_{j=1}^n|K_M|\cdot |B|\cdot |BK_{\Phi(S_j)}|\mu(\cV_j)=\frac{1}{3}\sum_{j=1}^n\Upsilon(S_j)\mu(\cV_j)\le \frac{1}{3},
\end{align*}
a contradiction.
We conclude that $G\curvearrowright (X, \mu)$ does not have property SC$''$.
\end{proof}

\begin{proposition} \label{P-SC'}
For a fixed $G$, let $\fC$ be a nonempty class of free p.m.p.\ actions $G\curvearrowright (X,\mu)$.
Consider the following conditions:
\begin{enumerate}
\item $\fC$ has property SC,
\item $\fC$ has property SC\primespace$'$,
\item $\fC$ has property SC\primespace$''$,
\item $G$ has the shrinking property.
\end{enumerate}
Then (i)$\Leftrightarrow$(ii)$\Rightarrow$(iii)$\Rightarrow$(iv). Moreover, if $\fC$ consists of
a single action or of all free p.m.p.\ actions of $G$ then (i)$\Leftrightarrow$(ii)$\Leftrightarrow$(iii)$\Rightarrow$(iv).
\end{proposition}

\begin{proof}
(i)$\Rightarrow$(ii). Let $\Upsilon$ be a function $\cF(G)\rightarrow [0, \infty)$.
Then we have an $S\in \overline{\cF}(G)$ witnessing property SC for $\fC$ with respect to the function $2\Upsilon$.
Take an $\eta>0$ with $2\eta|S|\Upsilon(S)<1/2$. Let $T_1\in \overline{\cF}(G)$. Then we have $C, n, S_1, \dots, S_n$ given by Definition~\ref{D-SC} for $T:=ST_1S\in \overline{\cF}(G)$. Put $S_{n+1}=S\in \overline{\cF}(G)$. Let $G\curvearrowright (X, \mu)$ be an action in $\fC$. Then we have $W, \cV_1, \dots, \cV_n$ satisfying conditions (i)-(iii) in Definition~\ref{D-SC}.
Let $W_1, W_2$ be Borel subsets of $X$ satisfying $\mu(W_1), \mu(W_2)\le \eta$.
Put
\[
\cV_{n+1}=(S_{n+1}(W_1\cup W_2)\cap W)\cup (W_1\cup W_2)\subseteq S(W_1\cup W_2).
\]
Then
\[
\mu(\cV_{n+1})\le |S|(\mu(W_1)+\mu(W_2))\le 2\eta |S|
\]
and hence
\[
\sum_{j=1}^{n+1}\Upsilon(S_j)\mu(\cV_j)\le \frac12 +\Upsilon(S)\mu(\cV_{n+1})\le \frac12 +2\eta|S|\Upsilon(S)\le \frac12
+ \frac12 =1,
\]
verifying condition (i) in Definition~\ref{D-SC'}.
Let $g\in T_1$ and $w_1\in W_1$, $w_2\in W_2$ with $gw_1=w_2$.
For $i=1, 2$ take an $s_i\in S$ such that $s_iw_i\in W$. Then $(w_i, s_iw_i)$ is an $S_{n+1}$-edge with both endpoints in $\cV_{n+1}$. Note that $(s_2gs_1^{-1})(s_1w_1)=s_2w_2$ and $s_2gs_1^{-1}\in T$. Thus $s_1w_1$ and $s_2w_2$ are connected by a path of length at most $C$ in which each edge is an $S_j$-edge with both endpoints in $\cV_j$ for some $1\le j\le n$. Then $w_1$ and $w_2$ are connected by a path of length at most $C+2$ in which each edge is an $S_j$-edge with both endpoints in $\cV_j$ for some $1\le j\le n+1$, verifying condition (ii) in Definition~\ref{D-SC'}. Thus $\fC$ has property SC$'$.

(ii)$\Rightarrow$(iii). Trivial.

(iii)$\Rightarrow$(iv). This follows from Lemmas~\ref{L-locally finite not SC''} and \ref{L-virtually cyclic not SC''} and Proposition~\ref{P-shrinking}.

(ii)$\Rightarrow$(i). Since (ii)$\Rightarrow$(iii)$\Rightarrow$(iv), $G$ has the shrinking property.
Let $S^\sharp\in \overline{\cF}(G)$ witness the shrinking property for $G$. Let $\Upsilon$ be a function $\cF(G)\rightarrow [0, \infty)$. Take an $\eta$ as given by property SC$'$ for $\fC$ with respect to the function $2\Upsilon$.
Take an $\varepsilon^\sharp>0$ with $\varepsilon^\sharp\Upsilon(S^\sharp)\le 1/2$. By the shrinking property for $G$ there is an $S\in \overline{\cF}(G)$ so that
for any $\delta>0$ there is a $C^\sharp\in \Nb$ such that given any free p.m.p.\ action $G\curvearrowright (X,\mu )$ we can find Borel sets $Z\subseteq \cV\subseteq X$ satisfying the following conditions:
\begin{enumerate}
\item[(a)] $S\cV=X$,
\item[(b)] $\mu(\cV)\le \varepsilon^\sharp$ and $\mu(Z)\le \delta$,
\item[(c)] every point of $\cV$ is connected to some point of $Z$ by an $S^\sharp$-path of length at most $C^\sharp$ with all points in $\cV$.
\end{enumerate}
Put $\delta=\eta$. Then we have $C^\sharp$ as above.
Let $T\in \overline{\cF}(G)$. Then we have $C, n, S_1, \dots, S_n$ as given by property SC$'$ for $T_1:=(S^\sharp)^{C^\sharp}T(S^\sharp)^{C^\sharp}\in \overline{\cF}(G)$. Put $S_{n+1}=S^\sharp\in \overline{\cF}(G)$. Now let $G\curvearrowright (X, \mu)$ be an action in $\fC$.
Then we have $Z, \cV$ satisfying the above conditions (a)-(c). Put $W_1=W_2=Z$. Then $\mu(W_1), \mu(W_2)\le \delta=\eta$. Thus we have $\cV_1, \dots, \cV_n$ satisfying the conditions in Definition~\ref{D-SC'}. Put $W=\cV$ and $\cV_{n+1}=\cV\subseteq X$. Then
\[
\sum_{j=1}^{n+1}\Upsilon(S_j)\mu(\cV_j)\le \frac12 +\Upsilon(S^\sharp)\mu(\cV)\le \frac12 +\Upsilon(S^\sharp)\varepsilon^\sharp\le \frac12 +\frac12 =1,
\]
verifying condition (i) in Definition~\ref{D-SC}. Clearly  $SW=S\cV=X$, verifying condition (ii) in Definition~\ref{D-SC}.  Let $g\in T$ and $w_1, w_2\in W$ with $gw_1=w_2$. For $i=1, 2$, we can connect $w_i$ to a point $z_i\in Z$ by an $S_{n+1}$-path of length at most $C^\sharp$ with all points in $\cV=\cV_{n+1}$. Then $z_i=s_iw_i$ for some $s_i\in (S^\sharp)^{C^\sharp}$. Note that $(s_2gs_1^{-1})z_1=z_2$ and $s_2gs_1^{-1}\in (S^\sharp)^{C^\sharp}T(S^\sharp)^{C^\sharp}=T_1$. Since $z_i\in W_i$, we can connect $z_1$ and $z_2$ by a path of length at most $C$ in which each edge is an $S_j$-edge with both endpoints in $\cV_j$ for some $1\le j\le n$. Then $w_1$ and $w_2$ are connected by a path of length at most $C+2C^\sharp$ in which each edge is an $S_j$-edge with both endpoints in $\cV_j$ for some $1\le j\le n+1$, verifying condition (iii) in Definition~\ref{D-SC}. Thus $\fC$ has property SC.

Now assume that $\fC$ consists of a single action or of all free p.m.p.\ actions of $G$. We just need to verify (iii)$\Rightarrow$(i). The argument in the above paragraph shows that $\fC$ satisfies the definition of
property SC without the bound $C$. From Remark~\ref{R-SC no bound} we conclude that $\fC$ has property SC.
\end{proof}

\subsection{Normal subgroups and property SC}\label{SS-normal}

\begin{proposition} \label{P-normal SC}
Suppose that $G$ has a normal subgroup $G^\flat$ with property SC.
Then $G$ has property SC.
\end{proposition}

\begin{proof}
By Proposition~\ref{P-SC'} the group $G^\flat$ has the shrinking property. Let $S_1\in \overline{\cF}(G^\flat)$ witness the shrinking property for $G^\flat$.

Let $\Upsilon$ be a function $\cF(G)\rightarrow [0, \infty)$.
Take $0<\varepsilon<1/(3\Upsilon(S_1))$. Then there is an $S\in \overline{\cF}(G^\flat)$ so that
for any $\delta>0$ there is a $C_1\in \Nb$ such that given any free p.m.p.\ action
$G^\flat \curvearrowright (X, \mu)$ we can find Borel sets $Z\subseteq \cV_1\subseteq X$ satisfying the following conditions:
\begin{enumerate}
\item $S\cV_1=X$,
\item $\mu(\cV_1)\le \varepsilon$ and $\mu(Z)\le \delta$,
\item every point of $\cV_1$ is connected to some point of $Z$ by an $S_1$-path of length at most $C_1$ with all points in $\cV_1$.
\end{enumerate}

Using the function $3\Upsilon$ in the definition of property SC for $G^\flat$,
we find an $S^\flat\in \overline{\cF}(G^\flat)$  such that for any $T^\flat\in \overline{\cF}(G^\flat)$
there are $C^\flat, n^\flat\in \Nb$, and  $S^\flat_1, \dots, S^\flat_{n^\flat}\in \overline{\cF}(G^\flat)$ such that for any free p.m.p.\ action
$G \curvearrowright (X, \mu)$ there are Borel subsets $W^\flat$ and $\cV^\flat_k$ of $X$ for $1\le k\le n^\flat$   satisfying the following conditions:
\begin{enumerate}
\item[(iv)] $3\sum_{k=1}^{n^\flat} \Upsilon(S^\flat_{k})\mu(\cV^\flat_{k})\le 1$,
\item[(v)] $S^\flat W^\flat=X$,
\item[(vi)] if  $w_1, w_2\in W^\flat$ satisfy $gw_1=w_2$ for some $g\in T^\flat$ then $w_1$ and $w_2$ are connected by a path of length at most $C^\flat$ in which each edge is an $S^\flat_k$-edge with both endpoints in $\cV^\flat_k$ for some $1\le k\le n^\flat$.
\end{enumerate}

Let $T\in \overline{\cF}(G)$. Set
\[
S_2=S^\flat T\cup (S^\flat T)^{-1}\in \overline{\cF}(G),
\]
and take $0<\delta<1/(3\Upsilon(S_2)|S_2|)$.
Then we have $C_1$ as above.
Set $T^\flat =\bigcup_{g\in T}(S^\flat S_1^{C_1}gS_1^{C_1}g^{-1}S^\flat\cup S^\flat gS_1^{C_1}g^{-1}S_1^{C_1}S^\flat)\in \overline{\cF}(G^\flat)$.
Then we have $C^\flat$, $n^\flat$, and $S^\flat_{k}$ for $1\le k\le n^\flat$ as above. Set
\[
C=2C_1+2+C^\flat\in \Nb.
\]

Now let $G \curvearrowright (X, \mu)$ be a free p.m.p.\ action.
Then we have Borel $Z\subseteq \cV_1\subseteq X$ satisfying conditions (i)-(iii) above.
Note that
\[
\Upsilon(S_1)\mu(\cV_1)\le \Upsilon(S_1)\varepsilon< \frac13 .
\]
We also have Borel sets $W^\flat$ and $\cV^\flat_{k}$ for $1\le k\le n^\flat$ as above, and
\[
\sum_{k=1}^{n^\flat}\Upsilon(S^\flat_{k})\mu(\cV^\flat_{k})\le \frac13 .
\]
Set $\cV_2=S_2Z$. Then
\[
\Upsilon(S_2)\mu(\cV_2)\le \Upsilon(S_2) |S_2|\mu(Z)\le \Upsilon(S_2) |S_2|\delta< \frac13.
\]
Combining all of these bounds we obtain
\[
\Upsilon(S_1)\mu(\cV_1)+\Upsilon(S_2)\mu(\cV_2)+\sum_{k=1}^{n^\flat}\Upsilon(S^\flat_{k})\mu(\cV^\flat_{k})\le 1,
\]
which verifies condition (i) in Definition~\ref{D-SC}.

Set $W=\cV_1$. Then $SW=S\cV_1=X$,
verifying condition (ii) in Definition~\ref{D-SC}.

Let $g\in T$ and $w_1, w_2\in W$ with $gw_1=w_2$.
For $i=1, 2$, we can connect
 $w_i$  to some $z_i\in Z$  by an $S_1$-path of length at most $C_1$ with all points in $\cV_1$.
 Then $w_i=t_iz_i$ for some $t_i\in S_1^{C_1}$.
Using the fact that $S^\flat W^\flat=X$
we have $gz_1=a_1u_1$ for some $u_1\in W^\flat$ and $a_1\in S^\flat$, and $z_2= a_2 u_2$ for some $u_2\in W^\flat$ and $a_2\in S^\flat$.
Note that $a_1^{-1}g$ and $a_2^{-1}$ are both in $S_2$.
Since $(a_1^{-1}g)z_1=u_1$, the pair $(z_1, u_1)$ is an $S_2$-edge with both endpoints in $\cV_2$. Also, since $a_2^{-1}z_2=u_2$, the pair $(z_2, u_2)$ is an $S_2$-edge with both endpoints in $\cV_2$.
Note that
\begin{align*}
(a_2^{-1}t_2^{-1}gt_1g^{-1}a_1) u_1&=a_2^{-1}t_2^{-1}gt_1z_1=a_2^{-1}t_2^{-1}g w_1=a_2^{-1}t_2^{-1}w_2=a_2^{-1}z_2=u_2.
\end{align*}
Since $a_2^{-1}t_2^{-1}gt_1g^{-1}a_1\in S^\flat S_1^{C_1}gS_1^{C_1}g^{-1}S^\flat\subseteq T^\flat$, this means that $u_2\in T^\flat u_1$. Then $u_1$ and $u_2$ are connected by a path of length at most $C^\flat$ in which each edge is an $S^\flat_k$-edge with both endpoints in $\cV^\flat_k$ for some $1\le k\le n^\flat$.
Therefore $w_1$ and $w_2$ are connected by a path of length at most $2C_1+2+C^\flat=C$ in which each edge is either an $S_j$-edge with both endpoints in $\cV_j$ for some $1\le j\le 2$ or an $S^\flat_k$-edge with both endpoints in $\cV^\flat_k$ for some $1\le k\le n^\flat$,
which verifies condition (iii) in Definition~\ref{D-SC}.
\end{proof}

In preparation for the next section we record one application of Proposition~\ref{P-normal SC}.
For this we need the following construction.

Let $G^\flat$ be a finite-index subgroup of $G$. Take a $B\in \cF(G)$ such that $G=\bigsqcup_{b\in B}bG^\flat$. Let $G^\flat \curvearrowright (X, \mu)$ be a p.m.p.\ action. Put $Y=\bigsqcup_{b\in B}bX$ and write $\nu$ for the probability measure on $Y$ which for each $b\in B$
is $\frac{1}{[G: G^\flat]}\mu$ on $bX$ under the natural identification of $bX$ and $X$. For every $g\in G$ and $b\in  B$ we have $gb=b'h$ for unique $b'\in B$ and $h\in G^\flat$, using which we set
\[
g(bx)=b'(hx)
\]
for all $x\in X$.
It is easily checked that this defines a p.m.p.\ action $G\curvearrowright (Y, \nu)$.
Furthermore, if $G^\flat \curvearrowright (X, \mu)$ is free
then so is $G\curvearrowright (Y, \nu)$.

\begin{proposition} \label{P-finite index}
Let $G^\flat$ be a finite-index subgroup of $G$. Then $G$ has property SC if and only
if $G^\flat$ does.
\end{proposition}

\begin{proof}
Assume that $G$ has property SC. Take $B\in \cF(G)$ as above containing $e_G$. Let $\Upsilon^\flat$ be a function $\cF(G^\flat)\rightarrow [0, \infty)$.
For each $F\in \cF(G)$, denote by $\varphi(F)$ the smallest element of $\cF(G^\flat)$ satisfying $F\subseteq B\varphi(F)$, and put
\[
\Upsilon(F)=[G: G^\flat]\Upsilon^\flat(\varphi(FB)\cup (\varphi(FB))^{-1})\ge 0.
\]
Then we have $S\in \overline{\cF}(G)$ witnessing property SC for $G$. Put $S^\flat=\varphi(SB)\cup (\varphi(SB))^{-1}\in \overline{\cF}(G^\flat)$. Let $T^\flat\in \overline{\cF}(G^\flat)$. Put $T=BT^\flat B^{-1}\in \overline{\cF}(G)$.
Then there are $C, n\in \Nb$ and $S_1, \dots, S_n\in \overline{\cF}(G)$ satisfying the conditions in Definition~\ref{D-SC}. Put $C^\flat=C$, $n^\flat=n$, and $S^\flat_j=\varphi(S_jB)\cup (\varphi(S_jB))^{-1}\in \overline{\cF}(G^\flat)$. Let
$G^\flat \curvearrowright (X, \mu)$ be a free p.m.p.\ action.
Define $(Y, \nu)$ as above. Then we have the free action $G\curvearrowright (Y, \nu)$ as above.
By property SC there exist Borel subsets $W$ and $\cV_j$ of $Y$ for $1\le j\le n$ satisfying the following conditions:
\begin{enumerate}
\item $\sum_{j=1}^n\Upsilon(S_j)\nu(\cV_j)\le 1$,
\item $SW=Y$,
\item if $w_1, w_2\in W$ satisfy $gw_1=w_2$ for some $g\in T$ then $w_1$ and $w_2$ are connected by a path of length at most $C$ in which each edge is an $S_j$-edge with both endpoints in $\cV_j$ for some $1\le j\le n$.
\end{enumerate}
We can write $W$ as $\bigsqcup_{b\in B}bW_b$ for some Borel sets $W_b\subseteq X$. Put $W^\flat=\bigcup_{b\in B} W_b\subseteq X$. Similarly, for each $1\le j\le n$ we can write $\cV_j$ as
$\bigsqcup_{b\in B}b\cV_{j, b}$ for some Borel sets $\cV_{j, b}\subseteq X$, and we put $\cV^\flat_j=\bigcup_{b\in B}\cV_{j, b}\subseteq X$.
We then have
\begin{align*}
\sum_{j=1}^{n^\flat}\Upsilon^\flat(S_j^\flat)\mu(\cV_j^\flat)&=\frac{1}{[G: G^\flat]}\sum_{j=1}^n\Upsilon(S_j)\mu(\cV_j^\flat)\le \sum_{j=1}^n\Upsilon(S_j)\nu(\cV_j)\le 1,
\end{align*}
verifying condition (i) in Definition~\ref{D-SC}.
Clearly $S^\flat W^\flat=X$, verifying condition (ii) in Definition~\ref{D-SC}.

Let $g\in T^\flat$ and $w_1, w_2\in W^\flat$ with $gw_1=w_2$. Then there are $b_1, b_2\in B$ such that $b_1w_1, b_2w_2\in W$. Note that $(b_2gb_1^{-1})b_1w_1=b_2w_2$, and $b_2gb_1^{-1}\in T$. Then there are $1\le l\le C$ and $1\le j_1, \dots, j_l\le n$, and $g_i\in S_{j_i}$ for $1\le i\le l$ such that $g_{i-1}\dots g_1b_1w_1, g_i\dots g_1b_1w_1\in \cV_{j_i}$ for all $1\le i\le l$, and $g_l\dots g_1b_1w_1=b_2w_2$. Recursively define $h_i\in G^\flat$ and $d_i\in B$ for $i=1, \dots, l$ by $g_id_{i-1}=d_ih_i$ and $d_{0}=b_1$. Then $h_i\in S_{j_i}^\flat$ and $g_i\dots g_1b_1w_1=d_ih_i\dots h_1w_1$ for all $1\le i\le l$.
Thus $h_{i-1}\dots h_1w_1, h_i\dots h_1w_1\in \cV^\flat_{j_i}$ for all $1\le j\le l$. Also, from $b_2w_2=g_l\dots g_1b_1w_1=d_lh_l\dots h_1w_1$
we get $w_2=h_l\dots h_1w_1$. Thus $w_1$ and $w_2$ are connected by a path of length at most $C^\flat$ in which each edge is an $S_j^\flat$-edge with both endpoints in $\cV_j^\flat$ for some $1\le j\le n^\flat$, verifying condition (iii) in Definition~\ref{D-SC}. Therefore $G^\flat$ has property SC. This proves the ``only if'' part.

Suppose now that $G^\flat$ has property SC.
Since $G^\flat$ has finite index in $G$, we can find a finite-index normal subgroup $G'$ of $G$ such that $G'\subseteq G^\flat$. By the ``only if'' part, $G'$ has property SC. By Proposition~\ref{P-normal SC}, the group $G$ has property SC. This establishes the ``if'' part.
\end{proof}

\subsection{Groups without property SC}\label{SS-without SC}

\begin{lemma} \label{L-free product}
Let $\Gamma$ be a (not necessarily infinite) countable group.
Let $\fC_{G*\Gamma}$ be a class of free p.m.p.\ actions of $G*\Gamma$.
Denote by $\fC_G$ the class of restriction actions $G\curvearrowright (X, \mu)$ for $G*\Gamma \curvearrowright (X, \mu)$ ranging over the actions in $\fC_{G*\Gamma}$.
Suppose that $\fC_{G*\Gamma}$ has property SC\primespace$'$.
Then $\fC_G$ has property SC\primespace$'$.
\end{lemma}

\begin{proof}
We may assume that $\Gamma$ is nontrivial.
For any $A\in \overline{\cF}(G)$, $B\in \overline{\cF}(\Gamma)$, and $k\in \Nb$ we denote by $A*_kB$ the subset of $G*\Gamma$ consisting of
all elements of the form $a_kb_ka_{k-1}b_{k-1}\dots a_1b_1$ for $a_1, \dots, a_k\in A$ and $b_1, \dots, b_k\in B$.
For each $F\in \cF(G*\Gamma)$ we take some $k\in \Nb$, $A\in \overline{\cF}(G)$, and $B\in \overline{\cF}(\Gamma)$ such that $F\subseteq A*_kB$, and we put $\Phi(F)=A*_kB \in \cF(G*\Gamma)$.

Let $\Upsilon_G: \cF(G)\rightarrow [0, \infty)$ be a function. Define $\Upsilon_{G*\Gamma}: \cF(G*\Gamma)\rightarrow [0, \infty)$ by
$\Upsilon_{G*\Gamma}(F)=|\Phi(F)|\Upsilon_G(A)$, where $\Phi(F)=A*_kB$.

Let $\eta_{G*\Gamma}>0$ witness property SC$'$ for $\fC_{G*\Gamma}$. Put $\eta_G=\eta_{G*\Gamma}>0$.

Let $T_G\in \overline{\cF}(G)$. Put $T_{G*\Gamma}=T_G\in \overline{\cF}(G*\Gamma)$. Then we have $n$, $C_{G*\Gamma}\in \Nb$, and $S_{G*\Gamma, 1}, \dots, S_{G*\Gamma, n}\in \overline{\cF}(G*\Gamma)$ witnessing property SC$'$ for $\fC_{G*\Gamma}$. For $1\le j\le n$
express $\Phi(S_{G*\Gamma, j})$ as $A_j*_{k_j}B_j$ as per its definition. Put $S_{G, j}=A_j\in \overline{\cF}(G)$ for $1\le j\le n$. Also, put $C_G=2C_{G*\Gamma}\max_{1\le j\le n}k_j\in \Nb$.

Let $G*\Gamma\curvearrowright (X, \mu)$ be an action in $\fC_{G*\Gamma}$.
Let $W_1$ and $W_2$ be Borel subsets of $X$ with $\mu(W_1), \mu(W_2)\le \eta_G=\eta_{G*\Gamma}$. Then there are Borel sets $\cV_{G*\Gamma, 1}, \dots, \cV_{G*\Gamma, n}\subseteq X$ satisfying the following conditions:
\begin{enumerate}
\item $\sum_{j=1}^n\mu(\cV_{G*\Gamma, j})\Upsilon_{G*\Gamma}(S_{G*\Gamma, j})\le 1$,
\item if $w_1\in W_1$ and $w_2\in W_2$ satisfy $gw_1=w_2$ for some $g\in T_{G*\Gamma}$ then $w_1$ and $w_2$ are connected by a path of length at most $C_{G*\Gamma}$ in which each edge is an $S_{G*\Gamma, j}$-edge with both endpoints in $\cV_{G*\Gamma, j}$ for some $1\le j\le n$.
\end{enumerate}

For each $1\le j\le n$ put $\cV_{G, j}=\Phi(S_{G*\Gamma, j})\cV_{G*\Gamma, j}\subseteq X$.
Then
\begin{align*}
\sum_{j=1}^n\mu(\cV_{G, j})\Upsilon_G(S_{G, j})&\le \sum_{j=1}^n\mu(\cV_{G*\Gamma, j})|\Phi(S_{G*\Gamma, j})|\Upsilon_G(S_{G, j})\\
&=\sum_{j=1}^n\mu(\cV_{G*\Gamma, j})\Upsilon_{G*\Gamma}(S_{G*\Gamma, j})\le 1,
\end{align*}
verifying condition (i) in Definition~\ref{D-SC'}.

For each $1\le j\le n$, if $h\in S_{G*\Gamma, j}$ and $x, y\in \cV_{G*\Gamma, j}$ with $hx=y$ then clearly $x$ and $y$ are connected by an $(A_j\cup B_j)$-path in $\cV_{G, j}$ of length at most $2k_j$.

Let $g\in T_G=T_{G*\Gamma}$ and $w_1\in W_1$, $w_2\in W_2$ with $gw_1=w_2$. Then $w_1$ and $w_2$ are connected by a path of length at most $C_{G*\Gamma}$ in which each edge is an $S_{G*\Gamma, j}$-edge with both endpoints in $\cV_{G*\Gamma, j}$ for some $1\le j\le n$. From the above paragraph we conclude that
$w_1$ and $w_2$ are connected by a path of length at most $C_G$ in which each edge is either an $A_j$-edge or a $B_j$-edge with both endpoints in $\cV_{G, j}$ for some $1\le j\le n$. We may assume that $w_1\neq w_2$.
Removing cycles in this path, we may assume that it contains no cycles. Since $g\in G$ and the action $G*\Gamma\curvearrowright (X, \mu)$ is free, we see that no $B_j$-edge for any $1\le j\le n$ appears in this path. Thus each edge of this path is an $S_{G, j}$-edge with both endpoints in $\cV_{G, j}$ for some $1\le j\le n$. This verifies condition (ii) in Definition~\ref{D-SC'}.
\end{proof}

By combining Proposition~\ref{P-SC'} and Lemmas~\ref{L-locally finite not SC''},
\ref{L-virtually cyclic not SC''}, and \ref{L-free product} we obtain the following proposition,
which, with a boost from Proposition~\ref{P-finite index}, then yields Theorem~\ref{T-without SC}.

\begin{proposition} \label{P-free product}
Suppose that $G$ is either locally finite or virtually cyclic,
and let $\Gamma$ be a (not necessarily infinite) countable group. Then no free
p.m.p.\ action $G*\Gamma \curvearrowright (X,\mu )$ has property SC.
\end{proposition}

\begin{theorem} \label{T-without SC}
Suppose that $G$ is either locally finite or virtually free.
Then $G$ does not have property SC.
\end{theorem}

We will see later in Proposition~\ref{P-amenable SC} that if $G$ is amenable then it has
property SC if and only if it is neither virtually cyclic nor locally finite.

\subsection{Groups with property SC} \label{SS-with SC}

The following notion of w-normality was formulated by Popa for the purpose of expressing his
cocycle superrigidity theorem in \cite{Pop07} and will be similarly convenient in our setting.

\begin{definition}\label{D-w-normal}
A subgroup $G_0$ of $G$ is {\it w-normal} in $G$ if there are a countable ordinal $\gamma$ and a subgroup $G_\lambda$ of $G$ for each ordinal $0\le \lambda\le \gamma$ satisfying the following conditions:
\begin{enumerate}
\item for any $\lambda<\lambda'\le \gamma$ one has $G_\lambda\subseteq G_{\lambda'}$,
\item $G=G_\gamma$,
\item for each $\lambda<\gamma$, $G_\lambda$ is normal in $G_{\lambda+1}$,
\item for each limit ordinal $\lambda'\le \gamma$, $G_{\lambda'}=\bigcup_{\lambda<\lambda'}G_\lambda$.
\end{enumerate}
\end{definition}

Our goal is to prove, via several lemmas culminating in Theorem~\ref{T-ordinal SC},
that if $G$ contains an amenable w-normal subgroup which is neither locally finite nor virtually cyclic then $G$ has property SC.
This will involve an elaboration of the graph-theoretic arguments from Section~8.1 of \cite{Aus16}.

Let $S$ be a symmetric finite subset of $G$.
By an {\it $S$-path} in $G$ we mean a finite tuple $(g_0 , g_1 ,\dots , g_n )$ of elements of $G$
such that $g_{i-1} g_i^{-1} \in S$ for all $i=1,\dots , n$, in which case we call $n$ the {\it length}
of the path and say that the path {\it connects} $g_0$ and $g_n$ (its {\it endpoints}).
We say that a set $K\subseteq G$ is {\it $S$-connected} if every pair of distinct elements of $K$ is connected
by an $S$-path.

Given an $r\in\Nb$, a set $E\subseteq G$ is said to be
{\it $(S,r)$-separated} if for all distinct $f,g\in E$ one has $S^r f \cap S^r g = \emptyset$.
Given a set $W\subseteq G$ and an $r\in\Nb$, a set $E\subseteq W$ is said to be
{\it $(S,r)$-spanning} for $W$ if every $f\in W$ is connected to some $g\in E$ by an $S$-path
of length at most $r$. We also simply say {\it $S$-spanning} when $r=1$.
This should not be confused with the graph-theoretic notion of spanning tree, which we also use below.

For finite sets $F,K\subseteq G$ and $\delta > 0$, we say that $K$ is {\it $(F,\delta )$-invariant}
if $|FK\setminus K| < \delta |K|$.
Note that, given an $r\in\Nb$, if $F$ contains $e_G$ and satisfies $|F|>1$
and $K$ is $(F,\delta' )$-invariant where $\delta' = \delta (1-|F|)/(1-|F|^r )$
then $K$ is $(F^r,\delta )$-invariant, for using the fact that $F$ contains $e_G$ we have
\begin{align*}
|F^r K\setminus K|
= \sum_{j=1}^r |F^j K\setminus F^{j-1} K|
\leq \sum_{j=1}^r |F|^{j-1} |FK\setminus K|
< \delta |K| .
\end{align*}

The following is a variation on Lemma~8.3 of \cite{Aus16}.
For an $\eps\geq 0$ and a collection $\sT$ of finite subsets of $G$, we say a finite set $K\subseteq G$ is
{\it tiled to within $\eps$ by $\sT$} if the members of $\sT$ are pairwise disjoint subsets of $K$ and
$|\bigsqcup_{T\in\sT} T| \geq (1-\eps )|K|$.

\begin{lemma}\label{L-weighted}
Suppose that $G$ is finitely generated and let $S$ be a finite symmetric generating set for $G$ containing $e_G$.
Let $\eps > 0$.
Let $F$ be a finite subset of $G$ and $\delta > 0$. Then there exists a $\zeta > 0$
such that every $(S,\zeta )$-invariant nonempty finite subset of $G$
is tiled to within $\eps$ by a collection of $S^2$-connected $(F,\delta )$-invariant finite subsets of $G$.
\end{lemma}

\begin{proof}
Take an $r\in\Nb$ such that $F\subseteq S^r$.
Set $\zeta = \min \{ \eps^2 , (\delta (1-|S|)/(1-|S|^r ))^2 \}$.
Let $K$ be an $(S,\zeta )$-invariant nonempty finite subset of $G$.
Consider the partition $K_1 \sqcup\dots\sqcup K_n$ of $K$ into maximal $S^2$-connected subsets.
Then $SK_i \cap SK_j = \emptyset$ for $i\neq j$, and so
\begin{align}\label{E-weighted}
\sum_{i=1}^n \frac{|K_i|}{|K|} \cdot\frac{|SK_i\setminus K_i|}{|K_i|}
= \frac{|SK\setminus K|}{|K|}
< \zeta .
\end{align}
Write $I$ for the set of all $i\in \{ 1,\dots , n\}$ such that $|SK_i\setminus K_i|/|K_i| < \sqrt{\zeta}$.
Then by (\ref{E-weighted}) we must have $\sum_{i\in I} |K_i | \geq (1-\sqrt{\zeta})|K|$, which shows, since
$\sqrt{\zeta} \leq\eps$, that the collection $\{ K_i \}_{i\in I}$ tiles $K$ to within $\eps$.
Moreover, since $\sqrt{\zeta} \leq \delta (1-|S|)/(1-|S|^r )$ and $F\subseteq S^r$ we infer from the observation
before the lemma that for each $i\in I$ the set $K_i$ is $(F,\delta )$-invariant, completing the proof.
\end{proof}

We next state a version of the Ornstein--Weiss tiling theorem that demands connectedness of the tiles.
It follows from one of the usual forms of the Ornstein--Weiss tiling theorem
(Theorem~4.46 of \cite{KerLi16}) and Lemma~\ref{L-weighted}.

\begin{lemma}\label{L-connected OW}
Suppose that $G$ is amenable and finitely generated. Let $S$ be a finite symmetric generating set for $G$
containing $e_G$.
Let $E$ be a finite subset of $G$ and $\delta > 0$.
Then there are $(E,\delta )$-invariant $S^2$-connected sets $F_1 ,\dots , F_m\in \cF (G)$
such that for any free p.m.p.\ action $G\curvearrowright (X, \mu)$
we can find Borel sets $Z_1 , \dots , Z_m \subseteq X$ such that
the collection $\{ F_k z : 1\leq k\leq m\text{ and } z\in Z_k \}$ is disjoint and its union
has $\mu$-measure at least $1-\delta$.
\end{lemma}

The following is essentially Lemma~8.5 of \cite{Aus16}.

\begin{lemma}\label{L-tree group}
Suppose that $G$ is finitely generated and not virtually cyclic, and
let $A$ be a generating set for $G$ in $\overline{\cF}(G)$.
Then there is a constant $b>0$ such that given any $r\in\Nb$ and 
$A$-connected finite set $F\subseteq G$ satisfying $|A^r F| \leq 2|F |$
there exists an $A$-connected set $T\subseteq A^rF$ such that $T\cap F$ is $(A, 2r)$-spanning for $F$  and $|T|\le b|F|/r$.
\end{lemma}

\begin{proof}
As $G$ is not virtually cyclic, we can find a $c>0$ such that $|A^n| \geq cn^2$ for all $n\in\Nb$
(Corollary~3.5 of \cite{Man12}).
Let $r$ and $F$ be as in the lemma statement.

Take a maximal $(A,r)$-separated subset $V$ of $F$. Then we have
\begin{align*}
|V||A^r|
= \bigg| \bigsqcup_{g\in V} A^r g \bigg|
\leq |A^r F| \leq 2|F|
\end{align*}
whence
\begin{align}\label{E-E}
|V| \leq \frac{2|F|}{|A^r|} \leq \frac{2}{cr^2} |F| .
\end{align}
The set $V$ is $(A,2r)$-spanning for $F$ by maximality.
Consider the graph $(V , E )$ whose edges are
those pairs of vertices which can be joined by an $A$-path within $A^r F$ of length at most $4r+1$.
Since $F$ is $A$-connected, for all $v_1,v_2\in V$ there is an $A$-path in $F$ connecting
$v_1$ to $v_2$. For each point $z$ in this path, we may connect $z$ to some $v_z\in V$ by an $A$-path $p_z$ in $A^rF$ of length at most $2r$
(using the $(A,2r)$-spanningness of $V$ in $F$ and
the fact that an $A$-path of length at most $2r$ with both endpoints in $F$
must be entirely contained in $A^r F$)
we can construct an $A$-path in $A^rF$
from $v_1$ to $v_2$ in which points of $V$ appear in every interval of length $4r+1$, by inserting $p_z$ and the reverse of $p_z$ at $z$.
This shows that the graph $(V, E)$ is connected.

Applying the standard procedure for producing a spanning tree, we start with $(V, E)$
and then recursively construct a sequence of graphs with vertex set $V$
by removing one edge at each stage so as to destroy some cycle in the
graph at that stage, until there are no more cycles to destroy and we arrive at a spanning tree $(V, E')$.
Then $(V, E')$ is an $A^{4r+1}$-tree in $F$ which is $(A,2r)$-spanning for $F$.

For each pair $(v,w)$ in $E'$, choose an $A$-path in $A^r F$ joining $v$ to $w$
of length at most $4r+1$.
Denote by $T$ the collection
of all vertices which appear in one of these paths.
Note that $T$ is an $A$-connected  set in $A^rF$ such that $T\cap F$ is $(A,2r)$-spanning for $F$.
Moreover, using \eqref{E-E} we have
\begin{align*}
|T| \leq |V|+4r|E'| \leq (4r+1)|V|
\leq 5r\cdot \frac{2}{cr^2} |F| = \frac{10}{cr} |F| .
\end{align*}
We can therefore take $b=10/c$.
\end{proof}

If we take $G^\flat=G$ in Proposition~\ref{P-normal SC}, then the proof there actually shows the following.

\begin{lemma} \label{L-same S}
Let $G_0$ be a subgroup of $G$ which has the shrinking property.
Then in the definition of property SC for a class $\fC$ of free p.m.p.\ actions of $G$
(Definition~\ref{D-SC}) it is possible, for each $\Upsilon$, to choose $S$ to be a subset of $G_0$
depending only on $\Upsilon |_{\cF(G_0)}$.
\end{lemma}

\begin{lemma} \label{L-union SC}
Let $\gamma$ be a countable ordinal, and
suppose that for each ordinal $\lambda<\gamma$ there is a subgroup $G_\lambda$ of $G$
so that the following conditions hold:
\begin{enumerate}
\item for any $\lambda<\lambda'<\alpha$ one has $G_\lambda\subseteq G_{\lambda'}$,

\item $G_\lambda$ has property SC for every $\lambda<\gamma$.
\end{enumerate}
Then $\bigcup_{\lambda<\gamma}G_\lambda$ has property SC.
\end{lemma}

\begin{proof} 
By Proposition~\ref{P-SC'} we know that $G_0$ has the shrinking property. Let $\Upsilon$ be a function $\cF(\bigcup_{\lambda<\gamma}G_\lambda)\rightarrow [0, \infty)$. By Lemma~\ref{L-same S} there is some $S\in \overline{\cF}(G_0)$ witnessing property SC for $G_\lambda$ with respect to $\Upsilon|_{\cF(G_\lambda)}$ for all $\lambda<\gamma$. Let $T\in \overline{\cF}(\bigcup_{\lambda<\gamma}G_\lambda)$. Then $T\in \overline{\cF}(G_\lambda)$ for some $\lambda<\gamma$. Then we have $C, n\in \Nb$ and $S_1, \dots, S_n\in \overline{\cF}(G_\lambda)$ satisfying the conditions in Definition~\ref{D-SC}.
\end{proof}

\begin{lemma} \label{L-amenable to SC}
Suppose that $G$ is amenable and not locally virtually cyclic. Then $G$ has property SC.
\end{lemma}

\begin{proof}
We consider first the case $G$ is finitely generated. Then $G$ is not virtually cyclic. By Proposition~\ref{P-SC'} it suffices to show that $G$ has property SC$'$.

Take an $S_0\in \overline{\cF}(G)$  generating $G$.
Set $S_1=S_0^2\in \overline{\cF}(G)$. Let $b > 0$ be as given by Lemma~\ref{L-tree group} with respect to the generating set $S_1$ for $G$.

Let $\Upsilon$ be a function $\cF(G)\rightarrow [0, \infty)$.
Choose an $r\in\Nb$ large enough so that
\begin{align}\label{E-action r shrinking}
3b\Upsilon(S_1)\leq r.
\end{align}
Set $S_2=S_1^{2r}\in \overline{\cF}(G)$. Take $\eta>0$ such that $2\eta |S_2|\Upsilon(S_2)<1/3$.

Let $T\in \overline{\cF}(G)$.
Put $n=3$ and $S_3=T\in \overline{\cF}(G)$. Take $0<\delta<1$ such that $2\delta |S_3|\Upsilon(S_3)<1/3$
and put $\zeta=\delta/(1+|S_2|\cdot |S_2TS_2|)>0$.
By Lemma~\ref{L-connected OW} there are $m\in \Nb$ and
$(S_2TS_2,\zeta)$-invariant $S_1$-connected sets $F_1 ,\dots , F_m\in \cF (G)$  such that for any free p.m.p.\ action $G\curvearrowright (X, \mu)$
we can find Borel sets $Z_1 , \dots , Z_m \subseteq X$ such that
the collection $\{ F_k z : 1\leq k\leq m\text{ and } z\in Z_k \}$ is disjoint and its union
has $\mu$-measure at least $1-\zeta $.
By our choice of $b$ via Lemma~\ref{L-tree group},
for each $k=1,\dots,m$ we can find an $S_1$-connected set $T_k^*\subseteq S_1^{r}F_k$  such that $T_k^\dag:=T_k^*\cap F_k$ is $(S_1, 2r)$-spanning for $F_k$ and $|T_k^*|\le b|F_k|/r$.
Put $C=2+\max_{1\le k\le m}|T_k^*|\in \Nb$.
For each $1\le k\le m$ denote by $F_k'$ the set of all $g\in F_k$ satisfying $S_2TS_2g\subseteq F_k$, and note that since $F_k$ is $(S_2TS_2, \zeta)$-invariant
we have $|F_k\setminus F_k'|\le |S_2TS_2|\zeta |F_k|$.

Now let $G\curvearrowright (X, \mu)$ be a free p.m.p.\ action,
and let $W_1, W_2$ be Borel subsets of $X$ with $\mu(W_1), \mu(W_2)\le \eta$.
Take $Z_1, \dots, Z_m$ as above.

For each $i=1, 2$ put $W_i'=W_i\cap S_2\bigcup_{k=1}^m(F_k'\cap T_k^\dag)Z_k\subseteq W_i$
and observe that
\begin{align*}
\mu(W_i\setminus W_i')
&\le \mu\bigg(X\setminus S_2\bigcup_{k=1}^m(F_k'\cap T_k^\dag)Z_k\bigg)\\
&\le \mu\bigg(X\setminus \bigcup_{k=1}^mF_kZ_k\bigg)+\sum_{k=1}^m\mu(F_kZ_k\setminus S_2(F_k'\cap T_k^\dag)Z_k)\\
&\le \zeta+\sum_{k=1}^m\mu(Z_k)|F_k\setminus S_2(F_k'\cap T_k^\dag)|\\
&\le \zeta+\sum_{k=1}^m\mu(Z_k)|S_2|\cdot |T_k^\dag\setminus F_k'|
\hspace*{8mm} \text{(since $F_k\subseteq S_2T_k^\dag$)}\\
&\le \zeta+|S_2TS_2|\zeta\sum_{k=1}^m\mu(Z_k)|S_2|\cdot |F_k|\\
&\le \zeta+|S_2TS_2|\cdot |S_2|\zeta=\delta .
\end{align*}

Set
\[
\cV_1=\bigcup_{k=1}^m T_k^*Z_k , \hspace*{6mm}
\cV_2=S_2(W_1\cup W_2) , \hspace*{6mm}
\cV_3=S_3((W_1\setminus W_1')\cup (W_2\setminus W_2')).
\]
Then
\[\Upsilon(S_1)\mu(\cV_1)\le \Upsilon(S_1)\sum_{k=1}^m\mu(Z_k) |T_k^*|\le \Upsilon(S_1)\sum_{k=1}^m\mu(Z_k)\frac{b|F_k|}{r}\le \Upsilon(S_1)\frac{b}{r}\overset{\eqref{E-action r shrinking}}{\le} \frac13
\]
and
\[
\Upsilon(S_2)\mu(\cV_2)\le \Upsilon(S_2)|S_2|(\mu(W_1)+\mu(W_2))\le 2\eta |S_2|\Upsilon(S_2)\le \frac13
\]
and
\[
\Upsilon(S_3)\mu(\cV_3)\le \Upsilon(S_3)|S_3|(\mu(W_1\setminus W_1')+\mu(W_2\setminus W_2'))\le 2\delta |S_3|\Upsilon(S_3)< \frac13
\]
so that
\[
\sum_{j=1}^3\Upsilon(S_j)\mu(\cV_j)\le 1,
\]
which verifies condition (i) in Definition~\ref{D-SC'}.

Let $g\in T$ and $w_1\in W_1, w_2\in W_2$ with $gw_1=w_2$. If $w_1\not\in W_1'$ or $w_2\not\in W_2'$, then $(w_1, w_2)$ is an $S_3$-edge with both endpoints in $\cV_3$.
Thus we may assume that $w_i\in W_i'$ for $i=1, 2$.
For $i=1, 2$, we have $w_i=s_it_iz_i$ for some $s_i\in S_2$, $1\le k_i\le m$, $t_i\in F_{k_i}'\cap T_{k_i}^\dag$, and $z_i\in Z_{k_i}$.
Then $(w_i, t_iz_i)$ is an $S_2$-edge with both endpoints in $\cV_2$. Note that $s_2^{-1}gs_1\in S_2TS_2$, and hence $(s_2^{-1}gs_1)t_1\in F_{k_1}$.
Since
\[
t_2z_2=s_2^{-1}w_2=s_2^{-1}gw_1=(s_2^{-1}gs_1)t_1z_1,
\]
we get that $k_1=k_2$, $t_2=(s_2^{-1}gs_1)t_1$, and $z_2=z_1$. Then $t_1z_1$ and $t_2z_2$ are connected by an $S_1$-path of length at most $|T_{k_1}^*|$ with all points in $\cV_1$.
Therefore $w_1$ and $w_2$ are connected by a path of length at most $C$ in which each edge is an $S_j$-edge with both endpoints in $\cV_j$ for some $1\le j\le 3$, verifying condition (ii) in Definition~\ref{D-SC'}. Therefore $G$ has property SC$'$, and hence has property SC. 

Now consider the case $G$ is not finitely generated. By hypothesis we can find an increasing sequence
$G_0 \subseteq G_1 \subseteq\dots$ of finitely generated subgroups of $G$ with union $G$ such that
$G_0$ is not virtually cyclic. Suppose for a given $n\in \Nb$ that $G_n$ is virtually cyclic. Then $G_n$ has a finite-index subgroup $G_n'$ isomorphic to $\Zb$.
Since $G_0$ is infinite, $G_0\cap G_n'$ is nontrivial. Then $G_0\cap G_n'$ has finite index in $G_n'$, and hence has finite index in $G_n$. Thus $G_0\cap G_n'$ has finite index in $G_0$. This shows that $G_0$ is virtually cyclic, a contradiction. Therefore for each $n\in \Nb$ the group $G_n$ is not virtually cyclic. Since subgroups of amenable groups are amenable, each $G_n$ is amenable. From the finitely generated case of the lemma we conclude that each $G_n$ has property SC. Then from Lemma~\ref{L-union SC} we get that $G$ has property SC.
\end{proof}

The following we obtain from Theorem~\ref{T-without SC} and
Lemmas~\ref{L-increasing virtual cyclic to SC} and \ref{L-amenable to SC}.

\begin{proposition} \label{P-amenable SC}
An amenable $G$ has property SC if and only if it is neither locally finite nor virtually cyclic.
\end{proposition}

From Proposition~\ref{P-amenable SC} and Lemma~\ref{L-union SC}  we finally get:

\begin{theorem} \label{T-ordinal SC}
Suppose that $G$ has a w-normal subgroup $G_0$ which is amenable but neither locally finite nor virtually cyclic.
Then $G$ has property SC.
\end{theorem}

\subsection{Product groups} \label{SS-product groups}

\begin{lemma} \label{L-small}
For every $\varepsilon>0$ there is an $S\in \overline{\cF}(G)$ such that for every free
p.m.p.\ action $G\curvearrowright (X, \mu)$ there is a Borel
set $\cV\subseteq X$ with $S\cV=X$ and $\mu(\cV)\le \varepsilon$.
\end{lemma}

\begin{proof}
By Proposition~\ref{P-shrinking}, $G$ either has the shrinking property or is amenable.
If $G$ has the shrinking property the lemma is clear, while if $G$ is amenable then the lemma follows from the Ornstein--Weiss quasitower theorem \cite[Theorem 4.46]{KerLi16}.
\end{proof}

For two countably infinite groups $G$ and $\Gamma$, we say that $G\times \Gamma$ has
{\it property SC for product actions} if the class $\fC_{\prdct}$ has property SC, where $\fC_{\prdct}$ consists of all p.m.p.\ actions of the form $G\times \Gamma \curvearrowright (X\times Y, \mu)$
where $G\curvearrowright X$ and $G\curvearrowright Y$ are free actions
on standard Borel spaces
and $G\times \Gamma \curvearrowright X\times Y$ is the product action, with
$\mu$ not necessarily being a product measure.

\begin{lemma} \label{L-product}
Let $G$ and $\Gamma$ be countably infinite groups. Suppose that $G$ has the shrinking property. Then $G\times \Gamma$ has property SC for product actions.
\end{lemma}

\begin{proof} We have $S_{1, G}\in \overline{\cF}(G)$ witnessing the shrinking property for $G$. Put $S_1=S_{1, G}\times \{e_\Gamma\}\in \overline{\cF}(G\times \Gamma)$.
Let $\Upsilon$ be a function $\cF(G\times \Gamma)\to [0, \infty)$. Take $\varepsilon>0$ such that $\varepsilon\Upsilon(S_1)\le 1/3$.

By our choice of $S_{1, G}$, there is an $S_G\in \overline{\cF}(G)$ such that for any $\delta>0$ there is a $C_1\in \Nb$ so that for any free p.m.p.\ action $G\curvearrowright (X, \mu_X)$ we can find Borel sets $Z_X\subseteq \cV_X\subseteq X$ satisfying the following conditions:
\begin{enumerate}
\item $S_G\cV_X=X$,
\item $\mu_X(\cV_X)\le \varepsilon$ and $\mu_X(Z_X)\le \delta$,
\item every point of $\cV_X$ is connected to some point of $Z_X$ by an $S_{1, G}$-path of length at most $C_1$ with all points in $\cV_X$.
\end{enumerate}

Put $S=S_G\times \{e_\Gamma\}\in \overline{\cF}(G\times \Gamma)$.

Let $T\in \overline{\cF}(G\times \Gamma)$. Take $T_G\in \overline{\cF}(G)$ and $T_\Gamma\in \overline{\cF}(\Gamma)$ such that $T\subseteq T_G\times T_\Gamma$. Put $S_2=(T_G\cup S_{1, G})\times \{e_\Gamma\}\in \overline{\cF}(G\times \Gamma)$.
Take an $\eta>0$ such that $\eta\Upsilon(S_2)\le 1/3$. By Lemma~\ref{L-small}  there is an $S_\Gamma\in \overline{\cF}(\Gamma)$ such that for any free p.m.p.\ action
$\Gamma\curvearrowright (Y, \mu_Y)$ there is a Borel set $\cV_Y\subseteq Y$ so that $S_\Gamma\cV_Y=Y$ and $\mu_Y(\cV_Y)\le \eta$.

Put $S_3=\{e_G\}\times (S_\Gamma T_\Gamma S_\Gamma)\in \overline{\cF}(G\times \Gamma)$. Pick a $\delta>0$ such that $\delta\Upsilon(S_3)\le 1/3$. Then we have $C_1$ as above. Put $C=4C_1+3\in \Nb$.

Let $G$ and $\Gamma$ act freely on standard Borel spaces $X$ and $Y$, respectively. Let $\mu$ be a $(G\times \Gamma)$-invariant Borel probability measure on $X\times Y$. Denote by $\mu_X$ and $\mu_Y$ the push-forward of $\mu$ under the projections $X\times Y\rightarrow X$ and $X\times Y\rightarrow Y$ respectively.
Then we have $Z_X, \cV_X$ and $\cV_Y$ as above.

Put $\cV_1=\cV_X\times Y$, $\cV_2=X\times \cV_Y$, and $\cV_3=Z_X\times Y$. Then
\begin{align*}
\sum_{j=1}^3\Upsilon(S_j)\mu(\cV_j)=\Upsilon(S_1)\mu_X(\cV_X)+\Upsilon(S_2)\mu_Y(\cV_Y)+\Upsilon(S_3)\mu_X(Z_X)\le 1,
\end{align*}
verifying condition (i) in Definition~\ref{D-SC}.

Put $W=\cV_X\times Y$. Then $SW=X\times Y$, verifying condition (ii) in Definition~\ref{D-SC}.

Let $t=(t_G, t_\Gamma)\in T$ and $w_1, w_2\in W$ with $tw_1=w_2$. Say, $w_i=(x_i, y_i)$ for $i=1, 2$. Then $x_i$ is connected to some $z_i\in Z_X$ by an $S_{1, G}$-path of length at most $C_1$ with all points in $\cV_X$. Put $w_i'=(z_i, y_i)$. Then $w_i$ is connected to  $w_i'$ by an $S_1$-path of length at most $C_1$ with all points in $\cV_1$. We have $sy_1\in \cV_Y$ for some $s\in S_\Gamma$. Put $w_1''=(z_1, sy_1)$ and $w_2''=(z_2, sy_1)$.  Note that $sy_1=st_\Gamma^{-1}y_2\in S_\Gamma T_\Gamma S_\Gamma y_2$.
Thus $(w_i', w_i'')$ is an $S_3$-edge with both endpoints in $\cV_3$ for $i=1, 2$. Also note that $z_2\in S_{1, G}^{C_1}x_2\subseteq S_{1, G}^{C_1}T_Gx_1\subseteq S_{1, G}^{C_1}T_GS_{1, G}^{C_1}z_1$. Thus $w_1''$ and $w_2''$ are connected by an $S_2$-path of length at most $2C_1+1$ with all points in $\cV_2$. We conclude that $w_1$ and $w_2$ are connected by a path of length at most $4C_1+3$ in which each edge is an $S_j$-edge with both endpoints in $\cV_j$ for some $1\le j\le 3$, verifying condition (iii) in Definition~\ref{D-SC}.
\end{proof}

Gaboriau showed that if $R$ and $S$ are countable aperiodic Borel equivalence relations on standard Borel spaces $X$ and $Y$, respectively, and $\mu$ is an $(R\times S)$-invariant Borel probability measure on $X\times Y$, then the cost of $R\times S$ on $(X\times Y, \mu)$ is equal to $1$ \cite{Gab00}\cite[Theorem 24.9]{KecMil04}.
The following result is an analogue of Gaboriau's theorem.

\begin{proposition} \label{P-product}
Let $G$ and $\Gamma$ be countably infinite groups. Then $G\times \Gamma$ has property SC for product actions if and only if at least one of $G$ and $\Gamma$ is not locally finite.
\end{proposition}

\begin{proof}
If at least one of $G$ and $\Gamma$, say $G$, has the shrinking property, then $G\times \Gamma$ has property SC for product actions by Lemma~\ref{L-product} and $G$ is not locally finite by Proposition~\ref{P-shrinking}. Thus we may assume that neither $G$ nor $\Gamma$ has the shrinking property. By Proposition~\ref{P-shrinking}, each of $G$ and $\Gamma$ is either locally finite or virtually cyclic. If at least one of $G$ and $\Gamma$ is not locally finite, then $G\times \Gamma$ is neither locally finite nor virtually cyclic, and so $G\times \Gamma$, being amenable, has property SC by Proposition~\ref{P-amenable SC}.
If on the other hand $G$ and $\Gamma$ are both locally finite then $G\times \Gamma$
is locally finite, which implies by Theorem~\ref{T-without SC} that $G\times \Gamma$ has no free p.m.p.\ actions
with property SC and hence does not itself have property SC for product actions.
\end{proof}

\section{Measure entropy and Shannon orbit equivalence} \label{S-Shannon OE}

We devote ourselves in this section to the proof of the following theorem,
which together with Theorem~\ref{T-ordinal SC} yields Theorem~\ref{T-measure}.

\begin{theorem} \label{T-Shannon SC to entropy}
Let $G$ and $H$ be countably infinite groups and let $G\curvearrowright (X,\mu )$ and $H\curvearrowright (Y,\nu )$ 
be free p.m.p.\ actions which are Shannon orbit equivalent.
Suppose that $G\curvearrowright (X, \mu)$ has property SC . Then
\[
h_\nu (H\curvearrowright Y)\ge h_\mu (G\curvearrowright X).
\]
\end{theorem}

For the purpose of proving the theorem we may assume, by conjugating the 
$H$-action by a Shannon orbit equivalence, that $(X, \mu) = (Y,\nu )$
and that the identity map from $X$ to itself provides a Shannon orbit equivalence
between the two actions. As usual denote the associated cocycles $G\times X \to H$
and $H\times X \to G$ by $\kappa$ and $\lambda$, respectively.

For each $g\in G$ we write $\sP_g$ for the countable Borel partition of $X$ consisting of the sets
$X_{g,t} = \{ x\in X : gx = tx \}$ for $t\in H$, and likewise for $t\in H$ we write
$\sP_t$ for the countable Borel partition of $X$ consisting of the sets $X_{g,t}$ for $g\in G$.
For every $F$ in $\cF(G)$ or $\cF(H)$, set ${}_F\sP=\bigvee_{g\in F}\sP_g$. Then $H_\mu({}_F\sP)<\infty$. For every $F\in \cF(G)$ and $L\in \cF(H)$, denote by ${}_{F, L}\overline{\sP}$ the finite set consisting of all $P\in {}_F\sP$ satisfying $\kappa(g, P)\in L$ for all $g\in F$ and set $X_{F, L}=\bigcup {}_{F, L}\overline{\sP}$. Denote by ${}_{F, L}\sP$ the finite partition ${}_{F, L}\overline{\sP}\cup \{X\setminus X_{F, L}\}$ of $X$. Similarly, denote by ${}_{L, F}\overline{\sP}$ the finite set consisting of all $P\in {}_L\sP$ satisfying $\lambda(t, P)\in F$ for all $t\in L$,
and set $X_{L, F}=\bigcup {}_{L, F}\overline{\sP}$ and ${}_{L, F}\sP={}_{L, F}\overline{\sP}\cup \{X\setminus X_{L, F}\}$.

Given a finite disjoint collection $\sC$ of Borel subsets of $X$ and a nonempty finite set $V$, we define on the set
of all maps with domain some collection of subsets of $X$ containing $\sC$ and codomain $\Pb_V$
the pseudometric
\[
\rho_\sC (\varphi , \psi ) = \sum_{A\in\sC} \m (\varphi (A)\Delta\psi (A)) .
\]

\begin{lemma} \label{L-Shannon approximation for group measure}
Let $L\in \overline{\cF}(H)$ and $0<\tau<1$.
Take an $F^\natural\in \overline{\cF}(G)$ such that $\mu(X_{L^2, F^\natural})\ge 1-\tau/30$, and take an $F\in \overline{\cF}(G)$ such that $F^\natural\subseteq F$ and $\mu(X_{L^2, F})\ge 1-\tau/(30|F^\natural|)$. Let $0<\tau'\le \tau/(60|F|^2)$.
Let $\pi: G\rightarrow \Sym(V)$ be an $(F, \tau')$-approximation for $G$. Let $\varphi\in \Hom_\mu({}_{L^2, F}\sP, F, \tau', \pi)$. Take $\sigma': L^2\rightarrow V^V$ such that
\[
\sigma'_tv=\pi_{\lambda(t, A)}v
\]
for all $t\in L^2$, $A\in {}_{L^2, F}\overline{\sP}$ and $v\in \varphi(A)$. Then there is an $(L, \tau)$-approximation $\sigma: H\rightarrow \Sym(V)$ for $H$ such that
$\rho_{\Hamm}(\sigma_t, \sigma'_t)\le \tau/5$ for all $t\in L^2$.
\end{lemma}

\begin{proof}
Denote by $V_F$ the set of all $v\in V$ satisfying $\pi_g\pi_hv=\pi_{gh}v$ for all $g, h\in F$ and $\pi_gv\neq \pi_hv$ for all distinct $g, h\in F$.
Then
\[
\m(V\setminus V_F)\le 2|F|^2\tau'\le \frac{\tau}{30}.
\]
Set
\[
V'=\bigcup_{g\in F}\bigcup_{B\in {}_{L^2, F}\overline{\sP}}(\varphi(g^{-1}B)\Delta \pi_{g^{-1}}\varphi(B))\subseteq V.
\]
Then
\[
\m(V')\le  |F|\tau'\le \frac{\tau}{60}.
\]
Set $V^*=V_F\setminus V'$. Then
\[
\m(V^*)\ge \m(V_F)-\m(V')\ge 1-\frac{\tau}{30}-\frac{\tau}{60}=1-\frac{\tau}{20}.
\]

For each $t\in L^2$ set $V^\sharp_t=\bigsqcup_{A, B\in {}_{L^2, F}\overline{\sP}}\varphi(A\cap \lambda(t, A)^{-1}B)$.
Then
\begin{align*}
\m(V^\sharp_t)&\ge \mu\bigg(\bigsqcup_{A, B\in {}_{L^2, F}\overline{\sP}}(A\cap \lambda(t, A)^{-1}B)\bigg)-\tau'\\
&\ge \mu\bigg(\bigsqcup_{A\in {}_{L^2, F^\natural}\overline{\sP}}\bigsqcup_{B\in {}_{L^2, F}\overline{\sP}} (A\cap \lambda(t, A)^{-1}B)\bigg)-\tau'\\
&\ge \mu(X_{L^2, F^\natural})-|F^\natural|(1-\mu(X_{L^2, F}))-\tau'\\
&\ge 1-\frac{\tau}{30}-\frac{\tau}{30}-\frac{\tau}{60}=1-\frac{\tau}{12}.
\end{align*}

Let $s, t\in L^2$ with $st\in L^2$.
Let $v\in V^*\cap V^\sharp_t$. Then $v\in \varphi(A\cap \lambda(t, A)^{-1}B)=\varphi(A)\cap \varphi(g^{-1}B)$ for some $A, B\in {}_{L^2, F}\overline{\sP}$, where $g=\lambda(t, A)\in F$.
Since $v\notin V'$, we have $v\in \varphi(A)\cap \pi_{g^{-1}}\varphi(B)$. Using
the fact that $v\in V_F$, we get $v\in \varphi(A)\cap \pi_g^{-1}(\varphi(B))$.
Put $h=\lambda(s, B)\in F$. Note that $\sigma'_tv=\pi_g v\in \varphi(B)$. Thus
\begin{align*}
\sigma_s'\sigma_t'v=\sigma_s'\pi_gv=\pi_h\pi_gv=\pi_{hg}v.
\end{align*}
For every $x\in A\cap g^{-1}B$ we have $tx=gx\in B$ and hence
\[
\lambda(st, x)=\lambda(s, tx)\lambda(t, x)=hg.
\]
Since $st\in L^2$ and $A\in {}_{L^2, F}\overline{\sP}$, we have $hg=\lambda(st, A)\in F$.
Thus $\sigma'_{st}w=\pi_{hg}w$ for all $w\in \varphi(A\cap g^{-1}B)$, and so
\[
\sigma_s'\sigma_t'v=\pi_{hg}v=\sigma'_{st}v.
\]
We conclude that
\begin{align} \label{E-Shannon group measure}
\rho_{\Hamm}(\sigma'_s\sigma'_t, \sigma'_{st})\le 1-\m(V^*\cap V^\sharp_t)\le
\frac{\tau}{20}+ \frac{\tau}{12}=\frac{2\tau}{15}.
\end{align}

Note that $\sigma'_{e_H}=\pi_{e_G}$ on $\varphi(X_{L^2, F})$.
Since $\pi$ is an $(F,\tau' )$-approximation for $G$ we have
$\rho_{\Hamm} (\pi_{e_G}, \id_V ) = \rho_{\Hamm} (\pi_{e_G}\pi_{e_G},\pi_{e_G}) \leq \tau'$ and hence
\begin{align*}
\rho_{\Hamm}(\sigma'_{e_H}, \id_V)&\le \rho_{\Hamm}(\sigma'_{e_H}, \pi_{e_G})+\rho_{\Hamm}(\pi_{e_G}, \id_V)\\
&\le 1-\m(\varphi(X_{L^2, F}))+\tau'\\
&\le 1-\mu(X_{L^2, F})+2\tau'\\
&\le \frac{\tau}{30}+\frac{\tau}{30}=\frac{\tau}{15}.
\end{align*}
For each $t\in L^2$ pick a $\sigma_t\in \Sym(V)$ such that $\sigma_tv=\sigma'_tv$ for all $v\in V$ satisfying $\sigma'_{t^{-1}}\sigma'_tv=v$.
For each $t\in L^2$, taking $s=t^{-1}$ in \eqref{E-Shannon group measure} we conclude that
\begin{align*}
\rho_{\Hamm}(\sigma_t, \sigma'_t)&\le \rho_{\Hamm}(\sigma'_{t^{-1}}\sigma'_t, \id_V)\\
&\le \rho_{\Hamm}(\sigma'_{t^{-1}}\sigma'_t, \sigma'_{e_H})+\rho_{\Hamm}(\sigma'_{e_H}, \id_V)\\
&\le \frac{2\tau}{15}+\frac{\tau}{15}=\frac{\tau}{5}.
\end{align*}
For any $s,t \in L$ we then have
\begin{align*}
\rho_{\Hamm}(\sigma_s\sigma_t, \sigma_{st})&\le \rho_{\Hamm}(\sigma_s, \sigma'_s)+\rho_{\Hamm}(\sigma_t, \sigma'_t)+\rho_{\Hamm}(\sigma'_s\sigma'_t, \sigma'_{st})+\rho_{\Hamm}(\sigma_{st}, \sigma'_{st})\\
&\le \frac{\tau}{5} +\frac{\tau}{5}+\frac{2\tau}{15}+\frac{\tau}{5}<\tau.
\end{align*}
Let $s, t\in L^2$ be distinct. Let $A\in {}_{L^2, F}\overline{\sP}$. Say, $g=\lambda(t, A)\in F$ and $h=\lambda(s, A)\in F$. Take $x\in A$. Then
$gx=tx\neq sx=hx$, and hence $g\neq h$. Thus for any $v\in \varphi(A)\cap V_F$ we have
$\sigma'_tv=\pi_gv\neq \pi_hv=\sigma'_sv$. This shows that $\sigma'_tv\neq \sigma'_sv$ for all $v\in V_F\cap \varphi(X_{L^2, F})$.
Therefore
\begin{align*}
\rho_{\Hamm}(\sigma_s, \sigma_t)&\ge \rho_{\Hamm}(\sigma'_s, \sigma'_t)-\rho_{\Hamm}(\sigma_s, \sigma'_s)-\rho_{\Hamm}(\sigma_t, \sigma'_t)\\
&\ge \m(V_F\cap \varphi(X_{L^2, F}))-\frac{\tau}{5}-\frac{\tau}{5}\\
&\ge \m(\varphi(X_{L^2, F}))-\frac{\tau}{30}-\frac{2\tau}{5}\\
&\ge \mu(X_{L^2, F})-\tau'-\frac{13\tau}{30}\\
&\ge 1-\frac{\tau}{30}-\frac{\tau}{60}-\frac{13\tau}{30}> 1-\tau.\qedhere
\end{align*}
\end{proof}

\begin{lemma} \label{L-Shannon hom and approximation for group}
Let $\sC$ be a finite Borel partition of $X$, $L\in \overline{\cF}(H)$, and $0<\tau<1$.
Take
an $F\in \overline{\cF}(G)$ such that
$\mu(X_{L^2, F})\ge 1-\tau/30$. Let $0<\tau'\le \tau/(60|F|^2)$.
Let $\pi: G\rightarrow \Sym(V)$ be a sofic approximation for $G$. Let $\varphi\in \Hom_\mu(\sC_L\vee {}_{L^2, F}\sP, F, \tau', \pi)$. Take a $\sigma': L^2\rightarrow V^V$ such that
\[
\sigma'_tv=\pi_{\lambda(t, A)}v
\]
for all $t\in L^2$, $A\in {}_{L^2, F}\overline{\sP}$ and $v\in \varphi(A)$. Let $\sigma: H\rightarrow \Sym(V)$ be  a sofic approximation for $H$ such that
$\rho_{\Hamm}(\sigma_t, \sigma'_t)\le \tau/5$ for all $t\in L^2$. Then the restriction of $\varphi$ to $\alg(\sC_L)$ lies in $\Hom_\mu(\sC, L, \tau, \sigma)$.
\end{lemma}

\begin{proof}
We have
\begin{align*}
\sum_{A\in \sC_L}|\m(\varphi(A))-\mu(A)|&\le \sum_{A\in (\sC_L\vee {}_{L^2, F}\sP)_F}|\m(\varphi(A))-\mu(A)|\le \tau'\le \tau ,
\end{align*}
while for each $t\in L$ we have
\begin{align*}
\lefteqn{\sum_{A\in \sC}\m(\varphi(tA)\Delta \sigma_t\varphi(A))}\hspace*{15mm}\\
&\le \sum_{A\in \sC}\m(\varphi(tA)\Delta \sigma'_t\varphi(A))+\sum_{A\in \sC}\m(\sigma'_t\varphi(A)\Delta \sigma_t\varphi(A))\\
&\le \sum_{A\in \sC}\sum_{B\in {}_{L^2, F}\overline{\sP}}\m(\varphi(t(A\cap B))\Delta \sigma'_t\varphi(A\cap B))\\
&\hspace*{15mm} \ +\sum_{A\in \sC}\m(\varphi(t(A\setminus X_{L^2, F}))\Delta \sigma_t'\varphi(A\setminus X_{L^2, F}))+2\rho_\Hamm(\sigma'_t, \sigma_t)\\
&\le \sum_{A\in \sC}\sum_{B\in {}_{L^2, F}\overline{\sP}}\m(\varphi(\lambda(t, B)(A\cap B))\Delta \pi_{\lambda(t, B)}\varphi(A\cap B))\\
&\hspace*{15mm} \ +\m(\varphi(t(X\setminus X_{L^2, F})))+\m(\varphi(X\setminus X_{L^2, F}))
+\frac{2\tau}{5}\\
&\le \sum_{g\in F}\sum_{A\in \sC}\sum_{B\in {}_{L^2, F}\overline{\sP}}\m(\varphi(g(A\cap B))\Delta \pi_g\varphi(A\cap B))\\
&\hspace*{15mm} \ +\mu(t(X\setminus X_{L^2, F}))+\mu(X\setminus X_{L^2, F})+2\tau'+\frac{2\tau}{5}\\
&\le |F|\tau'+\frac{\tau}{15}+\frac{\tau}{30}+\frac{2\tau}{5}\le \frac{\tau}{60}+\frac{\tau}{2}< \tau.\qedhere
\end{align*}
\end{proof}

For a finite family $\sC$ of Borel subsets of $X$,  a finite Borel partition $\sP$ of $X$, and $\tau\geq 0$ we write $\sC\overset{\tau}\subseteq \alg(\sP)$ if there is a $B\in \alg(\sP)$ such that $\mu(B)\ge 1-\tau$ and $A\cap B\in \alg(\sP)$ for every $A\in \sC$.

\begin{lemma} \label{L-almost inclusion}
Let $\sC$ be a finite Borel partition of $X$. Let $F\in \cF(G)$ and $L\in \cF(H)$. Then $\sC_F\overset{\tau}\subseteq \alg((\sC\vee {}_{F, L}\sP)_L)$ for $\tau=\mu(X\setminus X_{F, L})|F|$.
\end{lemma}
\begin{proof} Put
$B=\bigcap_{g\in F}gX_{F, L}\in \alg((\sC\vee {}_{F, L}\sP)_F)$. Then $\mu(X\setminus B)\le \tau$.

Let $A\in (\sC\vee {}_{F, L}\sP)_F$. Then $A=\bigcap_{g\in F}gA_g$ for some $A_g\in \sC\vee {}_{F, L}\sP$. If $A_g\not\subseteq X_{F, L}$ for some $g\in F$, then $A\cap B=\emptyset$.
If $A_g\subseteq X_{F, L}$ for all $g\in F$, then $A\subseteq B$ and $A=\bigcap_{g\in F}gA_g=\bigcap_{g\in F}\kappa(g, A_g)A_g\in \alg((\sC\vee {}_{F, L}\sP)_L)$.  Thus
\[
(\sC\vee {}_{F, L}\sP)_F\overset{\tau}\subseteq \alg((\sC\vee {}_{F, L}\sP)_L).
\]
Since $\sC_F\subseteq \alg((\sC\vee {}_{F, L}\sP)_F)$, we conclude that $\sC_F\overset{\tau}\subseteq \alg((\sC\vee {}_{F, L}\sP)_L)$.
\end{proof}

\begin{lemma} \label{L-Shannon separation}
Let $\sC$ be a finite Borel partition of $X$. Let $S\in \overline{\cF}(G)$ and let $W$ be a Borel subset of $X$ such that $SW=X$.
Let $L$ be a set in $\cF(H)$ containing $e_H$.
Set $\sC'=\{W, X\setminus W\}$ and $\sC''=\sC\vee \sC'\vee {}_{S^2, L}\sP$.
Then  there are a finite Borel partition $\sQ$ of $W$ contained in $\alg((\sC\vee \sC')_{S^2})$ and a map $\Theta: \sQ\rightarrow \cF(S)$ such that
$e_G\in \Theta(B)$ for every $B\in \sQ$ and
the sets $gB$ for $B\in \sQ$ and $g\in \Theta(B)$ form a partition $\sR$ of $X$ finer than $\sC$. Furthermore, for any such $\sQ$ and $\sR$ and
 any sofic approximations $\pi: G\rightarrow \Sym(V)$, $\sigma: H\rightarrow \Sym(V)$, any $\delta, \delta'>0$, any $\varphi, \psi\in \Hom_\mu(\sR, S, \delta', \pi)$, and any $\tilde{\varphi}, \tilde{\psi}\in \Hom_\mu(\sR\vee \sC'', L, \delta, \sigma)$ satisfying
\begin{enumerate}
\item $\tilde{\varphi}(B)=\varphi(B)$ and $\tilde{\psi}(B)=\psi(B)$  for all $B\in \sQ$,
\item $\varphi(W)=\psi(W)$
\end{enumerate}
one has
\[
\rho_{\sC}(\varphi, \psi)\le  2(\delta+\delta')|S|+2\delta|S|\cdot |L|+2|S|^3\mu(X\setminus X_{S^2, L})+|S|\cdot |L|\rho_{\sC\vee {}_{S^2, L}\sP}(\tilde{\varphi}, \tilde{\psi}).
\]
\end{lemma}

\begin{proof}
We prove the existence of $\sQ$ first.
Note that $(\sC')_S$ is the partition of $X$ generated by $gW$ for $g\in S$.
Since $SW=X$, every member of $(\sC')_S$ is contained in $gW$ for some $g\in S$.
Each member $A$ of $\sC\vee (\sC')_S$
is contained in some member of $(\sC')_S$, and hence is contained in $g_A^{-1}W$ for some $g_A\in S$. We shall choose $g_A=e_G$ when $A\subseteq W$. Denote by $\sQ$ the partition of $W$ generated by $g_AA$ for $A\in \sC\vee (\sC')_S$. Then $\sQ\subseteq \alg((\sC\vee (\sC')_S)_S)$. Note that if $A$ is a member of $\sC\vee (\sC')_S$
then it can be written as $g_A^{-1}(g_AA)$ and hence is the disjoint union of sets of the form
$g_A^{-1}B$ with $B$ belonging to $\sQ$.
Thus we can find a map $\Theta: \sQ\rightarrow \cF(S)$ such that $e_G\in \Theta(B)$ for every $B\in \sQ$ and such that the sets $gB$ for $B\in \sQ$ and $g\in \Theta(B)$ form a partition $\sR$ of $X$ finer than $\sC\vee (\sC')_S$. Note that
\[
(\sC\vee (\sC')_S)_S\preceq ((\sC\vee \sC')_S)_S=(\sC\vee \sC')_{S^2}.
\]
Thus
\[
\sQ\subseteq \alg((\sC\vee (\sC')_S)_S)\subseteq \alg((\sC\vee \sC')_{S^2}).
\]

Now let $\sQ, \sR, \pi, \sigma, \delta, \delta', \varphi, \psi, \tilde{\varphi}, \tilde{\psi}$ be as in the lemma statement.
For any partitions $\sC_1$ and $\sC_2$ of $X$ coarser than $(\sR\vee\sC'')_L$, we have
\begin{align*}
\rho_{\sC_1\vee \sC_2}(\tilde{\varphi}, \tilde{\psi})&=\sum_{A_1\in \sC_1, \, A_2\in \sC_2}\m((\tilde{\varphi}(A_1)\cap \tilde{\varphi}(A_2))\Delta (\tilde{\psi}(A_1)\cap \tilde{\psi}(A_2)))\\
&\le \sum_{A_1\in \sC_1, \, A_2\in \sC_2}\m((\tilde{\varphi}(A_1)\cap \tilde{\varphi}(A_2))\Delta (\tilde{\psi}(A_1)\cap \tilde{\varphi}(A_2))\\
&\hspace*{15mm} \ +\sum_{A_1\in \sC_1, \, A_2\in \sC_2}\m((\tilde{\psi}(A_1)\cap \tilde{\varphi}(A_2))\Delta (\tilde{\psi}(A_1)\cap \tilde{\psi}(A_2)))\\
&=\rho_{\sC_1}(\tilde{\varphi}, \tilde{\psi})+\rho_{\sC_2}(\tilde{\varphi}, \tilde{\psi}).
\end{align*}

For each $t\in L$, we have
\begin{align} \label{E-Shannon sep}
\rho_{t\sC''}(\tilde{\varphi}, \tilde{\psi})&=\rho_{\sC''}(\tilde{\varphi}\circ t, \tilde{\psi}\circ t)\\
&\le  \rho_{\sC''}(\tilde{\varphi}\circ t, \sigma_t\circ \tilde{\varphi})+\rho_{\sC''}(\sigma_t\circ \tilde{\varphi}, \sigma_t\circ \tilde{\psi})+\rho_{\sC''}(\sigma_t\circ \tilde{\psi}, \tilde{\psi}\circ t)\nonumber \\
&\le 2\delta+\rho_{\sC''}(\tilde{\varphi}, \tilde{\psi}). \nonumber
\end{align}
By Lemma~\ref{L-almost inclusion} we have $(\sC\vee \sC')_{S^2}\overset{\tau}\subseteq \alg((\sC'')_L)$ for $\tau=\mu(X\setminus X_{S^2, L})|S^2|$.
Then $\sQ\overset{\tau}\subseteq \alg((\sC'')_L)$.
Thus
\begin{align*}
\rho_{\sQ}(\tilde{\varphi}, \tilde{\psi})
&\le \rho_{(\sC'')_L}(\tilde{\varphi}, \tilde{\psi})+2\tau+2\delta\\
&\le \sum_{t\in L}\rho_{t\sC''}(\tilde{\varphi}, \tilde{\psi})+2\tau+2\delta\\
&\overset{\eqref{E-Shannon sep}}\le 2\delta|L|+|L|\rho_{\sC''}(\tilde{\varphi}, \tilde{\psi})+2\tau+2\delta\\
&\le 2\delta|L|+|L|\rho_{\sC\vee {}_{S^2, L}\sP}(\tilde{\varphi}, \tilde{\psi})+|L|\rho_{\sC'}(\tilde{\varphi}, \tilde{\psi})+2\tau+2\delta\\
&=2\delta|L|+|L|\rho_{\sC\vee {}_{S^2, L}\sP}(\tilde{\varphi}, \tilde{\psi})+2\tau+2\delta.
\end{align*}

As in \eqref{E-Shannon sep}, for each $g\in S$ we have
\begin{align*}
\rho_{g\sQ}(\varphi, \psi)\le 2\delta'+\rho_{\sQ}(\varphi, \psi).
\end{align*}
Thus
\begin{align*}
\rho_{\sC}(\varphi, \psi)&\le \rho_{\sR}(\varphi, \psi)\\
&\le \sum_{g\in S}\rho_{g\sQ}(\varphi, \psi)\\
&\le 2\delta'|S|+ |S|\rho_{\sQ}(\varphi, \psi)\\
&=2\delta'|S|+ |S|\rho_{\sQ}(\tilde{\varphi}, \tilde{\psi})\\
&\le 2\delta'|S|+2\delta|S|\cdot |L|+|S|\cdot |L|\rho_{\sC\vee {}_{S^2, L}\sP}(\tilde{\varphi}, \tilde{\psi})+2|S|\tau+2|S|\delta\\
&\le 2(\delta+\delta')|S|+2\delta|S|\cdot |L|+2|S|^3\mu(X\setminus X_{S^2, L})+|S|\cdot |L|\rho_{\sC\vee {}_{S^2, L}\sP}(\tilde{\varphi}, \tilde{\psi}).\qedhere
\end{align*}
\end{proof}

\begin{lemma} \label{L-Shannon partition}
Let $S$ be a set in $\cF(G)$ containing $e_G$ and let $W$ be a Borel subset of $X$ such that $SW=X$. Let $\sQ$ be a finite Borel partition of $W$ and $\Theta$  a function $\sQ\rightarrow \cF(S)$ such that $e_G\in \Theta(B)$ for every $B\in \sQ$ and such that the sets $gB$ for $B\in \sQ$ and $g\in \Theta(B)$ form a partition of $X$. Also, let $\sD$ be a finite Borel partition of $X$. Then there is a finite Borel partition $\sQ_1$ of $W$
satisfying the following conditions:
\begin{enumerate}
\item $\sQ\preceq \sQ_1$,
\item defining $\Theta(D)=\Theta(B)$ for $D\in \sQ_1$ and $B\in \sQ$ satisfying $D\subseteq B$ and denoting by $\sR_1$ the partition of $X$ consisting of $gD$ for $D\in \sQ_1$ and $g\in \Theta(D)$, one has $\sD\preceq \sR_1$.
\end{enumerate}
\end{lemma}

\begin{proof}
Denote by $\sQ_1$ the partition of $W$ generated by the sets $B\cap g^{-1}D$ for $B\in \sQ$, $g\in \Theta(B)$, and $D\in \sD$.
It is easily checked that  $\sQ_1$ satisfies the conditions.
\end{proof}

\begin{lemma} \label{L-Shannon hom}
Let $\sC$ be a Borel finite partition of $X$, $L\in \overline{\cF}(H)$, and $0<\delta<1$.
Let $S\in \overline{\cF}(G)$  and let $W$ be a Borel subset of $X$ such that $SW=X$. Take an  $L^\bullet\in \overline{\cF}(H)$ such that
$\mu(X\setminus X_{S, L^\bullet})\le \delta/(20|S|)$ and an $\overline{L}\in \overline{\cF}(H)$ containing $L^\bullet$. Let $\sQ_1\preceq\sQ_2$ be finite Borel partitions of $W$ and
let $\Theta: \sQ_1\cup \sQ_2\rightarrow \cF(S)$ be such that
\begin{enumerate}
\item $e_G\in \Theta(B_2)=\Theta(B_1)$ for all $B_2\in \sQ_2$ and $B_1\in \sQ_1$ with $B_2\subseteq B_1$,

\item for $j=1,2 $ the sets $gB$ for $B\in \sQ_j$ and $g\in \Theta(B)$ form a partition $\sR_j$ of $X$,

\item $\sR_1\succeq \sD:=\sC_L\vee {}_{S, \overline{L}}\sP$, and

\item $\sR_2\succeq (\sR_1)_{L\overline{L}}$.
\end{enumerate}
Denote by $\sQ_1'$ the set of all $B\in \sQ_1$ satisfying $B\subseteq X_{S, L^\bullet}$ and by $\sQ_2'$ the set of all $B\in \sQ_2$ satisfying $B\subseteq X_{S, \overline{L}}$. Denote by $\Lambda$ the set consisting of all $(B, g)$ for $B\in \sQ_2'$ and $g\in \Theta(B)\setminus \{e_G\}$.
Take $0<\delta'\le \delta/(20|S|)$.
Let $\sA$ be a finite Borel partition of $X$ refining $\sR_2$. Let $V$ be a nonempty finite set and
$\varphi$ a homomorphism $\alg(\sA)\rightarrow \Pb_V$ such that $\sum_{A\in \sA}|\m(\varphi(A))-\mu(A)|\le \delta'$.
Let $0<\bar{\tau}\le \delta/(20|L^\bullet|)$.
Let $\sigma: H\rightarrow \Sym(V)$ be an $(L\cup L^\bullet, \bar{\tau})$-approximation for $H$.
Define a map $\tilde{\varphi}': \sR_2\rightarrow \Pb_V$ by
$\tilde{\varphi}'(B)=\varphi(B)$ for all $B\in \sQ_2$, $\tilde{\varphi}'(gB)=\emptyset$ for all $B\in \sQ_2\setminus \sQ_2'$ and $g\in \Theta(B)\setminus \{e_G\}$, and
\[
\tilde{\varphi}'(gB)=\sigma_{\kappa(g, B)}\varphi(B)
\]
for all $(B, g)\in \Lambda$, and  extend $\tilde{\varphi}'$ to a map $\alg(\sR_2)\rightarrow \Pb_V$ by setting $\tilde{\varphi}'(D)=\bigcup_{A\in \sR_2, A\subseteq D}\tilde{\varphi}'(A)$ for $D\in \alg(\sR_2)$.
Suppose that
\begin{align} \label{E-Shannon disjoint1}
\sum_{(B, g)\in \Lambda}\m(\tilde{\varphi}'(gB)\cap \varphi(W))\le \frac{\delta}{40},
\end{align}
\begin{align} \label{E-Shannon disjoint2}
\sum_{(B, g), (B', g')\in \Lambda, (B, g)\neq (B', g')}\m(\tilde{\varphi}'(gB)\cap \tilde{\varphi}'(g'B'))\le \frac{\delta}{40},
\end{align}
and
\begin{align} \label{E-Shannon hom}
\sum_{B\in \sQ_1'}\sum_{t\in LL^\bullet}\m(\tilde{\varphi}'(tB)\Delta \sigma_t\varphi(B))\le \frac{\delta}{20}.
\end{align}
Then there is a homomorphism $\tilde{\varphi}: \alg(\sR_2)\rightarrow \Pb_V$ such that
$\tilde{\varphi}(B)=\varphi(B)$ for every $B\in \sQ_2$ and $\sum_{A\in \sR_2}\m(\tilde{\varphi}(A)\Delta \tilde{\varphi}'(A))\le \delta/5$.
Furthermore, the restriction of any such $\tilde{\varphi}$ to $\alg(\sC_L)$ lies in  $\Hom_\mu(\sC, L, \delta, \sigma)$.
\end{lemma}

\begin{proof}
Note that
\begin{align} \label{E-Shannon complement}
\mu\bigg(X\setminus \bigg(W\cup \bigcup_{(B, g)\in \Lambda}gB\bigg)\bigg)
&=\mu\bigg(\bigcup_{B\in \sQ_2\setminus \sQ_2'}\,\bigcup_{g\in \Theta(B)\setminus \{e_G\}}gB\bigg)\\
\nonumber &\le (|S|-1)\mu(X\setminus X_{S, L^\bullet})\le \frac{\delta}{20}.
\end{align}

We prove the existence of $\tilde{\varphi}$ first. If $W=X$, then we may take $\tilde{\varphi}(A)=\varphi(A)$ for all $A\in \alg(\sR_2)$.  Thus we may assume that $W\neq X$. Consider the case that $\Lambda$ is nonempty. List the elements of $\Lambda$ as $(B_1, g_1), \dots, (B_{|\Lambda|}, g_{|\Lambda|})$. We set
$ \tilde{\varphi}(B)=\varphi(B)$ for all $B\in \sQ_2$, $\tilde{\varphi}(gB)=\emptyset$ for all $B\in \sQ_2\setminus \sQ_2'$ and $g\in \Theta(B)\setminus \{e_G\}$,
\[
\tilde{\varphi}(g_kB_k)=\tilde{\varphi}'(g_kB_k)\setminus \bigg(\varphi(W)\cup \bigcup_{j=1}^{k-1}\tilde{\varphi}'(g_jB_j)\bigg)
\]
for all $1\le k<|\Lambda|$, and
\[
\tilde{\varphi}(g_{|\Lambda|}B_{|\Lambda|})=V\setminus \bigg(\varphi(W)\cup \bigcup_{j=1}^{|\Lambda|-1}\tilde{\varphi}'(g_jB_j)\bigg).
\]
Then the sets $\tilde{\varphi}(A)$ for $A\in \sR_2$ form a partition of $V$, and hence $\tilde{\varphi}$ extends uniquely to a homomorphism $\alg(\sR_2)\rightarrow \Pb_V$. We have
\begin{align*}
\m(\varphi(W))+\sum_{j=1}^{|\Lambda|}\m(\tilde{\varphi}'(g_jB_j))&=\m(\varphi(W))+\sum_{j=1}^{|\Lambda|}\m(\varphi(B_j))\\
&=\m(\varphi(W))+\sum_{g\in S\setminus \{e_G\}}\sum_{1\le j\le |\Lambda|, g_j=g}\m(\varphi(B_j))\\
&=\m(\varphi(W))+\sum_{g\in S\setminus \{e_G\}}\m\bigg(\varphi\bigg(\bigcup_{1\le j\le |\Lambda|, g_j=g}B_j\bigg)\bigg)\\
&\ge \mu(W)-\delta'+\sum_{g\in S\setminus \{e_G\}}\bigg(\mu\bigg(\bigcup_{1\le j\le |\Lambda|, g_j=g}B_j\bigg)-\delta'\bigg)\\
&=\mu\bigg(W\cup \bigcup_{(B, g)\in \Lambda}gB\bigg)-|S|\delta'\\
&\overset{\eqref{E-Shannon complement}}{\ge} 1-\frac{\delta}{20}-|S|\delta'\ge 1-\frac{\delta}{10},
\end{align*}
and hence
\begin{align*}
\lefteqn{\m\bigg(V\setminus \bigg(\varphi(W)\cup \bigcup_{j=1}^{|\Lambda|}\tilde{\varphi}'(g_jB_j)\bigg)\bigg)}\hspace*{25mm} \\
&=1-\m\bigg(\varphi(W)\cup \bigcup_{j=1}^{|\Lambda|}\tilde{\varphi}'(g_jB_j)\bigg)\\
&\le \frac{\delta}{10}+\m(\varphi(W))+\sum_{j=1}^{|\Lambda|}\m(\tilde{\varphi}'(g_jB_j))-\m\bigg(\varphi(W)\cup \bigcup_{j=1}^{|\Lambda|}\tilde{\varphi}'(g_jB_j)\bigg)\\
&=\frac{\delta}{10}+\sum_{k=1}^{|\Lambda|}\m\bigg(\tilde{\varphi}'(g_kB_k)\cap \bigg(\varphi(W)\cup \bigcup_{j=1}^{k-1}\tilde{\varphi}'(g_jB_j)\bigg)\bigg).
\end{align*}
Therefore
\begin{align*}
\sum_{A\in \sR_2}\m(\tilde{\varphi}(A)\Delta \tilde{\varphi}'(A))&=\sum_{k=1}^{|\Lambda|}\m(\tilde{\varphi}(g_kB_k)\Delta \tilde{\varphi}'(g_kB_k))\\
&=\m\bigg(V\setminus \bigg(\varphi(W)\cup \bigcup_{j=1}^{|\Lambda|}\tilde{\varphi}'(g_jB_j)\bigg)\bigg)\\
&\hspace*{15mm} \ +\sum_{k=1}^{|\Lambda|}\m\bigg(\tilde{\varphi}'(g_kB_k)\cap \bigg(\varphi(W)\cup \bigcup_{j=1}^{k-1}\tilde{\varphi}'(g_jB_j)\bigg)\bigg)\\
&\le \frac{\delta}{10}+2\sum_{k=1}^{|\Lambda|}\m\bigg(\tilde{\varphi}'(g_kB_k)\cap \bigg(\varphi(W)\cup \bigcup_{j=1}^{k-1}\tilde{\varphi}'(g_jB_j)\bigg)\bigg)\\
&\le \frac{\delta}{10}+2\sum_{k=1}^{|\Lambda|}\m\big(\tilde{\varphi}'(g_kB_k)\cap \varphi(W))\\&\hspace*{15mm} \ +2\sum_{1\le j<k\le |\Lambda|}\m(\tilde{\varphi}'(g_kB_k)\cap \tilde{\varphi}'(g_jB_j))\\
&\overset{\eqref{E-Shannon disjoint1}, \eqref{E-Shannon disjoint2}}\le \frac{\delta}{10}+\frac{\delta}{20}+\frac{\delta}{20}=\frac{\delta}{5}.
\end{align*}
This proves the existence of $\tilde{\varphi}$ when $\Lambda$ is nonempty. Next consider the case that $W\neq X$ and $\Lambda=\emptyset$. Choose a $B_0\in \sQ_2\setminus \sQ_2'$ and a $g_0\in \Phi(B)\setminus \{e_G\}$. Set $\tilde{\varphi}(B)=\varphi(B)$ for all $B\in \sQ_2$, $\tilde{\varphi}(g_0B_0)=V\setminus \varphi(W)$, and $\tilde{\varphi}(gB)=\emptyset$ for all $B\in \sQ_2\setminus \sQ_2'$ and $g\in \Phi(B)\setminus \{e_G\}$ satisfying $(B, g)\neq (B_0, g_0)$. Then
the sets $\tilde{\varphi}(A)$ for $A\in \sR_2$ form a partition of $V$, and hence $\tilde{\varphi}$ extends uniquely to a homomorphism $\alg(\sR_2)\rightarrow \Pb_V$. We have
\begin{align*}
\sum_{A\in \sR_2}\m(\tilde{\varphi}(A)\Delta \tilde{\varphi}'(A))
&=\m(\tilde{\varphi}(g_0B_0))\\
&=\m(V\setminus \varphi(W))\\
&\le \mu(X\setminus W)+\delta'\\
&\overset{\eqref{E-Shannon complement}}{\le} \frac{\delta}{20}+\frac{\delta}{20}<\frac{\delta}{5}.
\end{align*}
This proves the existence of $\tilde{\varphi}$.

Now let $\tilde{\varphi}$ be any homomorphism $\alg(\sR_2)\rightarrow \Pb_V$ satisfying
$\sum_{A\in \sR_2}\m(\tilde{\varphi}(A)\Delta \tilde{\varphi}'(A))\le \delta/5$.
We have
\begin{align} \label{E-Shannon hom20}
\lefteqn{\sum_{A\in \sR_2}|\m(\tilde{\varphi}'(A))-\mu(A)|}\hspace*{10mm} \\
&= \sum_{B\in \sQ_2, g\in \Theta(B)}|\m(\tilde{\varphi}'(gB))-\mu(gB)| \nonumber\\
&\le \sum_{B\in \sQ_2}|\m(\varphi(B))-\mu(B)|+\sum_{(B, g)\in \Lambda}|\m(\varphi(B))-\mu(B)|+\sum_{\substack{B\in \sQ_2\setminus \sQ_2' ,\\ g\in \Theta(B)\setminus \{e_G\}}}\mu(B) \nonumber\\
&\le \delta'+\sum_{g\in S\setminus \{e_G\}}\sum_{1\le j\le |\Lambda|, g_j=g}|\m(\varphi(B_j))-\mu(B_j)|+(|S|-1)\mu(X\setminus X_{S, L^\bullet})\nonumber\\
&\le |S|\delta'+\frac{\delta}{20}\le \frac{\delta}{20}+\frac{\delta}{20}=\frac{\delta}{10}, \nonumber
\end{align}
and hence
\begin{align*}
\sum_{A\in \sC_L}|\m(\tilde{\varphi}(A))-\mu(A)|&\le \sum_{A\in \sR_2}|\m(\tilde{\varphi}(A))-\mu(A)|\\
&\le \sum_{A\in \sR_2}|\m(\tilde{\varphi}(A))-\m(\tilde{\varphi}'(A))|+\sum_{A\in \sR_2}|\m(\tilde{\varphi}'(A))-\mu(A)|\\
&\overset{\eqref{E-Shannon hom20}}\le  \sum_{A\in \sR_2}\m(\tilde{\varphi}(A)\Delta \tilde{\varphi}'(A))+\frac{\delta}{10}\le \frac{\delta}{5}+\frac{\delta}{10}<\delta.
\end{align*}
Let $t\in L$.
Since $\sR_1$ refines $\sC_L$, it refines $\sC$ and $t\sC$. Thus
\begin{align} \label{E-Shannon hom9}
\sum_{A\in \sC}\m(\tilde{\varphi}(A)\Delta \tilde{\varphi}'(A))\le \sum_{A\in \sR_2}\m(\tilde{\varphi}(A)\Delta \tilde{\varphi}'(A))\le \frac{\delta}{5}
\end{align}
and
\begin{align} \label{E-Shannon hom10}
\sum_{A\in \sC}\m(\tilde{\varphi}(tA)\Delta \tilde{\varphi}'(tA))\le \sum_{A\in \sR_2}\m(\tilde{\varphi}(A)\Delta \tilde{\varphi}'(A))\le \frac{\delta}{5}.
\end{align}
For every $B\in \sQ_1'$ and $g\in \Theta(B)\setminus \{e_G\}$, we have
\begin{align} \label{E-Shannon disjoint4}
\tilde{\varphi}'(gB)=\tilde{\varphi}'\bigg(\bigcup_{B'\in \sQ_2', B'\subseteq B}gB'\bigg)
&=\bigcup_{B'\in \sQ_2', B'\subseteq B}\tilde{\varphi}'(gB')\\
&=\bigcup_{B'\in \sQ_2', B'\subseteq B}\sigma_{\kappa(g, B)}\varphi(B') \nonumber \\
&=\sigma_{\kappa(g, B)}\varphi(B). \nonumber
\end{align}
Note that
\begin{align} \label{E-Shannon hom21}
\lefteqn{\sum_{\substack{B\in \sQ_1\setminus \sQ_1',\\ g\in \Theta(B)}}\m(\tilde{\varphi}'(t gB)\Delta \sigma_t\tilde{\varphi}'(gB))}\hspace*{10mm}\\
&\le \sum_{\substack{B\in \sQ_1\setminus \sQ_1',\\ g\in \Theta(B)}}\m(\tilde{\varphi}'(t gB))+\sum_{B\in \sQ_1\setminus \sQ_1'}\m(\varphi(B))+\sum_{\substack{B\in \sQ_1\setminus \sQ_1', B\subseteq X_{S, \overline{L}},\\ g\in \Theta(B)\setminus \{e_G\}}}\m(\sigma_{\kappa(g, B)}\varphi(B))\nonumber \\
&\overset{\eqref{E-Shannon hom20}}\le \frac{\delta}{10}+\sum_{\substack{B\in \sQ_1\setminus \sQ_1',\\ g\in \Theta(B)}}\mu(t gB)+\sum_{B\in \sQ_1\setminus \sQ_1'}\m(\varphi(B))+\sum_{\substack{B\in \sQ_1\setminus \sQ_1', B\subseteq X_{S, \overline{L}},\\ g\in \Theta(B)\setminus \{e_G\}}}\m(\varphi(B))\nonumber\\
&\le \frac{\delta}{10}+|S|\mu(X\setminus X_{S, L^\bullet})+|S|\m(\varphi(X\setminus X_{S, L^\bullet}))\nonumber\\
&\le \frac{\delta}{10}+|S|\mu(X\setminus X_{S, L^\bullet})+|S|\mu(X\setminus X_{S, L^\bullet})+|S|\delta'\nonumber\\
&\le \frac{\delta}{10}+\frac{\delta}{20}+\frac{\delta}{20}+\frac{\delta}{20}=\frac{\delta}{4}. \nonumber
\end{align}
Thus
\begin{align*}
\lefteqn{\sum_{A\in \sC}\m(\tilde{\varphi}'(tA)\Delta \sigma_t\tilde{\varphi}'(A))}\hspace*{10mm}\\
&=\sum_{A\in \sC}\m\bigg(\bigg(\bigcup_{D\in \sR_1, D\subseteq A}\tilde{\varphi}'(tD)\bigg)\Delta \bigg(\bigcup_{D\in \sR_1, D\subseteq A}\sigma_t\tilde{\varphi}'(D)\bigg)\bigg) \\
&\le \sum_{A\in \sC}\sum_{D\in \sR_1, D\subseteq A}\m(\tilde{\varphi}'(tD)\Delta \sigma_t\tilde{\varphi}'(D))\\
&=\sum_{D\in \sR_1}\m(\tilde{\varphi}'(tD)\Delta \sigma_t\tilde{\varphi}'(D))\\
&= \sum_{B\in \sQ_1, g\in \Theta(B)}\m(\tilde{\varphi}'(tgB)\Delta \sigma_t\tilde{\varphi}'(gB))\\
&= \sum_{B\in \sQ_1', g\in \Theta(B)}\m(\tilde{\varphi}'(t\kappa(g, B)B)\Delta \sigma_t\tilde{\varphi}'(gB))+\sum_{\substack{B\in \sQ_1\setminus \sQ_1',\\ g\in \Theta(B)}}\m(\tilde{\varphi}'(t gB)\Delta \sigma_t\tilde{\varphi}'(gB))\\
&\overset{\eqref{E-Shannon hom21}}\le \sum_{B\in \sQ_1', g\in \Theta(B)}\m(\tilde{\varphi}'(t\kappa(g, B)B)\Delta \sigma_{t\kappa(g, B)}\varphi(B))\\
&\hspace*{15mm} \ +\sum_{B\in \sQ_1', g\in \Theta(B)}\m(\sigma_{t\kappa(g, B)}\varphi(B)\Delta \sigma_t\sigma_{\kappa(g, B)}\varphi(B))\\
&\hspace*{15mm} \ +\sum_{B\in \sQ_1', g\in \Theta(B)}\m(\sigma_{t}\sigma_{\kappa(g, B)}\varphi(B)\Delta \sigma_t\tilde{\varphi}'(gB))+\frac{\delta}{4}\\
&\overset{\eqref{E-Shannon disjoint4}}\le \sum_{B\in \sQ_1', t_1\in L L^\bullet}\m(\tilde{\varphi}'(t_1B)\Delta \sigma_{t_1}\varphi(B))+2\sum_{t_1\in L^\bullet}\rho_\Hamm(\sigma_{tt_1}, \sigma_t\sigma_{t_1})\\
&\hspace*{20mm} \ +\sum_{B\in \sQ_1'}\m(\sigma_t\sigma_{e_H}\varphi(B)\Delta \sigma_t\varphi(B))+\frac{\delta}{4}\\
&\overset{\eqref{E-Shannon hom}}\le \frac{\delta}{20}+2\bar{\tau}|L^\bullet|+\sum_{B\in \sQ_1'}\m(\sigma_{e_H}\varphi(B)\Delta \varphi(B))+\frac{\delta}{4}\\
&\le \frac{\delta}{20}+2\bar{\tau}|L^\bullet|+2\rho_{\Hamm}(\sigma_{e_H}, \id_V)+\frac{\delta}{4}\\
&= \frac{\delta}{20}+2\bar{\tau}|L^\bullet|+2\rho_{\Hamm}(\sigma_{e_H}\sigma_{e_H}, \sigma_{e_H})+\frac{\delta}{4}\\
&\le \frac{\delta}{20}+4\bar{\tau}|L^\bullet|+\frac{\delta}{4}\le \frac{\delta}{20}+\frac{\delta}{5}
+\frac{\delta}{4}=\frac{\delta}{2},
\end{align*}
and hence
\begin{align*}
\sum_{A\in \sC}\m(\tilde{\varphi}(tA)\Delta \sigma_t\tilde{\varphi}(A))
&\le  \sum_{A\in \sC}\m(\tilde{\varphi}(tA)\Delta \tilde{\varphi}'(tA))+ \sum_{A\in \sC}\m(\tilde{\varphi}'(tA)\Delta \sigma_t\tilde{\varphi}'(A))\\
&\hspace*{20mm} \ + \sum_{A\in \sC}\m(\sigma_t\tilde{\varphi}'(A)\Delta \sigma_t\tilde{\varphi}(A))\\
&\overset{\eqref{E-Shannon hom10}}\le \frac{\delta}{5}+\frac{\delta}{2}+\sum_{A\in \sC}\m(\tilde{\varphi}'(A)\Delta \tilde{\varphi}(A))\\
&\overset{\eqref{E-Shannon hom9}}\le \frac{\delta}{5}+\frac{\delta}{2}+\frac{\delta}{5}<\delta.
\end{align*}
Therefore the restriction of $\tilde{\varphi}$ to $\alg(\sC_L)$ lies in $\Hom_\mu(\sC, L, \delta, \sigma)$.
\end{proof}

\begin{lemma} \label{L-Shannon disjoint}
Let $\sC$ be a Borel finite partition of $X$, $L\in \overline{\cF}(H)$, and $0<\delta<1$.
Let $S\in \overline{\cF}(G)$ and $L^\bullet, \overline{L}\in \overline{\cF}(H)$ be such that
$L^\bullet\subseteq \overline{L}$ and
\[
\mu(X\setminus X_{S, \overline{L}})\le \gamma:=\delta/(200 |S|\cdot |LL^\bullet|).
\]
Set $L^\sharp:=\overline{L} L \overline{L}\in \overline{\cF}(H)$ and let $T\in \overline{\cF}(G)$ be such that
\[
\mu(X\setminus X_{L^\sharp, T})\le \kappa:=\delta/(10^4|\overline{L}|^2\cdot |LL^\bullet|).
\]
Let $C\in \Nb$ and $S_1, \dots, S_n\in \overline{\cF}(G)$, and let $W, \cV_1, \dots, \cV_n$ be Borel subsets of $X$ such that $SW=X$,
and for every $g\in T$ and $w\in W\cap g^{-1}W$ the points $w$ and $gw$ are connected by a path of length at most $C$ in which each edge is an $S_j$-edge with both endpoints in $\cV_j$ for some $1\le j\le n$.
Let $L^\dag\in \overline{\cF}(H)$ be such that
\[
\mu(X\setminus X_{\bigcup_{j=1}^nS_j, L^\dag})\le \zeta:=\kappa/(100|\textstyle\bigcup_{j=1}^nS_j|^C).
\]
Let $0\le \bar{\tau}\le \kappa/(10C|\bigcup_{j=1}^nS_j|^{C}\cdot |L^\dag|^{3C})$.
Let $F\in \overline{\cF}(G)$ be such that
$T\cup (\bigcup_{j=1}^nS_j)^C\subseteq F$ and  $\mu(X\setminus X_{(L^\sharp\cup (L^\dag)^C)^2, F})\le \min\{\zeta, \bar{\tau}/30\}$.
Denote by $\sD'$ the partition of $X$ generated by $W, \cV_1, \dots, \cV_n$.
Let $\sQ_1\preceq\sQ_2$ be finite Borel partitions of $W$ and let
$\Theta: \sQ_1\cup \sQ_2\rightarrow \cF(S)$ be such that
\begin{enumerate}
\item $e_G\in \Theta(B_2)=\Theta(B_1)$ for all $B_2\in \sQ_2$ and $B_1\in \sQ_1$ with $B_2\subseteq B_1$,

\item for $i=1,2 $ the sets $gB$ for $B\in \sQ_i$ and $g\in \Theta(B)$ form a partition $\sR_i$ of $X$,

\item $\sR_1\succeq \sD:=\sC_L\vee (\sD')_{T}\vee (\bigvee_{j=1}^n{}_{S_j, L^\dag}\sP)\vee {}_{\bigcup_{j=1}^nS_j, L^\dag}\sP\vee {}_{L^\sharp, T}\sP$, and

\item $\sR_2\succeq (\sR_1)_{L \overline{L}}$.
\end{enumerate}
Denote by $\sQ_1'$ the set of all $B\in \sQ_1$ satisfying $B\subseteq X_{S, L^\bullet}$, and denote by $\sQ_2'$ the set of all $B\in \sQ_2$ satisfying $B\subseteq X_{S, \overline{L}}$.
Denote by $\Lambda$ the set consisting of all $(B, g)$ for $B\in \sQ_2'$ and $g\in \Theta(B)\setminus \{e_G\}$.
Let $0<\tau'\le \min\{\kappa/(100|F|^3), \bar{\tau}/(60|F|^2)\}$ and  $0<\delta'\le \min\{\tau', \kappa/(10n|F|), \delta/(50 |LL^\bullet|(|S|+1))\}$.
Let $\sA$ be a finite Borel partition of $X$ refining $(\sR_2)_{(L^\sharp\cup (L^\dag)^C)^2}\vee {}_{(L^\sharp\cup (L^\dag)^C)^2, F}\sP\vee \sD_{(\bigcup_{j=1}^nS_j)^C}$.
Let $\pi: G\rightarrow \Sym(V)$
be an $(F, \tau')$-approximation for $G$
and
$\varphi, \varphi_0\in \Hom_\mu(\sA, F, \delta', \pi)$.
Let $\sigma: H\rightarrow \Sym(V)$ be an $(L^\sharp\cup (L^\dag)^C, \bar{\tau})$-approximation for $H$
such that $\rho_{\Hamm}(\sigma_t, \sigma'_t)\le \bar{\tau}/5$ for all $t\in (L^\sharp\cup (L^\dag)^C)^2$, where $\sigma'_t\in V^V$ for $t\in (L^\sharp\cup (L^\dag)^C)^2$ satisfies
\[
\sigma'_tv=\pi_{\lambda(t, A)}v
\]
for all $A\in {}_{(L^\sharp\cup (L^\dag)^C)^2, F}\overline{\sP}$ and $v\in \varphi_0(A)$.
Define $\tilde{\varphi}': \sR_2\rightarrow \Pb_V$ by
$\tilde{\varphi}'(B)=\varphi(B)$ for all $B\in \sQ_2$, $\tilde{\varphi}'(gB)=\emptyset$ for all $B\in \sQ_2\setminus \sQ_2'$ and $g\in \Theta(B)\setminus \{e_G\}$, and
\[
\tilde{\varphi}'(gB)=\sigma_{\kappa(g, B)}\varphi(B)
\]
for all $(B, g)\in \Lambda$, and  extend $\tilde{\varphi}'$ to a map $\alg(\sR_2)\rightarrow \Pb_V$ by setting $\tilde{\varphi}'(D)=\bigcup_{A\in \sR_2, A\subseteq D}\tilde{\varphi}'(A)$ for $D\in \alg(\sR_2)$.
Assume that $\varphi(W)=\varphi_0(W)$ and $\varphi(\cV_j\cap D)=\varphi_0(\cV_j\cap D)$ for all $1\le j\le n$ and $D\in {}_{S_j, L^\dag}\sP$.
Then \eqref{E-Shannon disjoint1}, \eqref{E-Shannon disjoint2}, and \eqref{E-Shannon hom} hold.
\end{lemma}

\begin{proof}
Denote by $V_1$ the set of all $v\in V$ satisfying $\pi_{g_1g_2}v\neq \pi_{g_1}\pi_{g_2}v$ for some $g_1, g_2\in (\bigcup_{j=1}^nS_j)^C$. Then
$\m(V_1)\le |F|^2\tau'\le \bar{\tau}/60$. Denote by $V_2$ the set of all $v\in V$ satisfying $\sigma'_tv\neq \sigma_tv$ for some $t\in L^\dag$.
Then $\m(V_2)\le |L^\dag|\bar{\tau}/5$.
Set $V_3=\bigcup_{g\in (\bigcup_{j=1}^nS_j)^C}\pi_g(V_1\cup V_2)$. Then
\[
\m(V_3)\le \Big|\bigcup_{j=1}^nS_j \Big|^C(\m(V_1)+\m(V_2))\le 2\Big|\bigcup_{j=1}^nS_j\Big|^{C}\cdot \frac{|L^\dag|\bar{\tau}}{5}\le \frac{\kappa}{25}.
\]
Denote by $V_4$ the set of all $v\in V$ satisfying $\sigma_{t_1t_2}v\neq \sigma_{t_1}\sigma_{t_2}v$ for some $t_1, t_2\in (L^\dag)^C$. Then
$\m(V_4)\le |L^\dag|^{2C}\bar{\tau}$.
Set $V_5=\bigcup_{l=1}^C\bigcup_{t_1, \dots, t_l\in  L^\dag}\sigma_{t_l}\dots \sigma_{t_1}V_4$. Then
\[
\m(V_5)\le C|L^\dag|^C\m(V_4)\le C|L^\dag|^{3C}\bar{\tau}\le \frac{\kappa}{10}.
\]
Denote by $V_6$ the union of the sets $\varphi(g(\cV_j\cap D))\Delta \pi_g\varphi(\cV_j\cap D)$ for $g\in (\bigcup_{i=1}^nS_i)^C$,
$1\le j\le n$, and $D\in {}_{S_j, L^\dag}\sP$.
Then
$\m(V_6)\le n|F|\delta'\le \kappa/10$.
Also, denote by $V_7$ the union of the sets $\varphi(g A)\Delta \pi_g\varphi(A)$ for $A\in \sA$ and $g\in T$.
Then $\m(V_7)\le |T|\delta'\le \kappa/10$. Put
\[
V_8 = \bigcup_{g\in (\bigcup_{j=1}^nS_j)^C}\pi_g \varphi_0(X\setminus (X_{\bigcup_{1\le j\le n}S_j, L^\dag}\cap X_{(L^\sharp\cup (L^\dag)^C)^2, F})).
\]
Then
$\m(V_8)\le |(\bigcup_{j=1}^nS_j)^C|(2\zeta+\delta')\le \kappa/5$.
Set $V'=V\setminus (V_3\cup V_5\cup V_6\cup V_7\cup V_8)$. We then have
\[
\m(V\setminus V')\le \m(V_3)+\m(V_5)+\m(V_6)+\m(V_7)+\m(V_8)\le \kappa.
\]

Denote by $\sQ_2''$ the set of all $B\in \sQ_2$ satisfying $B\subseteq X_{L^\sharp, T}$.
Let $t\in L^\sharp$ and $B\in \sQ_2''$. We claim that
\begin{align} \label{E-Shannon disjoint3}
V'\cap \varphi(W\cap tB)\subseteq \varphi(W)\cap \sigma_t\varphi(B).
\end{align}
Set $g=\lambda(t, B)\in T$. Denote by $\Xi_{t, B}$ the set consisting of all tuples $\xi=(k_1, \dots, k_l, g_1, \dots, g_l, D_1, \dots, D_l)$ such that $1\le l\le C$, $1\le k_1, \dots, k_l\le n$, $g_j\in S_{k_j}$ for all $1\le j\le l$, $g=g_l\dots g_1$, and $D_j\in {}_{S_{k_j}, L^\dag}\sP$ for all $1\le j\le l$ and such that the set
\[
\Omega_\xi:=B\cap g^{-1}W\cap \bigcap_{j=1}^l((g_j\cdots g_1)^{-1}\cV_{k_j}\cap (g_{j-1}\dots g_1)^{-1}(\cV_{k_j}\cap D_j))\in \alg(\sA)
\]
consisting of all $x\in B\cap t^{-1}W=B\cap g^{-1}W$ satisfying
$g_j\dots g_1x\in \cV_{k_j}$ and $g_{j-1}\dots g_1x\in \cV_{k_j}\cap D_j$ for all $1\le j\le l$
is nonempty.
Then
\[
B\cap g^{-1}W=\bigcup_{\xi\in \Xi_{t, B}}\Omega_\xi.
\]
Denote by $\Xi_{t, B}'$ the set of all $\xi=(k_1, \dots, k_l, g_1, \dots, g_l, D_1, \dots, D_l)\in \Xi_{t, B}$ such that $D_j\in {}_{S_{k_j}, L^\dag}\overline{\sP}$ for all $1\le j\le l$.
For each $\xi=(k_1, \dots, k_l, g_1, \dots, g_l, D_1, \dots, D_l)\in \Xi_{t, B}$,
we have
\begin{align*}
\lefteqn{V'\cap \varphi(g\Omega_\xi)}\hspace*{10mm}\\
\hspace*{0mm} &=V'\cap \varphi\bigg(gB\cap W\cap \bigcap_{j=1}^l((g_l\cdots g_{j+1}\cV_{k_j})\cap (g_l\dots g_j(\cV_{k_j}\cap D_j)))\bigg)\\
&=V'\cap \varphi(gB\cap W)\cap \bigcap_{j=1}^l(\varphi(g_l\dots g_{j+1}\cV_{k_j})\cap \varphi(g_l\dots g_j(\cV_{k_j}\cap D_j)))\\
&=V'\cap \varphi(gB\cap W)\cap \bigcap_{j=1}^l(\pi_{g_l\dots g_{j+1}}\varphi(\cV_{k_j})\cap \pi_{g_l\dots g_j}\varphi(\cV_{k_j}\cap D_j))\\
&=V'\cap \varphi(gB\cap W)\cap \bigcap_{j=1}^l(\pi_{g_l\dots g_{j+1}}\varphi_0(\cV_{k_j})\cap \pi_{g_l\dots g_j}\varphi_0(\cV_{k_j}\cap D_j))\\
&= V'\cap \varphi(gB\cap W)\cap \bigcap_{j=1}^l(\pi_{g_l}\dots \pi_{g_{j+1}}\varphi_0(\cV_{k_j})\cap \pi_{g_l}\dots \pi_{g_{j+1}}\pi_{g_j}\varphi_0(\cV_{k_j}\cap D_j))\\
&=V'\cap \varphi(gB\cap W)\cap \bigcap_{j=1}^l\pi_{g_l}\dots \pi_{g_{j+1}}(\varphi_0(\cV_{k_j})\cap \pi_{g_j}\varphi_0(\cV_{k_j}\cap D_j)).
\end{align*}
If $\xi\in \Xi_{t, B}\setminus \Xi_{t, B}'$, then $D_j=X\setminus X_{S_{k_j}, L^\dag}\subseteq X\setminus X_{\bigcup_{i=1}^nS_i, L^\dag}$ for some $1\le j\le l$ and hence
$V'\cap \varphi(g\Omega_\xi)=\emptyset$. Thus
\begin{gather*}
V'\cap \varphi(W\cap tB)=V'\cap \varphi(gB\cap W)=V'\cap \varphi\bigg(\bigcup_{\xi\in \Xi_{t, B}}g\Omega_\xi\bigg) \hspace*{30mm}\\
\hspace*{35mm} \ =\bigcup_{\xi\in \Xi_{t, B}}(V'\cap \varphi(g\Omega_\xi))=\bigcup_{\xi\in \Xi_{t, B}'}(V'\cap \varphi(g\Omega_\xi)).
\end{gather*}
Now let $\xi\in \Xi_{t, B}'$.
Let $w\in V'\cap \varphi(g\Omega_\xi)$. For each $1\le j\le l$ one has $w=\pi_{g_l\dots g_{j+1}}w_j=\pi_{g_l}\dots \pi_{g_{j+1}}w_j$ for some $w_j\in \varphi_0(\cV_{k_j})\cap \pi_{g_j}\varphi_0(\cV_{k_j}\cap D_j)$. We can also find some $w_0\in V$ such that $w=\pi_{g_l}\dots \pi_{g_1}w_0=\pi_gw_0$.  Note that
\[
w_0=\pi_g^{-1}w\in \pi_g^{-1}(V'\cap \varphi(g\Omega_\xi))\subseteq \varphi(\Omega_\xi)\subseteq \varphi(B).
\]
We have $w_j=\pi_{g_j}w_{j-1}$ for all $1\le j\le l$, and hence
$w_{j-1}\in \varphi_0(\cV_{k_j}\cap D_j)\cap \varphi_0(X_{(L^\sharp\cup (L^\dag)^C)^2, F})$ for all $1\le j\le l$. Set $t_j=\kappa(g_j, D_j)\in L^\dag$ for $1\le j\le l$. Then
$w_j=\pi_{g_j}w_{j-1}=\sigma_{t_j}'w_{j-1}=\sigma_{t_j}w_{j-1}$ for all $1\le j\le l$. Therefore
$w=\sigma_{t_l}\dots \sigma_{t_1}w_0=\sigma_{t_l\dots t_1}w_0$. Take $x\in \Omega_\xi$. We have
\begin{align*}
t=\kappa(g, x)=\kappa(g_l\dots g_1, x)&=\kappa(g_l\dots g_2, g_1x)\kappa(g_1, x)\\
&=\kappa(g_l\dots g_2, g_1x)t_1=\cdots=t_l\dots t_1.
\end{align*}
Thus $w=\sigma_tw_0\in \sigma_t\varphi(B)$ and hence $V'\cap \varphi(g\Omega_\xi)\subseteq \varphi(W)\cap \sigma_t\varphi(B)$. Therefore
\begin{align*}
V'\cap \varphi(W\cap tB)=\bigcup_{\xi\in \Xi_{t, B}'}(V'\cap \varphi(g\Omega_\xi))\subseteq \varphi(W)\cap \sigma_t\varphi(B).
\end{align*}
This proves our claim \eqref{E-Shannon disjoint3}.

Now let $t\in L^\sharp$.
Applying Lemma~\ref{L-Shannon hom and approximation for group} with $\sC=\sR_2$, $L=L^\sharp\cup (L^\dag)^C$, $\tau=\bar{\tau}$, $F=F$, $\tau'=\tau'$, $\pi=\pi$, $\sigma=\sigma$, $\sigma'=\sigma'$, and $\varphi$ being the restriction of $\varphi_0$ to $\alg(((\sR_2)_{(L^\sharp\cup (L^\dag)^C)^2}\vee {}_{(L^\sharp\cup (L^\dag)^C)^2, F}\sP)_F)$, we have that
the restriction of
$\varphi_0$ to $\alg((\sR_2)_{L^\sharp\cup (L^\dag)^C})$ lies in $\Hom_\mu(\sR_2, L^\sharp\cup (L^\dag)^C, \bar{\tau}, \sigma)$.
Thus
\[
\m((\varphi_0(W)\cap \sigma_t\varphi_0(W))\Delta(\varphi_0(W)\cap \varphi_0(tW))\le \m(\sigma_t\varphi_0(W)\Delta \varphi_0(tW))\le \bar{\tau}.
\]
Therefore
\begin{align} \label{E-Shannon hom30}
\m(\varphi_0(W)\cap \sigma_t\varphi_0(W))&\le \m(\varphi_0(W)\cap \varphi_0(tW))+\bar{\tau}\\
&=\m(\varphi_0(W\cap tW))+\bar{\tau}\nonumber \\
&\le \mu(W\cap tW)+\delta'+\bar{\tau}\nonumber \\
&\le \m(\varphi(W\cap tW))+2\delta'+\bar{\tau}\nonumber \\
&\le \m(V'\cap \varphi(W\cap tW))+\kappa+2\delta'+\bar{\tau}. \nonumber
\end{align}
Also note that
\begin{align*}
\lefteqn{\m((V'\cap \varphi(W\cap tW))\setminus (V'\cap \varphi(W\cap t(W\cap X_{L^\sharp, T}))))}
\hspace*{30mm}\\
&\le \m(\varphi(W\cap tW)\setminus \varphi(W\cap t(W\cap X_{L^\sharp, T})))\\
&=\m(\varphi((W\cap tW)\setminus (W\cap t(W\cap X_{L^\sharp, T}))))\\
&\le \m(\varphi(t(X\setminus X_{L^\sharp, T})))\\
&\le \delta'+\mu(t(X\setminus X_{L^\sharp, T}))\le \delta'+\kappa,
\end{align*}
and hence
\begin{align} \label{E-Shannon hom32}
\m(V'\cap \varphi(W\cap tW))\le \m(V'\cap \varphi(W\cap t(W\cap X_{L^\sharp, T})))+\delta'+\kappa.
\end{align}
Then
\begin{align} \label{E-Shannon hom3}
\lefteqn{\sum_{B\in \sQ_2}\m((\varphi(W)\cap \sigma_t\varphi(B))\setminus \varphi(W\cap tB))}\hspace*{10mm}\\
&\le \sum_{B\in \sQ_2\setminus \sQ_2''}\m(\varphi(W)\cap \sigma_t\varphi(B))+\sum_{B\in \sQ_2''}\m((\varphi(W)\cap \sigma_t\varphi(B))\setminus \varphi(W\cap tB))\nonumber\\
&\le \sum_{B\in \sQ_2\setminus \sQ_2''}\m(\varphi(B))+\sum_{B\in \sQ_2''}\m((\varphi(W)\cap \sigma_t\varphi(B))\setminus (V'\cap \varphi(W\cap tB))) \nonumber\\
&\overset{\eqref{E-Shannon disjoint3}}=\m(\varphi(W\setminus X_{L^\sharp, T})) \nonumber\\
&\hspace*{15mm} \ +\m((\varphi(W)\cap \sigma_t\varphi(W\cap X_{L^\sharp, T}))\setminus (V'\cap \varphi(W\cap t(W\cap X_{L^\sharp, T})))) \nonumber\\
&\le \m(\varphi(X\setminus X_{L^\sharp, T}))+\m(\varphi(W)\cap \sigma_t\varphi(W))-\m(V'\cap \varphi(W\cap t(W\cap X_{L^\sharp, T}))) \nonumber\\
&\overset{\eqref{E-Shannon hom32}}\le \mu(X\setminus X_{L^\sharp, T})+\delta'+\m(\varphi_0(W)\cap \sigma_t\varphi_0(W))\nonumber \\
&\hspace*{45mm} \ -\m(V'\cap \varphi(W\cap tW))+\delta'+\kappa \nonumber\\
&\overset{\eqref{E-Shannon hom30}}\le \kappa+\delta'+\kappa+2\delta'+\bar{\tau}+\delta'+\kappa=3\kappa+4\delta'+\bar{\tau}\le 8\kappa, \nonumber
\end{align}
and
\begin{align} \label{E-Shannon hom4}
\lefteqn{\sum_{B\in \sQ_2}\m(\varphi(W\cap tB)\setminus (\varphi(W)\cap \sigma_t\varphi(B)))}\hspace*{12mm}\\
&\le \sum_{B\in \sQ_2\setminus \sQ_2''}\m(\varphi(W\cap tB))+\sum_{B\in \sQ_2''}\m(\varphi(W\cap tB)\setminus (\varphi(W)\cap \sigma_t\varphi(B)))\nonumber \\
&\overset{\eqref{E-Shannon disjoint3}}\le \m(\varphi(W\cap t(W\setminus X_{L^\sharp, T})))+\m(V\setminus V')\nonumber\\
&\le \m(\varphi(t(X\setminus X_{L^\sharp, T})))+\kappa \nonumber \\
&\le \mu(t(X\setminus X_{L^\sharp, T}))+\delta'+\kappa\le 2\kappa+\delta'. \nonumber
\end{align}

For each $(B, g)\in \Lambda$, we have $ \tilde{\varphi}'(gB)=\sigma_{\kappa(g, B)}\varphi(B)$, $\kappa(g, B)\in \overline{L}$, and  $W\cap \kappa(g, B)B=W\cap gB=\emptyset$.
Thus
\begin{align*}
\sum_{(B, g)\in \Lambda}\m(\tilde{\varphi}'(gB)\cap \varphi(W))&= \sum_{(B, g)\in \Lambda}\m(\sigma_{\kappa(g, B)}\varphi(B)\cap \varphi(W))\\
&\le \sum_{t\in \overline{L}}\sum_{\substack{B\in \sQ_2',\\ tB\cap W=\emptyset}}\m(\sigma_t\varphi(B)\cap \varphi(W))\\
&\overset{\eqref{E-Shannon hom3}}\le |\overline{L}|8\kappa\\
&\le \frac{\delta}{40},
\end{align*}
verifying \eqref{E-Shannon disjoint1}.

For distinct $(B_1, g_1), (B_2, g_2)$ in $\Lambda$, we have $\kappa(g_1, B_1)B_1\cap \kappa(g_2, B_2)B_2=\emptyset$, and hence $B_1\cap \kappa(g_1, B_1)^{-1}\kappa(g_2, B_2)B_2=\emptyset$. For any $t_1, t_2\in L\overline{L}$, we have
\begin{align} \label{E-Shannon Hamming}
\rho_\Hamm(\sigma_{t_1}^{-1}\sigma_{t_2}, \sigma_{t_1^{-1}t_2})&\le \rho_\Hamm(\sigma_{t_1}^{-1}, \sigma_{t_1^{-1}})+\rho_\Hamm(\sigma_{t_1^{-1}}\sigma_{t_2}, \sigma_{t_1^{-1}t_2})\\
&\le\rho_\Hamm(\id_V, \sigma_{t_1}\sigma_{t_1^{-1}})+\bar{\tau} \nonumber\\
&\le \rho_\Hamm(\id_V, \sigma_{e_H})+\rho_\Hamm(\sigma_{e_H}, \sigma_{t_1}\sigma_{t_1^{-1}})+\bar{\tau} \nonumber\\
&= \rho_\Hamm(\sigma_{e_H}, \sigma_{e_H}\sigma_{e_H})+\rho_\Hamm(\sigma_{e_H}, \sigma_{t_1}\sigma_{t_1^{-1}})+\bar{\tau} \nonumber\\
&\le 3\bar{\tau}. \nonumber
\end{align}
Thus
\begin{align*}
\lefteqn{\sum_{\substack{(B_1, g_1), (B_2, g_2)\in \Lambda,\\ (B_1, g_1)\neq (B_2, g_2)}}\m(\tilde{\varphi}'(g_1B_1)\cap \tilde{\varphi}'(g_2B_2))}\hspace*{12mm}\\
&=\sum_{\substack{(B_1, g_1), (B_2, g_2)\in \Lambda,\\ (B_1, g_1)\neq (B_2, g_2)}}\m(\sigma_{\kappa(g_1, B_1)}\varphi(B_1)\cap \sigma_{\kappa(g_2, B_2)}\varphi(B_2))\\
&\le \sum_{t_1, t_2\in \overline{L}}\sum_{\substack{B_1, B_2\in \sQ_2',\\ B_1\cap t_1^{-1}t_2B_2=\emptyset}}\m(\sigma_{t_1}\varphi(B_1)\cap \sigma_{t_2}\varphi(B_2))\\
&=\sum_{t_1, t_2\in \overline{L}}\sum_{\substack{B_1, B_2\in \sQ_2',\\ B_1\cap t_1^{-1}t_2B_2=\emptyset}}\m(\varphi(B_1)\cap \sigma_{t_1}^{-1}\sigma_{t_2}\varphi(B_2))\\
&\le \sum_{t_1, t_2\in \overline{L}}\Bigg[\rho_\Hamm(\sigma_{t_1}^{-1}\sigma_{t_2}, \sigma_{t_1^{-1}t_2})+\sum_{\substack{B_1, B_2\in \sQ_2',\\ B_1\cap t_1^{-1}t_2B_2=\emptyset}} \m(\varphi(B_1)\cap \sigma_{t_1^{-1}t_2}\varphi(B_2))\Bigg]\\
&\overset{\eqref{E-Shannon Hamming}}\le 3|\overline{L}|^2 \bar{\tau}+\sum_{t_1, t_2\in \overline{L}}\sum_{\substack{B_1, B_2\in \sQ_2', \\B_1\cap t_1^{-1}t_2B_2=\emptyset}} \m(\varphi(B_1)\cap \varphi(W)\cap \sigma_{t_1^{-1}t_2}\varphi(B_2))\\
&\overset{\eqref{E-Shannon hom3}}\le 3|\overline{L}|^2 \bar{\tau}+|\overline{L}|^28\kappa +\sum_{t_1, t_2\in \overline{L}}\sum_{\substack{B_1, B_2\in \sQ_2',\\ B_1\cap t_1^{-1}t_2B_2=\emptyset}} \m(\varphi(B_1)\cap \varphi(W\cap t_1^{-1}t_2B_2))\\
&=|\overline{L}|^2(3\bar{\tau}+8\kappa)\le \frac{\delta}{40},
\end{align*}
verifying \eqref{E-Shannon disjoint2}.

For each $t\in LL^\bullet$, we have
\begin{align} \label{E-Shannon hom40}
\lefteqn{\sum_{B\in \sQ_1'}\m(\sigma_t\varphi(B)\setminus \tilde{\varphi}'(tB))}\hspace*{10mm}\\
&=\sum_{B\in \sQ_1'}\m(\varphi(B)\setminus \sigma_t^{-1}\tilde{\varphi}'(tB))\nonumber \\
&=\sum_{B\in \sQ_1'}\big[\m(\varphi(B))-\m(\varphi(B)\cap \sigma_t^{-1}\tilde{\varphi}'(tB))\big]\nonumber\\
&=\sum_{B\in \sQ_1'}\bigg[\m(\varphi(B))-\m\bigg(\varphi(B)\cap  \sigma_t^{-1}\tilde{\varphi}'\bigg(\bigcup_{\substack{B_1\in \sQ_2, g_1\in \Theta(B_1),\\ g_1B_1\subseteq tB}}g_1B_1\bigg)\bigg)\bigg]\nonumber\\
&= \sum_{B\in \sQ_1'}\bigg[\m(\varphi(B))-\m\bigg(\bigcup_{\substack{B_1\in \sQ_2, g_1\in \Theta(B_1),\\ g_1B_1\subseteq tB}}(\varphi(B)\cap  \sigma_t^{-1}\tilde{\varphi}'(g_1B_1))\bigg)\bigg]\nonumber\\
&= \sum_{B\in \sQ_1'}\bigg[\m(\varphi(B))-\m\bigg(\bigcup_{B_1\in \sQ_2, B_1\subseteq tB}(\varphi(B)\cap  \sigma_t^{-1}\varphi(B_1))\bigg)\nonumber\\
&\hspace*{20mm} \ -\m\bigg(\bigcup_{\substack{(B_1, g_1)\in \Lambda,\\  g_1B_1\subseteq tB}}(\varphi(B)\cap \sigma_t^{-1}\sigma_{\kappa(g_1, B_1)}\varphi(B_1))\bigg)\bigg]\nonumber\\
&\le \sum_{B\in \sQ_1'}\bigg[\m(\varphi(B))-\m\bigg(\bigcup_{B_1\in \sQ_2', B_1\subseteq tB}(\varphi(B)\cap  \sigma_t^{-1}\sigma_{e_H}\varphi(B_1))\bigg)\nonumber\\
&\hspace*{20mm} \ -\m\bigg(\bigcup_{\substack{(B_1, g_1)\in \Lambda,\\  g_1B_1\subseteq tB}}(\varphi(B)\cap \sigma_t^{-1}\sigma_{\kappa(g_1, B_1)}\varphi(B_1))\bigg)\bigg]+\rho_\Hamm(\sigma_{e_H}, \id_V)\nonumber\\
&= \sum_{B\in \sQ_1'}\bigg[\m(\varphi(B))-\m\bigg(\bigcup_{t_1\in \overline{L}}\bigcup_{\substack{B_1\in \sQ_2',\\ \lambda(t_1, B_1)\in \Theta(B_1),\\ t_1B_1\subseteq tB}}(\varphi(B)\cap \varphi(W)\cap \sigma_t^{-1}\sigma_{t_1}\varphi(B_1))\bigg)\bigg]\nonumber\\
&\hspace*{20mm} \ +\rho_\Hamm(\sigma_{e_H}\sigma_{e_H}, \sigma_{e_H})\nonumber\\
&\le \sum_{B\in \sQ_1'}\bigg[\m(\varphi(B))-\m\bigg(\bigcup_{t_1\in \overline{L}}\bigcup_{\substack{B_1\in \sQ_2',\\ \lambda(t_1, B_1)\in \Theta(B_1),\\ t_1B_1\subseteq tB}}(\varphi(B)\cap \varphi(W)\cap \sigma_{t^{-1}t_1}\varphi(B_1))\bigg)\bigg]\nonumber\\
&\hspace*{20mm} \ +\sum_{t_1\in \overline{L}}\rho_\Hamm(\sigma_t^{-1}\sigma_{t_1}, \sigma_{t^{-1}t_1})+\bar{\tau}\nonumber\\
&\overset{\eqref{E-Shannon hom4}, \eqref{E-Shannon Hamming}}\le \sum_{B\in \sQ_1'}\bigg[\m(\varphi(B))-\m\bigg(\bigcup_{t_1\in \overline{L}}\bigcup_{\substack{B_1\in \sQ_2',\\ \lambda(t_1, B_1)\in \Theta(B_1),\\ t_1B_1\subseteq tB}}(\varphi(B)\cap \varphi(W\cap t^{-1}t_1B_1))\bigg)\bigg]\nonumber\\
&\hspace*{20mm} \ + |\overline{L}|(2\kappa+\delta')+3 |\overline{L}|\bar{\tau}+\bar{\tau}\nonumber\\
&\le \sum_{B\in \sQ_1'}\bigg[\m(\varphi(B))-\m\bigg(\bigcup_{\substack{t_1\in \overline{L}, B_1\in \sQ_2',\\ \lambda(t_1, B_1)\in \Theta(B_1),\\ t_1B_1\subseteq tB}}\varphi(B\cap t^{-1}t_1B_1)\bigg)\bigg]
+|\overline{L}|(2\kappa+\delta'+4\bar{\tau})\nonumber\\
&= \sum_{B\in \sQ_1'}\m\bigg(\varphi\bigg(B\cap \bigcup_{\substack{B_1\in \sQ_2\setminus \sQ_2',\\ g_1\in \Theta(B_1)}}t^{-1}g_1B_1\bigg)\bigg)+ |\overline{L}|(2\kappa+\delta'+4\bar{\tau})\nonumber\\
&\le \m\bigg(\varphi\bigg(\bigcup_{\substack{B_1\in \sQ_2\setminus \sQ_2',\\ g_1\in \Theta(B_1)}}t^{-1}g_1B_1\bigg)\bigg)+|\overline{L}|(2\kappa+\delta'+4\bar{\tau})\nonumber\\
&\le \mu\bigg(\bigcup_{\substack{B_1\in \sQ_2\setminus \sQ_2',\\ g_1\in \Theta(B_1)}}t^{-1}g_1B_1\bigg)+\delta'+|\overline{L}|(2\kappa+\delta'+4\bar{\tau})\nonumber\\
&\le  |S|\mu(X\setminus X_{S, \overline{L}})+|\overline{L}|(2\kappa+2\delta'+4\bar{\tau})\nonumber\\
&\le |S|\gamma +|\overline{L}|8\kappa\le \frac{\delta}{100|LL^\bullet|}, \nonumber
\end{align}
and also
\begin{align} \label{E-Shannon hom41}
 \sum_{B\in \sQ_1'}\m(\tilde{\varphi}'(tB))&\le \sum_{B\in \sQ_1'}\,\sum_{\substack{B_1\in \sQ_2, g_1\in \Theta(B_1),\\ g_1B_1\subseteq tB}}\m(\tilde{\varphi}'(g_1B_1)) \\
&\le \sum_{B\in \sQ_1'}\,\sum_{\substack{B_1\in \sQ_2, g_1\in \Theta(B_1),\\ g_1B_1\subseteq tB}}\m(\varphi(B_1))\nonumber \\
&\le |S|\delta'+\sum_{B\in \sQ_1'}\,\sum_{\substack{B_1\in \sQ_2, g_1\in \Theta(B_1),\\ g_1B_1\subseteq tB}}\mu(B_1)\nonumber \\
&=  |S|\delta'+\sum_{B\in \sQ_1'}\,\sum_{\substack{B_1\in \sQ_2, g_1\in \Theta(B_1),\\ g_1B_1\subseteq tB}}\mu(g_1B_1)\nonumber \\
&=  |S|\delta'+\sum_{B\in \sQ_1'}\mu(B)\nonumber \\
&\le (|S|+1)\delta'+\sum_{B\in \sQ_1'}\m(\varphi(B))\nonumber.
\end{align}
Thus
\begin{align*}
\lefteqn{\sum_{B\in \sQ_1'}\sum_{t\in LL^\bullet}\m(\sigma_t\varphi(B)\Delta \tilde{\varphi}'(tB))}\hspace*{15mm}\\
&=\sum_{B\in \sQ_1'}\sum_{t\in LL^\bullet}\big[2\m(\sigma_t\varphi(B)\setminus \tilde{\varphi}'(tB))+\m(\tilde{\varphi}'(tB))-\m(\sigma_t\varphi(B))\big]\\
&\overset{\eqref{E-Shannon hom40}}\le \frac{\delta}{50}+\sum_{B\in \sQ_1'}\sum_{t\in LL^\bullet}\big[\m(\tilde{\varphi}'(tB))-\m(\varphi(B))\big]\\
&\overset{\eqref{E-Shannon hom41}}\le \frac{\delta}{50}+|LL^\bullet|(|S|+1)\delta'\le \frac{\delta}{20},
\end{align*}
verifying \eqref{E-Shannon hom}.
\end{proof}

For $F\in \cF(G)$ and $L\in \cF(H)$ we denote by ${}_{F, L}\sW$ the countable Borel partition of $X$ consisting of $X_{F, L}$ and $P\in {}_F\sP\setminus {}_{F, L}\sP$.
For any $F\in \cF(G)$, the fact that ${}_F\sP$ has finite Shannon entropy means that for every $\varepsilon>0$ we can find a $\Gamma(F, \varepsilon)\in \cF(H)$ such that $H_\mu({}_{F, \Gamma(F, \varepsilon)}\sW)<\varepsilon$.

\begin{proof}[Proof of Theorem~\ref{T-Shannon SC to entropy}]
We may assume that $h_\mu (G\curvearrowright X) \neq-\infty$, which means
in particular that $G$ is sofic.

By Stirling's formula there is a function $\Psi: (0, 1)\rightarrow (0, 1)$ such that  for any nonempty finite set $V$ and any $0<\varepsilon<1$ the number of subsets $V'$ of $V$ satisfying $|V'|/|V|\le \Psi(\varepsilon)$ is at most $e^{\varepsilon|V|}$.

Let $\Pi = \{ \pi_k : G\to \Sym (V_k ) \}_{k=1}^\infty$ be a sofic approximation sequence for $G$ and $\overline{\sC}$ a finite Borel partition of $X$ with $h_{\Pi, \mu}(G\curvearrowright X, \overline{\sC})\ge 0$. Let $0<\varepsilon<1$.
To establish the theorem it is enough to show the existence of
a sofic approximation sequence $\Sigma$ for $H$ and a finite Borel partition $\sC^\flat$ of $X$ such that $h_{\Sigma, \mu}(H\curvearrowright X, \sC^\flat)\ge h_{\Pi, \mu}(G\curvearrowright X, \overline{\sC})-6\varepsilon$.

Enumerate the elements of $\cF(G)$ as $\bar{F}_1, \bar{F}_2, \dots$.
Take two decreasing sequences $1>\delta_1>\delta_2>\dots$ and $1>\tau_1>\tau_2>\dots$ converging to $0$,
an increasing sequence $\{L_k\}_{k\in \Nb}$ in $\overline{\cF}(H)$ with union $H$,
and an increasing sequence $\{\sU_k\}_{k\in \Nb}$ of finite Borel partitions of $X$ such that the algebra generated by $\bigcup_{k\in \Nb}\sU_k$ is dense in the Borel $\sigma$-algebra of $X$ with respect to the
pseudometric $d(A,B) = \mu (A\Delta B)$ (such a sequence can be found in view of the fact that every 
atomless standard probability space is measure isomorphic to the unit interval equipped with the Lebesgue measure
on its Borel $\sigma$-algebra \cite[Theorem~A.20]{KerLi16}).

We define $\Upsilon: \cF(G)\rightarrow [0, \infty)$
by $\Upsilon(\bar{F}_\ell)=2/(\Psi(\varepsilon)\Psi(\varepsilon/(2^\ell|{}_{\bar{F}_\ell, \Gamma(\bar{F}_\ell, \varepsilon/2^\ell)}\sP|)))$ for all $\ell\in \Nb$.

Since $G\curvearrowright (X, \mu)$ has property SC, there is an $S\in \overline{\cF}(G)$ such that for any $T_k\in \overline{\cF}(G)$ there are $n_k, C_k\in \Nb$,   $S_{k, 1}, \dots, S_{k, n_k}\in \overline{\cF}(G)$, and
Borel sets $W_k, \cV_{k, 1}, \dots, \cV_{k, n_k}\subseteq X$
satisfying the following conditions:
\begin{enumerate}
\item $\sum_{j=1}^{n_k} \Upsilon(S_{k,j})\mu(\cV_{k, j})\le 1$,
\item $SW_k=X$,
\item if $w_1, w_2\in W_k$ satisfy $gw_1=w_2$ for some $g\in T_k$ then $w_1$ and $w_2$ are connected by a path of length at most $C_k$ in which each edge is an $S_{k, j}$-edge with both endpoints in $\cV_{k, j}$ for some $1\le j\le n_k$.
\end{enumerate}
We may
assume that the sets $S_{k, 1}, \dots, S_{k, n_k}$ are distinct.
From (iii) we have the inclusion $W_k\subseteq \bigcup_{j=1}^{n_k}\cV_{k, j}$.
Take an $L^\flat\in \overline{\cF}(H)$ such that $\mu(X\setminus X_{S^2, L^\flat})\le \Psi(\varepsilon/|\overline{\sC}|)/(10|S|^3)$.

Fix $k\in \Nb$. Put $L_k^*=L_k\cup L^\flat\in \overline{\cF}(H)$. Take $0<\delta_k^*\le \min\{\delta_k, \Psi(\varepsilon/|\overline{\sC}|)/(10|S|\cdot |L_k^*|)\}$. Take also
an $L^\bullet_k\in \overline{\cF}(H)$ such that  $\mu(X\setminus X_{S, L_k^\bullet})\le \delta_k^*/(20|S|)$ and an $\overline{L}_k\in \overline{\cF}(H)$ such that $L_k^\bullet\subseteq \overline{L}_k$  and  $\mu(X\setminus X_{S, \overline{L}_k})\le \gamma_k:=\delta_k^*/(200 |S|\cdot |L_k^*L_k^\bullet|)$. Put
$L_k^\sharp=\overline{L}_k L_k^* \overline{L}_k\in \overline{\cF}(H)$. Choose a $T_k\in \overline{\cF}(G)$ such that $\mu(X\setminus X_{L_k^\sharp, T_k})\le \kappa_k:=\delta_k^*/(10^4|\overline{L}_k|^2\cdot |L_k^*L_k^\bullet|)$. Then we have $n_k, C_k$, $S_{k, j}$ for $1\le j\le n_k$, $W_k$, and $\cV_{k, j}$ for $1\le j\le n_k$ as above.

Say $S_{k, j}=\bar{F}_{\ell_{k, j}}$ for $1\le j\le n_k$.
Take an $L_k^\dag\in \overline{\cF}(H)$ such that $\bigcup_{j=1}^{n_k}\Gamma(S_{k, j}, \varepsilon/2^{\ell_{k, j}})\subseteq L_k^\dag$ and
\[
\mu\big(X\setminus X_{\bigcup_{j=1}^{n_k}S_{k, j}, L_k^\dag}\big)\le \zeta_k:=\kappa_k/(100|\textstyle\bigcup_{j=1}^{n_k}S_{k, j}|^{C_k}),
\]
and take $0< \bar{\tau}_k\le \min\{\tau_k, \kappa_k/(10C_k|\bigcup_{j=1}^{n_k}S_{k, j}|^{C_k}\cdot |L_k^\dag|^{3C_k})\}$.
Take an $F^\natural_k\in \overline{\cF}(G)$ such that $\mu(X\setminus X_{(L_k^\sharp\cup (L_k^\dag)^{C_k})^2, F^\natural_k})\le \bar{\tau}_k/30$, and
take an $F_k\in \overline{\cF}(G)$ containing $F^\natural_{k}\cup T_k\cup (\bigcup_{j=1}^{n_k}S_{k, j})^{C_k}\cup S$ such that $\mu(X\setminus X_{(L_k^\sharp\cup (L_k^\dag)^{C_k})^2, F_k})\le \min\{\zeta_k, \bar{\tau}_k/(30|F^\natural_{k}|)\}$.

Set $\sC'_k=\{W_k, X\setminus W_k\}$ and  $\sC''_k=\overline{\sC}\vee \sC_k'\vee {}_{S^2, L^\flat}\sP$.   Applying Lemma~\ref{L-Shannon separation} with
$\sC=\overline{\sC}$, $S=S$, $W=W_k$, and $L=L^\flat$ we find
 a finite  Borel partition $\sQ_k$ of $W_k$ contained in $\alg((\overline{\sC}\vee \sC_k')_{S^2})$ and a map $\Theta_k: \sQ_k\rightarrow \cF(S)$ such that $e_G\in \Theta_k(B)$ for every $B\in \sQ_k$
and the sets $gB$ for $B\in \sQ_k$ and $g\in \Theta_k(B)$ form a partition $\sR_k$ of $X$ finer than $\overline{\sC}$. Set $\sC_k^*=\sR_k\vee \sC_k''\vee \sU_k$.

Denote by $\sD'_k$ the partition of $X$ generated by $W_k, \cV_{k, 1}, \dots, \cV_{k, n_k}$.
Put
\[
\sD_k=(\sC_k^*)_{L_k^*}\vee (\sD'_k)_{T_k}\vee \bigg(\bigvee_{j=1}^{n_k}{}_{S_{k, j}, L_k^\dag}\sP\bigg)\vee {}_{\bigcup_{j=1}^{n_k}S_{k, j}, L_k^\dag}\sP\vee {}_{L_k^\sharp, T_k}\sP\vee {}_{S, \overline{L}_k}\sP.
\]
Applying Lemma~\ref{L-Shannon partition} first with $S=S$, $W=W_k$, $\Theta=\Theta_k$, $\sQ=\sQ_k$, and $\sD=\sD_k$ 
to get partitions $\sQ_{k, 1}$ and $\sR_{k, 1}$ and then
again with $S=S$, $W=W_k$, $\Theta=\Theta_k$, $\sQ=\sQ_{k, 1}$, and $\sD=(\sR_{k, 1})_{L_k^*\overline{L}_k}$, we find finite Borel partitions
$\sQ_{k, 1}$ and $\sQ_{k, 2}$ of $W_k$ such that $\sQ_k\preceq \sQ_{k, 1}\preceq \sQ_{k, 2}$,
and $\sD_k\preceq \sR_{k, 1}$ and $(\sR_{k, 1})_{L_k^*\overline{L}_k}\preceq \sR_{k, 2}$, where for $i=1, 2$ we set $\Theta_k(B_i)=\Theta_k(B)$ for $B\in \sQ$ and $B_i$ in $\sQ_{k, i}$  satisfying $B_i\subseteq B$, and $\sR_{k, i}$ 
is the partition of $X$ consisting of the sets $gB_i$ for $B_i\in \sQ_{k, i}$ and $g\in \Theta_k(B_i)$.
Denote by $\sQ_{k, 1}'$ the set of all $B\in \sQ_{k, 1}$ satisfying $B\subseteq X_{S, L_k^\bullet}$, and denote by $\sQ_{k, 2}'$ the set of all $B\in \sQ_{k, 2}$ satisfying $B\subseteq X_{S, \overline{L}_k}$.
Denote by $\Lambda_k$ the set consisting of the pairs $(B, g)$ for all $B\in \sQ_{k, 2}'$ and $g\in \Theta_k(B)\setminus \{e_G\}$.

Let $1\le j\le n_k$. Put $L_{k, j}^\dag=\Gamma(S_{k, j}, \varepsilon/2^{\ell_{k, j}})$. Since $L_{k,j}^\dag\subseteq L_k^\dag$, we have ${}_{S_{k, j}, L_{k, j}^\dag}\overline{\sP}\subseteq {}_{S_{k, j}, L_k^\dag}\overline{\sP}$. Denote by $\sW_{k, j}$ the finite partition of $X$ consisting of $X_{S_{k, j}, L_{k, j}^\dag}, X\setminus X_{S_{k, j}, L_k^\dag}$, and the elements of ${}_{S_{k, j}, L_k^\dag}\overline{\sP}\setminus {}_{S_{k, j}, L_{k, j}^\dag}\overline{\sP}$. Then $\sW_{k, j}$ is coarser than ${}_{S_{k, j}, L_{k, j}^\dag}\sW$ and ${}_{S_{k, j}, L_k^\dag}\sP$, and hence
\[
H_\mu(\sW_{k, j})\le H_\mu( {}_{S_{k, j}, L_{k, j}^\dag}\sW)\le \frac{\varepsilon}{2^{\ell_{k, j}}}.
\]
By \cite[Proposition 10.2]{KerLi16} we can find an $\eta_{k, j}>0$ such that for any large enough finite set $V$ the number of homomorphisms $\varphi: \alg(\sW_{k, j})\rightarrow \Pb_V$ satisfying
$\sum_{A\in \sW_{k, j}}|\m(\varphi(A))-\mu(A)|\le \eta_{k, j}$ is at most $e^{(H_\mu(\sW_{k, j})+\varepsilon/2^{\ell_{k, j}})|V|}\le e^{2(\varepsilon/2^{\ell_{k, j}})|V|}$.

Take
\[
0<\tau_k'\le \min\{\kappa_k/(100|F_k|^3), \bar{\tau}_k/(60|F_k|^2)\}
\]
and
\begin{gather*}
0<\delta_k'\le \min\Big\{\kappa_k/(10n_k|F_k|), \delta_k^*/(50 |L_k^*L_k^\bullet|(|S|+1)), \tau_k', \hspace*{40mm}\\
\hspace*{60mm} \Psi(\varepsilon)/2,  \min_{1\le j\le n_k}1/\Upsilon(S_{k, j}), \min_{1\le j\le n_k}\eta_{k, j}\Big\}.
\end{gather*}
Let $\sA_k$ be a finite Borel partition of $X$ refining
\[
(\sR_{k, 2})_{(L_k^\sharp\cup (L_k^\dag)^{C_k})^2}\vee {}_{(L_k^\sharp\cup (L_k^\dag)^{C_k})^2, F_k}\sP\vee (\sD_k)_{(\bigcup_{j=1}^{n_k}S_{k, j})^{C_k}}.
\]

Take $m_k\ge k$ large enough so that
\[
\frac{1}{|V_{m_k}|} \log| \Hom_\mu(\sA_k, F_k, \delta'_k, \pi_{m_k})|_{\overline{\sC}}
\ge h_{\Pi, \mu}(G\curvearrowright X, \overline{\sC})-\varepsilon
\]
and so that $\pi_{m_k}: G\rightarrow \Sym(V_{m_k})$ is an $(F_k, \tau_k')$-approximation for $G$.

Pick a subset $\Phi$ of $\Hom_\mu(\sA_k, F_k, \delta'_k, \pi_{m_k})$ such that
different elements of $\Phi$ have different restrictions to $\overline{\sC}$ and
\[
|\Phi|=|\Hom_\mu(\sA_k, F_k, \delta'_k, \pi_{m_k})|_{\overline{\sC}}.
\]
Take a maximal subset $\Phi_1$ of $\Phi$ which is $(\rho_{\overline{\sC}}, \Psi(\varepsilon/|\overline{\sC}|))$-separated 
in the sense that $\rho_{\overline{\sC}} (\varphi ,\psi ) > \Psi(\varepsilon/|\overline{\sC}|)$
for all distinct $\varphi ,\psi\in\Phi_1$. For each $\varphi\in \Phi_1$, if $\psi\in \Phi$ satisfies $\rho_{\overline{\sC}}(\varphi, \psi)\le \Psi(\varepsilon/|\overline{\sC}|)$ then for each $A\in \overline{\sC}$ the number of possibilities for $\psi(A)$ is at most $e^{\varepsilon|V_{m_k}|/|\overline{\sC}|}$ since $\m(\varphi(A)\Delta \psi(A))\le \Psi(\varepsilon/|\overline{\sC}|)$. Thus for each $\varphi\in \Phi_1$ the number of $\psi\in \Phi$ satisfying $\rho_{\overline{\sC}}(\varphi, \psi)\le \Psi(\varepsilon/|\overline{\sC}|)$ is at most $e^{\varepsilon|V_{m_k}|}$. Therefore
\[
|\Phi|\le |\Phi_1|e^{\varepsilon|V_{m_k}|}.
\]

For every $1\le j\le n_k$ and $\varphi\in \Phi_1$ we have
\[
\m(\varphi(\cV_{k, j}))\le \mu(\cV_{k, j})+\delta'_k\le \frac{1}{\Upsilon(S_{k, j})}+\delta_k'\le \frac{2}{\Upsilon(S_{k, j})}\le \Psi(\varepsilon/(2^{\ell_{k, j}}|{}_{S_{k, j}, L_{k, j}^\dag}\sP|)).
\]
Thus for every $1\le j\le n_k$ and $D\in {}_{S_{k, j}, L_{k, j}^\dag}\sP$, the number of possibilities for $\varphi(\cV_{k, j}\cap D)$ for $\varphi\in \Phi_1$ is at most $e^{\varepsilon|V_{m_k}|/(2^{\ell_{k, j}}|{}_{S_{k, j}, L_{k, j}^\dag}\sP|)}$.
We can then find
a subset $\Phi_2$ of $\Phi_1$ such that for every $1\le j\le n_k$ and $D\in {}_{S_{k, j}, L_{k, j}^\dag}\sP$ the set
$\varphi(\cV_{k, j}\cap D)$ is the same for all $\varphi\in \Phi_2$ and
\[
|\Phi_1|\le |\Phi_2|\prod_{j=1}^{n_k}e^{\varepsilon|V_{m_k}|/2^{\ell_{k, j}}}\le |\Phi_2|e^{\varepsilon|V_{m_k}|}.
\]
In particular, the sets $\varphi(\cV_{k,j})$ for $1\le j\le n_k$ are the same for all $\varphi\in \Phi_2$. 

Since $\Upsilon\ge 2/\Psi(\varepsilon)$, for each $\varphi\in \Phi_2$ we have
\[
\m(\varphi(W_k))\le \mu(W_k)+\delta'_k\le \sum_{j=1}^{n_k}\mu(\cV_{k, j})+\delta'_k\le \frac{\Psi(\varepsilon)}{2}+\delta_k'\le \Psi(\varepsilon).
\]
Thus the number of possibilities of $\varphi(W_k)$ for $\varphi\in \Phi_2$ is at most $e^{\varepsilon|V_{m_k}|}$. It follows that there is a subset $\Phi_3$ of $\Phi_2$ such that
$\varphi(W_k)$ is the same for all $\varphi\in \Phi_3$ and
\[
|\Phi_2|\le |\Phi_3|e^{\varepsilon|V_{m_k}|}.
\]

For each $1\le j\le n_k$, since $\delta_k'\le \eta_{k, j}$ the number of possibilities for $\varphi|_{\sW_{k, j}}$ for $\varphi\in \Phi_3$ is at most $e^{2(\varepsilon/2^{\ell_{k, j}})|V_{m_k}|}$. Thus there is a subset $\Phi_4$ of $\Phi_3$ such that for each $1\le j\le n_k$ the restriction
$\varphi|_{\sW_{k, j}}$ is the same for all $\varphi\in \Phi_4$ and
\[
|\Phi_3|\le |\Phi_4|\prod_{j=1}^{n_k} e^{2(\varepsilon/2^{\ell_{k, j}})|V_{m_k}|}\le |\Phi_4|e^{2\varepsilon|V_{m_k}|}.
\]
Note that the set $\varphi(W_k)$ is the same for all $\varphi\in \Phi_4$,
and for every $1\le j\le n_k$ and $D\in {}_{S_{k, j}, L_k^\dag}\sP$ the set
$\varphi(\cV_{k, j}\cap D)$ is the same for all $\varphi\in \Phi_4$.

Fix a $\varphi_0\in \Phi_4$. For each  $t\in (L^\sharp_k\cup (L^\dag_k)^{C_k})^2$ take a map
$\sigma'_{k,t}:V_{m_k}\rightarrow V_{m_k}$ such that
\[
\sigma'_{k, t}v=\pi_{m_k, \lambda(t, A)}v
\]
for all $A\in {}_{(L_k^\sharp\cup (L_k^\dag)^{C_k})^2, F_k}\overline{\sP}$  and $v\in \varphi_0(A)$.
Applying Lemma~\ref{L-Shannon approximation for group measure} with $L=L^\sharp_k\cup (L^\dag_k)^{C_k}$, $\tau=\bar{\tau}_k, F^\natural=F^\natural_k$, $F=F_k$, $\tau'=\tau'_k, \pi=\pi_{m_k}$,  $\sigma'=\sigma'_k$, and $\varphi$ being the restriction of $\varphi_0$ to $\alg(({}_{(L^\sharp_k\cup (L^\dag_k)^{C_k})^2, F_k}\sP)_{F_k})$, we find an $(L^\sharp_k\cup (L^\dag_k)^{C_k}, \bar{\tau}_k)$-approximation $\sigma_k: H\rightarrow \Sym(V_{m_k})$ for $H$
such that $\rho_{\Hamm}(\sigma_{k, t}, \sigma'_{k, t})\le \bar{\tau}_k/5$ for all $t\in (L^\sharp_k\cup (L^\dag_k)^{C_k})^2$.

Let $\varphi\in \Phi_4$.
Define $\tilde{\varphi}': \sR_{k, 2}\rightarrow \Pb_{V_{m_k}}$ by
$\tilde{\varphi}'(B)=\varphi(B)$ for all $B\in \sQ_{k, 2}$, $\tilde{\varphi}'(gB)=\emptyset$ for all $B\in \sQ_{k, 2}\setminus \sQ_{k, 2}'$ and $g\in \Theta_k(B)\setminus \{e_G\}$,  and
\[
\tilde{\varphi}'(gB)=\sigma_{k, \kappa(g, B)}\varphi(B)
\]
for all $(B, g)\in \Lambda_k$.  Extend $\tilde{\varphi}'$ to a map $\alg(\sR_{k,2})\rightarrow \Pb_{V_{m_k}}$ by setting $\tilde{\varphi}'(D)=\bigcup_{A\in \sR_{k, 2}, A\subseteq D}\tilde{\varphi}'(A)$ for $D\in \alg(\sR_{k, 2})$. Applying Lemma~\ref{L-Shannon disjoint}
with $\sC=\sC_k^*$, $L=L_k^*$, $\delta=\delta_k^*$, $S=S$, $L^\bullet=L_k^\bullet$, $\overline{L}=\overline{L}_k$, $T=T_k$, $C=C_k$, $n=n_k$, $S_j=S_{k, j}$, $W=W_k$, $\cV_j=\cV_{k, j}$, $L^\dag=L_k^\dag$, $\bar{\tau}=\bar{\tau}_k$, $F=F_k$, $\sQ_1=\sQ_{k, 1}$, $\sQ_2=\sQ_{k, 2}$, $\Theta=\Theta_k$, $\tau'=\tau'_k$, $\delta'=\delta'_k$, $\sA=\sA_k$, $\pi=\pi_{m_k}$, $\varphi=\varphi$, $\varphi_0=\varphi_0$, and $\sigma=\sigma_k$,
by our choice of $\Phi_4$ we have
\begin{align*}
\sum_{(B, g)\in \Lambda_k}\m(\tilde{\varphi}'(gB)\cap \varphi(W_k))\le \frac{\delta_k^*}{40},
\end{align*}
\begin{align*}
\sum_{\substack{(B, g), (B', g')\in \Lambda_k,\\ (B, g)\neq (B', g')}}\m(\tilde{\varphi}'(gB)\cap \tilde{\varphi}'(g'B'))\le \frac{\delta_k^*}{40},
\end{align*}
and
\begin{align*}
\sum_{B\in \sQ_{k, 1}'}\sum_{t\in L_k^*L_k^\bullet}\m(\tilde{\varphi}'(tB)\Delta \sigma_{k, t}\varphi(B))\le \frac{\delta_k^*}{20}.
\end{align*}
Applying Lemma~\ref{L-Shannon hom} with $\sC=\sC_k^*$, $L=L_k^*$, $\delta=\delta_k^*$, $S=S$, $W=W_k$, $L^\bullet=L_k^\bullet$, $\overline{L}=\overline{L}_k$, $\sQ_1=\sQ_{k, 1}$, $\sQ_2=\sQ_{k, 2}$, $\Theta=\Theta_k$, $\sA=\sA_k$, $\varphi=\varphi$, $\delta'=\delta_k'$, $\bar{\tau}=\bar{\tau}_k$, and $\sigma=\sigma_k$,
we find a homomorphism $\tilde{\varphi}: \alg(\sR_{k,2})\rightarrow \Pb_{V_{m_k}}$ such that
$\tilde{\varphi}(B)=\varphi(B)$ for every $B\in \sQ_{k, 2}$ and $\sum_{A\in \sR_{k, 2}}\m(\tilde{\varphi}(A)\Delta \tilde{\varphi}'(A))\le \delta_k^*/5$.
Furthermore, the restriction of $\tilde{\varphi}$ to $\alg((\sC_k^*)_{L_k^*})$ lies in $\Hom_\mu(\sC_k^*, L_k^*, \delta_k^*, \sigma_k)$.

For any distinct $\varphi, \psi$ in $\Phi_4$, applying Lemma~\ref{L-Shannon separation} with $\sC=\overline{\sC}$, $S=S$, $W=W_k$, $L=L^\flat$, $\sQ=\sQ_k$, $\pi=\pi_{m_k}$, $\sigma=\sigma_k$, $\delta=\delta_k^*$, and $\delta'=\delta'_k$ we have
\begin{align*}
\Psi(\varepsilon/|\overline{\sC}|)&\le \rho_{\overline{\sC}}(\varphi, \psi)\\
&\le 2(\delta_k^*+\delta'_k)|S|+2\delta_k^*|S|\cdot |L^\flat|+2|S|^3\mu(X\setminus X_{S^2, L^\flat})\\
&\hspace*{30mm} \ + |S|\cdot |L^\flat|\rho_{\overline{\sC}\vee {}_{S^2, L^\flat}\sP}(\tilde{\varphi}, \tilde{\psi})\\
&\le \frac{4}{5}\Psi(\varepsilon/|\overline{\sC}|)+|S|\cdot |L^\flat|\rho_{\overline{\sC}\vee {}_{S^2, L^\flat}\sP}(\tilde{\varphi}, \tilde{\psi}),
\end{align*}
and hence
\[
\rho_{\overline{\sC}\vee {}_{S^2, L^\flat}\sP}(\tilde{\varphi}, \tilde{\psi})\ge \varepsilon':=
\frac{\Psi(\varepsilon/|\overline{\sC}|)}{5|S|\cdot |L^\flat|}.
\]
Thus
\begin{align*}
\lefteqn{\frac{1}{|V_{m_k}|}\log |\Hom_\mu(\overline{\sC}\vee {}_{S^2, L^\flat}\sP\vee \sU_k, L_k, \delta_k, \sigma_k)|_{\overline{\sC}\vee {}_{S^2, L^\flat}\sP}}\hspace*{30mm}\\
&\ge
\frac{1}{|V_{m_k}|}\log |\Hom_\mu(\sC_k^*, L_k^*, \delta_k^*, \sigma_k)|_{\overline{\sC}\vee {}_{S^2, L^\flat}\sP}\\
&\ge \frac{1}{|V_{m_k}|}\log |\Phi_4|\\
&\ge \frac{1}{|V_{m_k}|}\log |\Phi|-5\varepsilon\\
&\ge h_{\Pi, \mu}(G\curvearrowright X, \overline{\sC})-6\varepsilon.
\end{align*}

Since $L_k\subseteq L_k^\sharp$ and $\bar{\tau}_k\le \tau_k$ for every $k\in \Nb$, the sequence
$\Sigma=\{\sigma_k\}_{k\in \Nb}$ is a sofic approximation sequence for $H$.
Set $\sC^\flat=\overline{\sC}\vee {}_{S^2, L^\flat}\sP$.
For any finite partition $\sU$ of $X$ contained in the algebra generated by $\bigcup_{k\in \Nb}\sU_k$, any $L\in \cF(H)$ containing $e_H$, and any $\delta>0$, we have $\sU\preceq \sU_k$, $L\subseteq L_k$, and $\delta>\delta_k$ for all large enough $k$, and hence
\begin{align*}
\lefteqn{h_{\Sigma, \mu}(\sC^\flat, \sC^\flat\vee \sU, L, \delta)}\hspace*{15mm}\\
&= \varlimsup_{k\to \infty}\frac{1}{|V_{m_k}|}\log |\Hom_\mu(\sC^\flat\vee \sU, L, \delta, \sigma_k)|_{\sC^\flat} \\
&\ge \varlimsup_{k\to \infty}\frac{1}{|V_{m_k}|}\log |\Hom_\mu(\overline{\sC}\vee {}_{S^2, L^\flat}\sP\vee \sU_k, L_k, \delta_k, \sigma_k)|_{\overline{\sC}\vee {}_{S^2, L^\flat}\sP}\\
&\ge h_{\Pi, \mu}(G\curvearrowright X, \overline{\sC})-6\varepsilon.
\end{align*}
Since the algebra generated by $\bigcup_{k\in \Nb}\sU_k$ is dense in the Borel $\sigma$-algebra of $X$ with respect to the
pseudometric $d(A,B) = \mu (A\Delta B)$, by \cite[Lemma 10.13]{KerLi16} we conclude that
\begin{gather*}
h_{\Sigma, \mu}(H\curvearrowright X, \sC^\flat)\ge h_{\Pi, \mu}(G\curvearrowright X, \overline{\sC})-6\varepsilon ,
\end{gather*}
as desired.
\end{proof}

\begin{remark}
Theorem~\ref{T-Shannon SC to entropy}, and hence also Theorem~\ref{T-measure}, actually
uses only that $\kappa$ is Shannon, not that $\lambda$ is Shannon.
\end{remark}

\end{document}